\documentclass{amsart}
\usepackage{color}
\usepackage{graphicx}
\usepackage{hyperref}
\usepackage{amssymb}
\usepackage{amsfonts}
\usepackage{euscript}
\usepackage[all]{xy}
\usepackage{bbm}
\usepackage{tikz}

\numberwithin{equation}{section}
\numberwithin{figure}{section}
\renewcommand{\subsection}[1]{\hspace{-\parindent}\refstepcounter{subsection}{\bf (\arabic{section}\alph{subsection}) #1.}\addcontentsline{toc}{subsection}{\bf #1.}}

\newenvironment{nouppercase}{%
  \renewcommand{\uppercasenonmath}[1]{}}{}
\newcommand{\mybox}[1]{\parbox[t]{37.5em}{\parindent0em\parskip.5em #1}}
\newcommand{\smallprint}[1]{{\scriptstyle #1}}

\theoremstyle{plain}
\newtheorem{thm}{Theorem}[section]

\newtheorem{prop}[thm]{Proposition}

\newtheorem{remark}[thm]{Remark}

\newtheorem{proposition}[thm]{Proposition}
\newtheorem{example}[thm]{Example}

\newtheorem{lemma}[thm]{Lemma}

\newtheorem{construction}[thm]{Construction}
\newtheorem*{claim*}{Claim} 
\newtheorem*{lemma*}{Lemma}
\newtheorem*{theorem*}{Theorem}
\newtheorem*{conjecture*}{Conjecture}

\newcommand{\bC}{{\mathbb C}}

\newcommand{\bK}{{\mathbb K}}

\newcommand{\bQ}{{\mathbb Q}}
\newcommand{\bR}{{\mathbb R}}

\newcommand{\bZ}{{\mathbb Z}}

\newcommand{\scrA}{\EuScript A}
\newcommand{\scrB}{\EuScript B}
\newcommand{\scrC}{\EuScript C}
\newcommand{\scrD}{\EuScript D}

\newcommand{\scrF}{\EuScript F}

\newcommand{\scrH}{\EuScript H}
\newcommand{\scrI}{\EuScript I}

\newcommand{\scrL}{\EuScript L}
\newcommand{\scrM}{\EuScript M}
\newcommand{\scrN}{\EuScript N}
\newcommand{\scrO}{\EuScript O}
\newcommand{\scrP}{\EuScript P}
\newcommand{\scrQ}{\EuScript Q}
\newcommand{\scrR}{\EuScript R}
\newcommand{\scrS}{\EuScript S}

\newcommand{\scrW}{\EuScript W}

\newcommand{\scrZ}{\EuScript Z}

\newcommand{\frakg}{\mathfrak{g}}

\newcommand{\bfC}{\boldsymbol{C}}

\newcommand{\bfL}{\boldsymbol{L}}
\newcommand{\bfO}{\boldsymbol{O}}
\newcommand{\bfS}{\boldsymbol{S}}

\newcommand{\bfd}{\boldsymbol{d}}
\newcommand{\bfg}{\boldsymbol{g}}

\newcommand{\bfm}{\boldsymbol{m}}
\newcommand{\bfp}{\boldsymbol{p}}
\newcommand{\bfr}{\boldsymbol{r}}
\newcommand{\bfu}{\boldsymbol{u}}
\newcommand{\bfw}{\boldsymbol{w}}
\newcommand{\bfx}{\boldsymbol{x}}
\newcommand{\bfz}{\boldsymbol{z}}

\newcommand{\bfphi}{\boldsymbol{\phi}}
\newcommand{\bfrho}{\boldsymbol{\rho}}
\newcommand{\bfgamma}{\boldsymbol{\gamma}}

\newcommand{\bfzeta}{\boldsymbol{\zeta}}
\newcommand{\bfsigma}{\boldsymbol{\sigma}}

\newcommand{\half}{{\textstyle\frac{1}{2}}}
\newcommand{\quarter}{\textstyle\frac{1}{4}}

\newcommand{\iso}{\cong}
\newcommand{\htp}{\simeq}
\newcommand{\smooth}{C^\infty}

\newcommand{\tw}{\mathit{tw}}

\newcommand{\Aut}{\mathit{Aut}}
\newcommand{\cornerbar}[1]{
  \tikz[baseline=(n.base)]{\node(n)[inner sep=1pt]{$#1$};
    \draw[line cap=round](n.south west)--(n.north west)--(n.north east);
  }
}
\newcommand{\ch}{\mathit{ch}}
\newcommand{\va}{\mathit{va}}

\title[LEFSCHETZ FIBRATIONS]{\Large\larger\rm Fukaya $A_\infty$-structures associated to\\ Lefschetz fibrations. VI}
\author{Paul Seidel}
\subjclass[2020]{Primary 53D37; Secondary 18G70, 14D05.}

\begin{document}
\begin{nouppercase}
\maketitle
\end{nouppercase}

\begin{abstract}
To a symplectic Lefschetz pencil on a monotone symplectic manifold, we associate an algebraic structure, which is a pencil of categories in the sense of noncommutative geometry. 
\end{abstract}
%

\section{Introduction}
This paper constructs a Fukaya-categorical structure for a certain class of symplectic Lefschetz pencils. It is intended as a unifying concept, encapsulating the information which Lagrangian Floer theory can extract from those Lefschetz pencils. Pieces of this structure have been used before in both theory and applications, but the construction in its entirety is new.

\subsection{Floer theory relative to an anticanonical divisor}
We start the story at the point which is technically easiest to access (but not the most fundamental one, from our perspective). Let $M$ be a closed monotone symplectic manifold, meaning that
\begin{equation} \label{eq:monotone}
[\omega_M] = c_1(M).
\end{equation}
Suppose we have a smooth symplectic hypersurface $D \subset M$ which is an anticanonical divisor, meaning Poincar{\'e} dual to $c_1(M)$. The complement $M \setminus D$ is an exact symplectic manifold. Let $L \subset M \setminus D$ be a Lagrangian submanifold, assumed to have $H^1(L) = 0$ for simplicity, and to be {\em Spin}. Its Floer cohomology inside $M \setminus D$ is a graded abelian group, and comes with an isomorphism
\begin{equation} \label{eq:hf}
\mathit{HF}^*(L,L) \iso H^*(L).
\end{equation}
It carries additional information in the form of an invariant counting holomorphic discs that intersect $D$ exactly once (this goes back to the Addendum of \cite{oh93}, and was fully developed in \cite{cho-oh02, auroux07}), which in our context is just a number
\begin{equation} \label{eq:wl}
W(L) \in \mathit{HF}^0(L,L) = \bZ.
\end{equation}
On a more abstract level, this reflects the fact that the Fukaya category $\scrF(M \setminus D)$ has a natural formal deformation with deformation parameter of degree $2$; this deformation is the relative Fukaya category of the pair $(M,D)$, see \cite{sheridan-fano}. Inverting the deformation parameter yields a full subcategory of the Fukaya category $\scrF(M)$ of the closed symplectic manifold, and this has been traditionally a means of understanding $\scrF(M)$.

With more effort, one can extend the picture to noncompact Lagrangian submanifolds, with the following outcome. The wrapped Fukaya category $\scrW(M \setminus D)$ is a graded $A_\infty$-category over $\bZ$, and one can equip it with a deformation involving $D$, which extends that of $\scrF(M \setminus D)$.


\subsection{Floer theory for Lefschetz fibrations\label{subsec:old}}
Let's start in a different geometric context, that of Lefschetz fibrations $\Pi: E \rightarrow \bC$ over the complex plane, such that $[\omega_E]$ and $c_1(E)$ are both zero. Here, the natural class of Lagrangian submanifolds are those $L \subset E$ whose image in $\bC$ is allowed to go to infinity in positive real direction (that includes Lefschetz thimbles, as well as closed submanifolds; we continue to assume $H^1(L) = 0$ and $L$ is {\em Spin}). One can define Floer cohomology for such submanifolds so that \eqref{eq:hf} continues to hold, by using a suitable Hamiltonian term. There is also a Fukaya category $\scrF(\Pi)$, of which these are the cohomology level morphisms.

\begin{remark}
The Fukaya category of a Lefschetz fibration has had a troubled infancy (to which the author has contributed through the never-finished \cite{abouzaid-seidel11b}). The author first learned about it from Kontsevich. One can simplify the technicalities by allowing only a basis of Lefschetz thimbles as objects, and this was done in \cite{seidel02} as well as some earlier papers in this series \cite{seidel12b,seidel14b}. The book \cite{seidel04} gave a complete definition of $\scrF(\Pi)$, but it used a complicated workaround involving a double branched cover. Instead, the most frequently used approach is a categorical quotient construction \cite{abouzaid-seidel11b, ganatra-pardon-shende17}, and we will adopt that here. There is also an alternative \cite{lefschetz4andahalf}, which is algebraically much more straightforward but uses a larger automorphism group of the base than just rotations.
\end{remark}

Any Lefschetz fibration comes with a canonical Hamiltonian automorphism $\phi$, which is obtained by applying a full rotation to the part at $\infty$ of the base, and then lifting that. By using continuation maps, one can define distinguished elements
\begin{equation} \label{eq:c-elements}
C(L) \in \mathit{HF}^0(\phi(L),L),
\end{equation}
of a somewhat different nature than \eqref{eq:wl}: if $L$ is closed, $\mathit{HF}^*(\phi(L),L) \iso \mathit{HF}^*(L,L)$, and \eqref{eq:c-elements} is just the identity endomorphism. On the $A_\infty$-level, $\phi$ induces an autoequivalence of $\scrF(\Pi)$, which can be intrinsically described as its Serre functor; and the elements \eqref{eq:c-elements} lift to a natural transformation from $\phi$ (the Serre functor) to the identity functor. The last-mentioned observation goes back to \cite{seidel06}; it is worked out in \cite{seidel12b} in the context of a basis of vanishing cycles, or more elegantly in \cite{lefschetz5} using the approach from \cite{lefschetz4andahalf}; and we will give yet another definition here. This canonical natural transformation plays an important role in understanding the geometry of the Lefschetz fibration: one can easily construct examples with isomorphic Fukaya categories but different natural transformations.

\begin{example} \label{th:simple-example}
Take an exact Lefschetz fibration which has only two critical points. For simplicity, let's work with Floer cohomology with coefficients in a field $\bK$. In this case, the Fukaya category of the Lefschetz fibration is determined by a single graded $\bK$-vector space $H = \mathit{HF}(L_0,L_1)$, which is the Floer cohomology between the two Lefschetz thimbles in the total space $M^{2n}$. However, one can also consider $H$ as Floer cohomology between the vanishing cycles in a fibre $F$ of the Lefschetz fibration, $H = \mathit{HF}(L_{F,0},L_{F,1})$. In that context, one has additional multiplicative structure, namely
\begin{equation} \label{eq:simple-multiplication}
H^i \otimes (H^{n-1+i})^\vee = \mathit{HF}^i(L_{F,0},L_{F,1}) \otimes \mathit{HF}^{-i}(L_{F,1},L_{F,0}) \longrightarrow \mathit{HF}^0(L_{F,1},L_{F,1}) = \bK,
\end{equation}
which is nonzero iff the vanishing cycles are mutually isomorphic in the Fukaya category of the fibre. In terms of the Fukaya category of the Lefschetz fibration, that multiplicative structure is part of what's encoded in the natural transformation. 
For instance, consider two Lefschetz fibrations whose fibres are genus zero surfaces, and with the vanishing cycles as in Figure \ref{fig:simple-example}. The Fukaya categories of the Lefschetz fibrations are the same ($H$ is of total dimension two, with one generator each in two adjacent degrees), but they are distinguished by \eqref{eq:simple-multiplication}.
\end{example}
\begin{figure}
\begin{centering}
\includegraphics[scale=0.75]{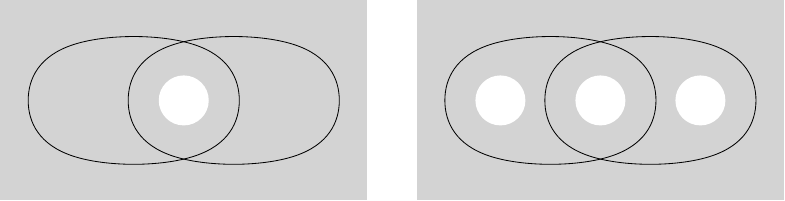}
\caption{\label{fig:simple-example}The vanishing cycles from Example \ref{th:simple-example}.}
\end{centering}
\end{figure}%

In \cite{seidel12b}, this natural transformation was exhibited as a piece of a richer algebraic relationship between the Fukaya category of the Lefschetz fibration and that of the fibre. Various names have been proposed for this kind of relationship \cite{seidel12b, kontsevich-vlassopoulos13, seidel14b}; here, we follow the last-mentioned reference in calling it a ``noncommutative divisor'' (more specifically, because the Serre functor is involved, one might want to call it a noncommutative anticanonical divisor, but that will play no role in the current paper). The terminology is inspired by mirror symmetry, where this structure corresponds to singling out an anticanonical divisor on the mirror to the Lefschetz fibration.

\subsection{Floer theory for Lefschetz pencils}
Now, let's consider the situation of an anticanonical Lefschetz pencil on a monotone symplectic manifold $M$. This fits into both our previous contexts: any smooth member $D$ of the pencil is an anticanonical divisor; and the complement carries an exact Lefschetz fibration $\Pi: M \setminus D \rightarrow \bC$. The basic insight in \cite{seidel14b} was that in this situation, there is a second natural transformation from the Serre functor of $\scrF(\Pi)$ to the identity, which encodes the geometry of holomorphic curves passing through $D$, and in particular (when combined with the first natural transformation) contains the numbers \eqref{eq:wl}; see Section \ref{subsec:formal-c}. Just like the first transformation was part of a structure of noncommutative divisor, this pair of transformations was conjectured to be part of a richer structure of noncommutative pencil: that algebraic notion was defined in \cite[Section 1b]{seidel14b}, and its existence in this geometric context postulated in \cite[Conjecture 1.6]{seidel14b}. Moreover, \cite[Section 1c]{seidel14b} wove it into a broader conjectural picture which includes not just the Fukaya category of the fibre, but also the wrapped category $\scrW(M \setminus D)$ and its deformation as described before. 

Among subsequent developments, the relation with the wrapped category now falls under the much more general ``stop-removal'' formalism from \cite{ganatra-pardon-shende18}. In a slightly different situation (anticanonical Lefschetz fibrations rather than pencils), the dependence of the natural transformations on the K\"ahler parameters is of interest, because it leads to a computation of mirror maps \cite{lefschetz2andahalf, lefschetz8}. The contribution of this paper is different, namely to give a direct definition of the structure we've been talking about all this time:

\begin{construction} \label{th:1}
Take a monotone symplectic manifold with an anticanonical symplectic Lefschetz pencil, and its associated Lefschetz fibration. Then, the Fukaya category of that fibration carries the structure of a noncommutative pencil.
\end{construction}

One can question whether such a geometric construction is really necessary; one could bypass it by only constructing the natural transformations geometrically, and then using obstruction theory \cite{seidel14b} to recover the rest. This obstruction-theoretic approach is used in \cite{lefschetz8}; but it handicaps possible future developments, where one wants to connect properties of  the noncommutative pencil to symplectic geometry. It also does not seem to work any longer for generalizations, such as the one we are about to state (following an indication in \cite[Remark 1.7]{seidel14b}):

\begin{construction} \label{th:2}
On a monotone symplectic manifold, consider a symplectic Lefschetz pencil, such that the first Chern class of the manifold is $N$ times the class of the hypersurfaces in the pencil, for some integer $N \geq 1$. The Fukaya category of the associated Lefschetz fibration carries a graded noncommutative pencil, with degrees $(2N-2,0)$.
\end{construction}

Just like their geometric counterparts, noncommutative pencils have ``fibres'' parametrized by the projective line, which are $A_\infty$-categories. In our applications the ``fibre at $\infty$'' of the noncommutative pencil is $\scrF(D \setminus B)$, the Fukaya category of the symplectic divisor with the base locus of the pencil removed. This reflects a relation that applies to more general Lefschetz fibrations and was already discussed in Section \ref{subsec:old} (our version of it is proved in Proposition \ref{th:restriction-functor}). On the other hand, the fibre at $0$  seems to hold particular interest in Construction \ref{th:2}: it is $\bZ$-graded, and from a first look, could plausibly be a Calabi-Yau category of dimension $n + 1-2N$ (here, the hypersurfaces of the original symplectic pencil had real dimension $2n-2$). In principle, there should be another approach to defining the ``fibres'' at points other than $\infty$, in terms of $\scrW(M \setminus D)$, but that does not necessarily shed more light on their nature (and the agreement between the two approaches may no longer hold in generalizations where $D$ is allowed to be singular).

\begin{remark}
In spite of the way in which the results above were formulated, we do not actually need a whole geometric Lefschetz pencil for the definitions: they only depend on what one can see in a neighbourhood of a single divisor (basically, the divisor itself and its perturbed versions, all intersecting in the base locus). However, if one wanted to study the resulting categories further, having the entire Lefschetz pencil at one's disposal would certainly be useful (see Section \ref{subsec:global-geometry}, and in particular Remark \ref{th:lefschetz-pencil}).
\end{remark}

\begin{remark}
We use monotonicity as a technical tool, which (unsuprisingly) allows us to deal with bubbling off of holomorphic spheres in a fairly simplistic way. It is not inconceivable that this condition could be dropped, and replaced with the assumption that the hypersurfaces should be Poincar\'e dual to the symplectic class (which is the most commonly used notion of Lefschetz pencil on a general closed symplectic manifold). On the technical side, one would have to find more sophisticated arguments, possibly borrowed from relative Gromov-Witten theory. For the resulting algebraic structures, one can only expect to have a $\bZ/2$-grading in general.
\end{remark}


Concerning the actual techniques, all our Floer-theoretic constructions are based on certain moduli spaces of decorated Riemann surfaces, namely the popsicle spaces from \cite{abouzaid-seidel07} (even though the class of surfaces, and the Cauchy-Riemann equations they carry, is broadly the same, our use of it is in a sense complementary to that in \cite{abouzaid-seidel07}; see Remark \ref{th:abouzaid} for a comparison). This additional structure on the surfaces will be used in two different ways: one is to equip them with one-forms that enter into inhomogeneous versions of the pseudo-holomorphic map equation; the other is to impose intersection conditions with a symplectic divisor.

\subsection{Content of the paper}
The first three parts (Sections \ref{sec:algebra}--\ref{sec:popsicle}) are review and background material concerning, respectively, homological algebra, pseudo-holomorphic curves, and popsicle spaces. After that, we define the ``fibre at $\infty$'' of the noncommutative pencil (Section \ref{sec:first-construction}), and relate it to the Fukaya category of the fibre (Section \ref{sec:restriction}). Then we give the construction of the ``fibre at $0$'', which is similar on a formal-algebraic level but encodes different geometric information (Section \ref{sec:anticanonical}). After that, we unify the two ideas to build the entire noncommutative pencil (Section \ref{sec:flavour}). One consequence of this gradual buildup is that techniques are explained in most detail when they first occur, and then adapted to more complicated but analogous situations in a more summary fashion. We conclude with an example (Section \ref{sec:example}). 

{\em Acknowledgments.} Two aspects of this paper, namely, the definition of the Fukaya category of a Lefschetz fibration using localisation, and the construction of the functor that maps it to the Fukaya category of the fibre, are scavenged from \cite{abouzaid-seidel11b}. This work was partially supported by the Simons Foundation, through a Simons Investigator award and the Simons Collaboration for Homological Mirror Symmetry; by NSF grant DMS-1500954; and by Princeton University and the Institute for Advanced Study, through visiting appointments. 

\section{Basic notions: algebra\label{sec:algebra}}
This section recalls the algebraic language in which the results of the paper will be phrased. Among other things, we fix the basic conventions concerning signs and units for $A_\infty$-structures. We also discuss categorical localisation, which will be used as a technical workaround in our constructions.

\subsection{$A_\infty$-categories\label{subsec:categories}}
Fix a (commutative) coefficient ring $R$, assumed to be of finite global dimension. A (strictly unital $\bZ$-graded) $A_\infty$-category $\scrA$ over $R$ consists of the following data. First, a set of objects; secondly, for any pair of objects $(X_0,X_1)$, a free graded $R$-module $\scrA(X_0,X_1)$ (we denote the degree of a homogeneous morphism $x \in \scrA(X_0,X_1)$ by $|x|$, and the reduced degree by $\|x\| = |x| - 1$); and finally, composition maps 
\begin{equation}
\mu_{\scrA}^d: \scrA(X_{d-1},X_1) \otimes \cdots \otimes \scrA(X_0,X_1) \longrightarrow
\scrA(X_0,X_d)[2-d]
\end{equation}
for $d \geq 1$, satisfying the $A_\infty$-associativity relations
\begin{equation} \label{eq:associativity}
\begin{aligned}
& \sum_{i,j} (-1)^\ast \mu_{\scrA}^{d-j+1}(x_d,\dots,x_{i+j},\mu_{\scrA}^j(x_{i+j-1},\dots,x_i), x_{i-1},\dots,x_1) = 0; 
\\[-.5em]
& \smallprint{\qquad \qquad \ast = \|x_1\| + \cdots + \|x_{i-1}\|.}
\end{aligned}
\end{equation}
The strict unitality property means that we have nonzero elements $e = e_X \in \scrA^0(X,X)$ such that
\begin{equation} \label{eq:e}
\left\{
\begin{aligned}
& \scrA^0(X,X)/R e \text{ is a free $R$-module,} \\
& \mu^1_{\scrA}(e) = 0, \\
& \mu^2_{\scrA}(x,e) = (-1)^{|x|} \mu^2_{\scrA}(e,x) = x, \\
& \mu^d_{\scrA}(\dots, e, \dots) = 0 \;\; \text{ for $d>2$.}
\end{aligned}
\right.
\end{equation}
The associated cohomology level category $A = H^*(\scrA)$ has the graded spaces $H^*(\scrA(X_0,X_1))$ as morphisms, with composition induced by $(x_2,x_1) \mapsto (-1)^{|x_1|} \mu^2_{\scrA}(x_2,x_1)$. When considering $A_\infty$-functors, we also impose a strictly unitality condition. 

\begin{remark}
The assumptions on $R$ and $\scrA(X_0,X_1)$ ensure that certain standard arguments based on quasi-isomorphism work. Similarly, the freeness assumption in \eqref{eq:e} ensures that the reduced Hochschild complex is well-behaved. It is possible to work with $A_\infty$-categories without imposing those assumptions  (for instance, \cite{bespalov-lyubashenko-manzyuk} uses  $A_\infty$-categories whose morphism spaces are arbitrary $R$-modules, for any commutative ring $R$), but then, quasi-isomorphisms have to be replaced with more precise chain homotopy arguments. 
\end{remark}

$A_\infty$-categories are ``homotopy invariant algebraic structures''. One concrete aspect of that are transfer results \cite{markl04}, of which the following is a sample:
\begin{equation} \label{eq:transfer}
\mybox{
Let $\scrB$ be an $A_\infty$-category. Take graded subspaces $\scrA(X_0,X_1) \subset \scrB(X_0,X_1)$ for all pairs of objects, such that both $\scrA(X_0,X_1)$ and $\scrB(X_0,X_1)/\scrA(X_0,X_1)$ are free graded $R$-modules. For $X_0 = X_1 = X$, we also require that $\scrA^0(X,X)$ should contain the identity, and that $\scrA^0(X,X)/R e$ is free. Suppose that $\scrA(X_0,X_1)$ is a subcomplex, and the inclusion into $\scrB(X_0,X_1)$ is a quasi-isomorphism. Then, there is an $A_\infty$-category $\scrA$ with the same objects as $\scrB$, and with morphism spaces $\scrA(X_0,X_1)$, together with an $A_\infty$-quasi-equivalence $\scrA \rightarrow \scrB$ whose linear term is the inclusion.
}
\end{equation}
A sketch of the proof, restricted to $A_\infty$-algebras for notational simplicity, goes as follows. We know that $\scrB/\scrA$ is acyclic, and therefore contractible. One takes a contracting homotopy for $\scrB/\scrA$, and extends that to a self-map $h$ of $\scrB$ which vanishes on $\scrA$. Then, the chain map $\pi = \mathit{id} - \mu^1_{\scrB} h - h \mu^1_{\scrB}$ is a projection from $\scrB$ to $\scrA$.
%
The $A_\infty$-operations on $\scrA$ can be constructed through explicit formulae: of course $\mu^1_{\scrA}(x) = \mu^1_{\scrB}(x)$, and then further
\begin{equation} \label{eq:explicit-transfer}
\left\{
\begin{aligned}
& \mu^2_{\scrA}(x_2,x_1) = \pi \mu^2_{\scrB}(x_2,x_1), \\
& \mu^3_{\scrA}(x_3,x_2,x_1) = \pi \big( \mu^3_{\scrB}(x_3,x_2,x_1) 
+ \mu^2_{\scrB}(x_3,h \mu^2_{\scrB}(x_2,x_1)) + \mu^2_{\scrB}(h \mu^2_{\scrB}(x_3,x_2),x_1) \big),
\\
& \dots
\end{aligned}
\right.
\end{equation}
Inspection of those formulae shows that $\scrA$ is again strictly unital.

\subsection{$A_\infty$-bimodules\label{subsec:bimodules}}
An $A_\infty$-bimodule $\scrQ$ over $\scrA$ consists of free graded $R$-modules $\scrQ(X_0,X_1)$, and structure maps
\begin{multline}
\qquad \mu_{\scrQ}^{s;1;r}: \scrA(X_{r+s}, X_{r+s+1}) \otimes \cdots \otimes \scrA(X_{r+1},X_{r+2}) \otimes \scrQ(X_r,X_{r+1}) \\  
\otimes \scrA(X_{r-1},X_r) \otimes \cdots \otimes \scrA(X_0,X_1)
\longrightarrow \scrQ(X_0,X_{r+s+1})[1-r-s] \qquad
\end{multline}
for $r,s \geq 0$, satisfying the bimodule equations
\begin{equation}
\begin{aligned}
& \sum_{i,j} (-1)^\ast \mu_{\scrQ}^{s;1;r-j+1}(x_{r+s},\dots,x_{r+1}; y; x_r,\dots,x_{i+j},\mu_{\scrA}^j(x_{i+j}-1,\dots,x_{i}),x_{i-1},\dots,x_1) \\ 
+ & \sum_{i,j} (-1)^\ast \mu_{\scrQ}^{s-j;1;i}(x_{r+s},\dots,x_{r+j}; \mu_{\scrQ}^{j-1;1;r-i+1}(x_{r+j-1},\dots,x_{r+1};y;x_r,\dots,x_{i}); x_{i-1},\dots,x_1) \\
+ & \sum_{i,j} (-1)^\dag \mu_{\scrQ}^{s-j+1;1;r}(x_{r+s},\dots,\mu_{\scrA}^j(x_{r+i+j-1},\dots,x_{r+i}),\dots,x_{r+1};y;x_r,\dots,x_1)
= 0; \\[-.5em]
& \smallprint{\qquad \qquad \ast = \|x_{1}\|+\cdots+\|x_{i-1}\|, \;\; \dag = \|x_1\| + \cdots + \|x_{r+i-1}\| + |y|;}
\end{aligned}
\end{equation}
and the unitality conditions
\begin{equation}
\left\{
\begin{aligned}
& \mu_{\scrQ}^{0;1;1}(y;e) = (-1)^{\|y\|} \mu_{\scrQ}^{1;1;0}(e;y) = y, \\
& \mu_{\scrQ}^{s;1;r}(\dots,e,\dots;y;\dots) = 0, \; \mu_{\scrQ}^{s;1;r}(\dots;y;\dots,e,\dots) = 0 \;\; \text{ for $r+s>1$.}
\end{aligned}
\right.
\end{equation}
There is a shift operation on bimodules: $\scrQ[1]$ has the same $R$-modules as $\scrQ$ with gradings shifted down, and the structure maps 
\begin{equation}
\begin{aligned}
& \mu^{s;1;r}_{\scrQ[1]}(x_{r+s},\dots,x_{r+1};y;x_r,\dots,x_1) = (-1)^\ast \mu^{s;1;r}_{\scrQ}(x_{r+s},\dots,x_{r+1};y;x_r,\dots,x_1); \\
& \smallprint{\qquad \qquad \ast = \|x_1\| + \cdots + \|x_r\| + 1.}
\end{aligned}
\end{equation}
For instance, the diagonal bimodule $\Delta_{\scrA}$ consists of the same $R$-modules as $\scrA$ itself. Its bimodule operations can be described by saying that the bimodule operations on the shifted version $\Delta_{\scrA}[1]$ are the same as the $A_\infty$-operations of the category: $\mu^{s;1;r}_{\Delta_\scrA[1]} = \mu^{s+1+r}_{\scrA}$. The linear dual of a bimodule $\scrQ$ consists of the graded $R$-modules $\scrQ^\vee(X_0,X_1) = \scrQ(X_1,X_0)^\vee = \mathit{Hom}_R(\scrQ(X_1,X_0),R)$, with $A_\infty$-operations that satisfy
\begin{equation} \label{eq:dual-bimodule}
\langle \mu_{\scrQ^\vee}^{s;1;r}(x_{r+s},\dots,x_{r+1};\xi;x_r,\dots,x_1), y \rangle = (-1)^{\|y\|} \langle \xi, \mu_{\scrQ}^{r;1;s}(x_r,\dots,x_1;y; x_{r+s},\dots,x_{r+1})\rangle,
\end{equation}
with respect to the canonical pairing $\langle \cdot, \cdot \rangle$. This construction should be used with caution, as the $\scrQ^\vee(X_0,X_1)$ may no longer be free $R$-modules in general; however, that will not be an issue in our only application (Section \ref{sec:example}).

A morphism $\phi: \scrQ \rightarrow \scrR$ of bimodules (by which we mean a degree $0$ cocycle in the dg category of $A_\infty$-bimodules) consists of maps
\begin{multline}
\qquad \phi^{s;1;r}: \scrA(X_{r+s}, X_{r+s+1}) \otimes \cdots \otimes \scrA(X_{r+1},X_{r+2}) \otimes \scrQ(X_r,X_{r+1}) \\
\otimes \scrA(X_{r-1},X_r) \otimes \cdots \otimes \scrA(X_0,X_1)
\longrightarrow \scrR(X_0,X_{r+s+1})[-r-s], \qquad
\end{multline}
such that
\begin{equation}
\phi^{s;1;r}(\dots,e,\dots;y;\dots) = 0, \;\; \phi^{s;1;r}(\dots;y;\dots,e,\dots) = 0,
\end{equation}
and
\begin{equation} \label{eq:bimodule-map-equation}
\begin{aligned}
& \sum_{i,j} (-1)^\ast \phi^{s;1;r-j+1}(x_{r+s},\dots,x_{r+1}; y; x_r,\dots,x_{i+j},\mu_{\scrA}^j(x_{i+j-1},\dots,x_{i}),\dots,x_1) \\ 
+ & \sum_{i,j} (-1)^\ast \phi^{s-j;1;i-1}(x_{r+s},\dots,x_{r+j+1};\mu_{\scrQ}^{j-1;1;r-i+1}(x_{r+j-1},\dots,x_{r+1};y;x_r,\dots,x_{i});
\dots,x_1) \\
+ & \sum_{i,j} (-1)^\dag \phi^{s-j+1;1;r}(x_{r+s},\dots,\mu_{\scrA}^j(x_{r+i+j-1},\dots,x_{r+i}),\dots,x_{r+1};y;x_r,\dots,x_1)
\\  = &
\sum_{i,j} \mu_{\scrR}^{s-j;1;i-1}(x_{r+s},\dots,x_{r+j};\phi^{j-1;1;r-i+1}(x_{r+j-1},\dots,x_{r+1};y;x_r,\dots,x_{i});
x_{i-1},\dots,x_1); \\[-.5em]
& \smallprint{\qquad \qquad \ast = \|x_1\|+\cdots+\|x_{i-1}\|, \;\; \dag = \|x_1\| + \cdots + \|x_{r+i-1}\| + |y|;}
\end{aligned}
\end{equation}

One can associate to any $\scrA$-bimodule $\scrQ$ an $A_\infty$-category, as follows:
\begin{equation}
\mybox{
The trivial extension category $\scrA \oplus \scrQ[1]$ has the same objects as $\scrA$, and morphism spaces 
\[
(\scrA \oplus \scrQ)(X_0,X_1) = \scrA(X_0,X_1) \oplus \scrQ(X_0,X_1)[1].
\]
The composition maps consist of the $A_\infty$-structure of $\scrA$ together with the bimodule structure of $\scrQ[1]$. 
}
\end{equation}
In particular, we have $\mu_{\scrA \oplus \scrQ[1]}^d(x_d,\dots,x_1) = 0$ whenever at least two of the $x_k$ belong to $\scrQ$ summands.

\subsection{Noncommutative linear systems\label{subsec:nc-linear-system}}
We will use an additional grading, called the weight grading. Given a weight-graded $R$-module $M$, write $M^{(p)}$ for the weight $p$ part. Fix a nonzero finite-dimensional free $R$-module $V$, given weight $-1$. Consider the weight-graded symmetric algebra $R[V]$. We will work with weight-graded modules over $R[V]$. To say that a weight-graded module is free means that it is of the form $R[V] \otimes_R M$, where $M$ is a weight-graded free $R$-module. We have assumed that $R$ has finite global dimension; hence so does $R[V]$ (as an ungraded ring). It follows that the category of weight-graded $R[V]$-modules also has finite global dimension. In particular, (unbounded) complexes of free weight-graded modules are homotopically projective.

We will consider $A_\infty$-categories over $R[V]$ which are weight-graded (in addition to the usual grading). This means that the morphism spaces in each homological degree are free weight-graded $R[V]$-modules, and that the composition maps are $R[V]$-multilinear and homogeneous with respect to weights. Following \cite[Section 2g]{seidel14b}, we define:
\begin{equation} \label{eq:nc-linear-system}
\mybox{
A noncommutative linear system, of dimension $\mathrm{rank}(V)\!-\!1$, is a weight-graded $A_\infty$-category $\scrL$ over $R[V]$ with the following property: each space $\scrL(X_0,X_1)$ is generated over $R[V]$ by elements of weights $0$ and $-1$.
}
\end{equation}
Note that this implies that the positive weight parts $\scrL(X_0,X_1)^{(p)}$, $p>0$, are zero. Because we have imposed restrictions on $\scrL(X_0,X_1)$ as an $R[V]$-module, the notion of noncommutative linear system is not ``homotopy invariant''. These restrictions would have to be relaxed if, for instance, one wanted to work with dg rather than $A_\infty$-structures. For us, this issue becomes relevant only at one point, where we will need to start with an object with weaker properties, and obtain a noncommutative linear system by transfer:
\begin{equation} \label{eq:weight01}
\mybox{
Suppose that $\scrM$ is a weight-graded $A_\infty$-category over $R[V]$, such that all morphism spaces satisfy $\scrM(X_0,X_1)^{(p)} = 0$ for $p>0$. Let $\scrL(X_0,X_1) \subset \scrM(X_0,X_1)$ be the $R[V]$-submodule generated by elements of weight $-1,0$. By assumption, this is a subcomplex. Suppose that the quotient $\scrM(X_0,X_1)/\scrL(X_0,X_1)$ is contractible, for each $(X_0,X_1)$.
Then, there is a noncommutative linear system $\scrL$ with the same objects, with morphism spaces $\scrL(X_0,X_1)$, and a quasi-equivalence of weight-graded $A_\infty$-categories $\scrL \rightarrow \scrM$, whose linear part is the inclusion. Moreover, the $A_\infty$-compositions $\mu_\scrL^d(x_d,\dots,x_1)$ agree with those on $\scrM$ as long as the sum of weights of $x_1,\dots,x_d$ is $\geq -1$ (which means that at most one has weight $-1$, and all others weight $0$).
}
\end{equation}
This is a direct application of the version of \eqref{eq:transfer} with added weight-gradings: the quotient $\scrM(X_0,X_1)/\scrL(X_0,X_1)$ is a complex of free weight-graded modules, and is acyclic, hence contractible. The homotopy $h$ that enters into the construction vanishes on the $\scrL$ subspaces, and hence in particular on parts of weight $\geq -1$; hence, all expressions $h \mu_{\scrM}^j$ in \eqref{eq:explicit-transfer} vanish if the sum of weights of the inputs is $\geq -1$. This implies the property stated in the last sentence of \eqref{eq:weight01}.

\begin{remark}
Since the quotient $\scrM(X_0,X_1)/\scrL(X_0,X_1)$ is a complex of free weight-graded $R[V]$-modules, it is contractible iff it is acyclic. Hence, in \eqref{eq:weight01} it is sufficient to know that $\scrL(X_0,X_1) \subset \scrM(X_0,X_1)$ is a quasi-isomorphism.
\end{remark}

Let's look at noncommutative linear systems in more concrete terms. 
\begin{equation} \label{eq:splitting}
\mybox{
One can write $\scrL(X_0,X_1) \iso (\scrA(X_0,X_1) \oplus \scrQ(X_0,X_1)[1]) \otimes R[V]$, where
\[
\begin{aligned} 
& \scrA(X_0,X_1) = \scrL(X_0,X_1)^{(0)}, \\
& \scrQ(X_0,X_1)[1] = \scrL(X_0,X_1)^{(-1)}/V \scrL(X_0,X_1)^{(0)}.
\end{aligned}
\]
The summands $\scrA(X_0,X_1)$ and $\scrQ(X_0,X_1)$ are free graded $R$-modules, which carry weight $0$ and $-1$, respectively. This splitting is not canonical: one can change it by any automorphism of the right hand side which is the identity plus an arbitrary $R$-module map $\alpha: \scrQ(X_0,X_1) \rightarrow \scrA(X_0,X_1) \otimes V$ of degree $-1$.
}
\end{equation}
In terms of \eqref{eq:splitting}, the compositions become maps
\begin{equation} \label{eq:nc-pencil-maps}
\bigotimes_{i=1}^d \big(\scrA(X_{i-1},X_i) \oplus \scrQ(X_{i-1},X_i)[1]\big) \stackrel{\mu_{\scrL}^d}{\longrightarrow}
\big( \scrA(X_0,X_d)[2-d] \oplus \scrQ(X_0,X_d)[3-d] \big) \otimes R[V].
\end{equation}
Homogeneity with respect to weights means that the number of $\scrQ$ factors in the source of \eqref{eq:nc-pencil-maps} matches that on the output plus the degree of polynomials in $R[V]$. The simplest piece of \eqref{eq:nc-pencil-maps} that can be treated in isolation is the weight $0$ part:
\begin{equation} \label{eq:ambient-space}
\mybox{
Let $\scrL$ be a noncommutative linear system. The subspaces $\scrA(X_0,X_1)$ from \eqref{eq:splitting} form an $A_\infty$-category, called the ``ambient space'' of the linear system. We also say that $\scrL$ is a noncommutative linear system ``on $\scrA$'', or that $\scrA$ ``carries'' the system.
}
\end{equation}
Next:
\begin{equation} \label{eq:dual-bundle}
\mybox{
The composition maps on the weight $(-1)$ part canonically induce maps
\[
\begin{aligned}
& \scrA(X_{r+s},X_{r+s+1}) \otimes \cdots \otimes \scrA(X_{r+1},X_{r+2}) \otimes \scrQ(X_r,X_{r+1}) \otimes \scrA(X_{r-1},X_r)
\\ & \quad \otimes \cdots \otimes \scrA(X_0,X_1) \longrightarrow \scrQ(X_0,X_d)[1-r-s],
\end{aligned}
\]
which give $\scrQ$ the structure of an $\scrA$-bimodule, called the ``dual bundle'' of the linear system.
}
\end{equation}
In \cite{seidel14b}, we asked for $\scrQ$ to be an invertible bimodule (with respect to tensor product). That condition has been omitted here, since the basic definitions do not depend on it. However, if one wanted to keep the meaning of the theory close to that of classical linear systems, the condition definitely needs to be re-imposed; compare Remark \ref{th:classical} below.
\begin{equation} \label{eq:dual-section}
\mybox{
Consider again the weight $(-1)$ part, but now extract the other component
\[
\begin{aligned}
& \sigma^{s;1;r}: \scrA(X_{r+s},X_{r+s+1}) \otimes \cdots \otimes \scrA(X_{r+1},X_{r+2}) \otimes \scrQ(X_r,X_{r+1}) \otimes \scrA(X_{r-1},X_r)
\\ & \quad \otimes \cdots \otimes \scrA(X_0,X_1) \longrightarrow \scrA(X_0,X_d) \otimes V[-r-s].
\end{aligned}
\]
This can be viewed as a linear family of $A_\infty$-bimodule maps $\sigma_w: \scrQ \rightarrow \Delta_{\scrA}$, parametrized by the dual space $W = \mathit{Hom}(V,R)$. These maps depend on the choice of splitting, but their cohomology classes in the dg category of bimodule maps are canonical. Indeed, changing the splittings by $\alpha$, as in \eqref{eq:splitting}, yields bimodule maps which differ from the original one by the boundary of $\alpha$. We call these bimodule maps the ``sections'' of the noncommutative linear system.
}
\end{equation}
\begin{equation} \label{eq:fibres}
\mybox{
Specializing a noncommutative linear system to a nonzero $w \in W$ (by tensoring with the corresponding simple $R[V]$-module, which destroys the weight grading) yields an $A_\infty$-category $\scrF_w$ with
\[
\scrF_w(X_0,X_1) = \scrA(X_0,X_1) \oplus \scrQ(X_0,X_1)[1],
\]
which contains $\scrA$ as an $A_\infty$-subcategory. We call these categories the ``fibres'' of the linear system. If we use the inclusion $\scrA \subset \scrF_w$ to consider the diagonal bimodule $\Delta_{\scrF_w}$ as an $\scrA$-bimodule, then as such, it sits in a short exact sequence
\[
0 \rightarrow \Delta_{\scrA} \longrightarrow \Delta_{\scrF_w} \longrightarrow \scrQ[1] \rightarrow 0,
\]
whose connecting homomorphism is the section associated to $w$.
}
\end{equation}
Of course, one can also specialize to $w = 0$, but that is not interesting, since one always gets the trivial extension $\scrA \oplus \scrQ[1]$. Multiplying $w$ by an element of $R^\times$ does not change the isomorphism type of $\scrF_w$; the two structures are related by rescaling the $\scrQ$ part of the morphism space. With that in mind, we will sometimes consider the $\scrF_w$ to be parametrized by the projective space $\mathbb{P}(W)$. 

We use the terminology ``noncommutative divisor'' and ``noncommutative pencil'' for the cases $V = R$ and $V = R^2$, with notation $\scrD$ and $\scrP$. Let's look briefly at the case of a noncommutative divisor. There, the only nontrivial fibre
\begin{equation}
\scrF = \scrF_1, 
\end{equation}
thought of as an $A_\infty$-category containing $\scrA$, has all the information on the noncommutative divisor. Indeed, this was the definition of noncommutative divisor used in \cite[Section 2f]{seidel14b} (and which, under a different name, already appears in \cite{seidel09}; for related notions, see \cite{kontsevich-vlassopoulos13, brav-dyckerhoff16}).

\begin{remark} \label{th:switch-q}
Suppose temporarily that $R$ is a field of characteristic zero. One can then view the classification theory of noncommutative linear systems, with given $\scrA$ and $\scrQ$, as a special instance of the general Maurer-Cartan deformation formalism \cite[Section 2f-2g]{seidel14b}. Using a general property of that formalism (invariance under $L_\infty$-quasi-isomorphisms) one sees that, given a quasi-isomorphism of bimodules $\scrQ \htp \tilde\scrQ$, there is an induced bijection between (equivalence classes, in a suitable sense, of) noncommutative linear systems with dual bundle $\scrQ$ or $\tilde{\scrQ}$. In fact, the same holds for general $R$, but requires one to inspect the argument more carefully, ensuring that one can get away with using only obstruction theory instead of the full Maurer-Cartan formalism.
\end{remark}

\begin{remark} \label{th:classical}
As the terminology indicates, there is a parallel with algebraic geometry, which we will make explicit now. Let $A$ be a scheme over $R$, and $Q$ a line bundle (a rank one locally free $\scrO_A$-module sheaf) on $A$. Taking $V$ and $W$ as before, suppose that we have a linear family of sections of the inverse $Q^{-1}$, parametrized by $W$, by which we mean an injective map
\begin{equation} \label{eq:section-1}
\sigma: W \longrightarrow \Gamma(A,Q^{-1}) = \mathit{Hom}_{\scrO_A}(Q,\scrO_A).
\end{equation}
In classical algebro-geometric language, this yields a linear system of hypersurfaces $F_w = \sigma_w^{-1}(0)$, parametrized by $w \in \mathbb{P}(W)$. The sheaf $\scrO_{F_w}$, pushed forward to $A$, admits the Koszul resolution
\begin{equation} 
\big\{ Q \stackrel{\sigma_w}{\longrightarrow} \scrO_A \big\} \iso \scrO_{F_w},
\end{equation}
whose left hand side is a sheaf of differential graded algebras (the exterior algebra over $Q$, negatively graded, and with differential given by $\sigma_w$). Let's look at the family of such resolutions,
\begin{equation} \label{eq:commutative-case}
L \stackrel{\mathrm{def}}{=} \big\{ Q \otimes R[V] \stackrel{\sigma}{\longrightarrow} \scrO_A \otimes R[V] \big\},
\end{equation}
where we think of $\sigma$ as an element of $\mathit{Hom}_{\scrO_A}(Q,\scrO_A) \otimes V$. The dga structure of \eqref{eq:commutative-case} is compatible with the weight grading, defined by giving $Q$ and $V$ weight $-1$, and $\scrO_A$ weight $0$ (this homogeneity property precisely expresses the linear dependence of $\sigma_w$ on $w$). In those terms, one recovers the original data as in \eqref{eq:splitting}:
\begin{equation}
\begin{aligned}
& \scrO_A = L^{(0)}, \\
& Q = L^{(-1)}/ V L^{(0)}.
\end{aligned}
\end{equation}
The analogy is incomplete in several respects; for one thing, in our definition of noncommutative linear system, there is no counterpart of the injectivity condition for \eqref{eq:section-1} (and it would not seem appropriate to impose one).
\end{remark}

The following slight generalization (which has no classical parallel in the vein of Remark \ref{th:classical}) occurs naturally in the symplectic context. Namely, we consider a finite-dimensional free $R$-module $V$ as before, but now assume that it comes with a (homological) grading, which is supposed to be even. Given the weight grading as before, we therefore have a bigraded symmetric algebra $R[V]$. A free bigraded $R[V]$-module is a bigraded module of the form $R[V] \otimes M$, where $M$ is a free bigraded $R$-module. Consider $A_\infty$-categories over $R[V]$ which are bigraded (by homological degree and by weight), and whose morphism spaces are free bigraded $R[V]$-modules. The corresponding version of \eqref{eq:nc-linear-system} is:
\begin{equation} \label{eq:graded-linear-system}
\mybox{
A graded noncommutative linear system is a bigraded $A_\infty$-category $\scrL$ over $R[V]$ such that each space $\scrL(X_0,X_1)$ is generated over $R[V]$ by elements of weights $0$ and $-1$.
}
\end{equation}
There is also an analogue of \eqref{eq:weight01}, replacing complexes of modules by bigraded $R[V]$-modules together with a differential (which preserves weight and increases the homological degree by $1$).

\begin{remark}
Morphism spaces in the $A_\infty$-categories under discussion are (bigraded) dg modules over $R[V]$, such that the underlying module is (bigraded) free. For a general such dg module $M$, the freeness condition may not guarantee that $M$ is homotopically projective in the dg category of such modules. However, in all our constructions, we are dealing only with dg modules whose positive weight parts $M^{(p)}$, $p>0$, are trivial. In that case, there is a natural bounded below increasing filtration $F_k M$, consisting of the submodules generated by elements of weight $\geq -k$; and each quotient is of the form
\begin{equation}
F_k M/F_{k-1}M \iso R[V] \otimes (\text{\it bigraded chain complex of free $R$-modules});
\end{equation}
this implies homotopical projectivity, by \cite[Section 3.1]{keller}. In particular, in the appropriate version of \eqref{eq:weight01}, it is again enough to assume that $\scrL(X_0,X_1) \subset \scrM(X_0,X_1)$ is a quasi-isomorphism.
\end{remark}

Much of our analysis from \eqref{eq:splitting}--\eqref{eq:dual-section} carries over: the subspaces $\scrA = \scrL^{(0)}$ (omitting the objects for brevity) form an $A_\infty$-category $\scrA$, with its homological grading; the quotients $\scrQ = \scrL^{(-1)}/\scrL^{(0)}$ form an $\scrA$-bimodule; and we get a bimodule map $\sigma: \scrQ \rightarrow \scrA \otimes V$. That bimodule map has homological degree zero, but since $V$ itself is graded, it should be viewed as a collection of bimodule maps $\scrQ \rightarrow \scrA$ of different (even) degrees. Suppose that $w \in W$ is a homogeneous element. Its evaluation $R[V] \rightarrow R$ becomes a graded homomorphism if one considers the homological grading minus $|w|$ times the weight grading. Correspondingly:
\begin{equation} \label{eq:graded-fibres}
\mybox{
Given a graded noncommutative linear system, the fibre at a homogeneous $w \in W$ is a $\bZ$-graded $A_\infty$-category, with
\[
\scrF_w(X_0,X_1) = \scrA(X_0,X_1) \oplus \scrQ(X_0,X_1)[1-|w|].
\]
The general fibres (at non-homogeneous points $w$) are $\bZ/2$-graded $A_\infty$-categories. As before, there is a straightforward isomorphism $\scrF_w \iso \scrF_{tw}$ for $t \in R^\times$. Moreover, let $\rho$ be the action of $R^\times$ on $W$ which has weight $-k$ in degree $k$. Then, if we take the fibre at $\scrF_w$ and rescale all its $A_\infty$-operations by $t$, the outcome is isomorphic to $\scrF_{\rho_t(w)}$ (to see that, one applies the automorphism of $\scrF_w$ which acts by $t^{k-1}$ in homological degree $k$). 
}
\end{equation}
Suppose for instance that $R$ is an algebraically closed field, and $V = R \oplus R$ where the two summands have different degrees. Then, each nontrivial fibre of the pencil has one of the following properties: it's isomorphic to the fibre at $(1,0)$; or to that at $(0,1)$; or to a rescaled version of that at $(1,1)$.

\subsection{Localisation\label{subsec:quotient}}
Take a (small) triangulated category $A$, whose morphism spaces we write as $A(X_0,X_1)$. Suppose that we are given a collection of morphisms $S$, meaning subsets $S(X_0,X_1) \subset A(X_0,X_1)$ for all $(X_0,X_1)$. (With our geometric applications in mind, we often use superscripts $+$ to label the objects involved as sources in $S$, and $-$ for the target objects; even though that notation has no particular meaning at this point.) Let $C_S$ be the full triangulated subcategory of $A$ generated by the cones of morphisms in $S$. We write the quotient by that subcategory as 
\begin{equation}  \label{eq:localisation}
S^{-1}A \stackrel{\text{def}}{=} A/C_S.
\end{equation}
It comes with an exact functor $A \rightarrow S^{-1}A$, the localisation functor, which turns elements of $S$ into isomorphisms. The universal property is:
\begin{equation} \label{eq:universal-property}
\mybox{
Suppose that $A \rightarrow B$ is an exact functor between triangulated categories, which turns elements of $S$ into isomorphisms. Then that functor factors through the localisation functor to $S^{-1}A$, in a way which is unique up to isomorphism of functors.
}
\end{equation}
(This is localisation in the world of triangulated categories, which does not necessarily agree with the notion of the same name from general category theory \cite{gabriel-zisman67}.) We recall a basic fact \cite[Lemma 4.8.1]{krause10}:
\begin{equation} \label{eq:projective-object}
\mybox{
Suppose that an object $X_0^+$ has the following property: composition with any element of $S(X_1,X_1^-)$, for arbitrary objects $(X_1,X_1^-)$, yields an isomorphism $A(X_0^+,X_1) \rightarrow A(X_0^+,X_1^-)$. Then, the localisation functor gives an isomorphism $A(X_0^+,X_1) \rightarrow S^{-1}A(X_0^+,X_1)$ for any $X_1$.
}
\end{equation}
One can call objects $X_0^+$ from \eqref{eq:projective-object} $S$-projective. An $S$-projective resolution of $X_0$ would be an $S$-projective object $X_0^+$ together with a morphism in $S(X_0^+,X_0)$. Such resolutions can be used to compute morphisms in the localized category. Here is one straightforward application:
\begin{equation} \label{eq:quotient-hom}
\mybox{
Let $(X_0,X_1)$ be objects, such that for any $X_0^+$ and any element of $S(X_0^+,X_0)$, composition with that element is an isomorphism $A(X_0,X_1) \rightarrow A(X_0^+,X_1)$. Assume moreover that $X_0$ admits an $S$-projective resolution. Then, the localisation functor gives an isomorphism $A(X_0,X_1) \iso S^{-1}A(X_0,X_1)$.
}
\end{equation}
In our applications, there are no $S$-projective resolutions, but we will be able to satisfy a weaker condition:
\begin{equation} \label{eq:more-projective}
\mybox{
We say that an object $X_0$ of $A$ admits approximately $S$-projective resolutions if the following is satisfied. Given finitely many $s_i \in S(X_i,X_i^-)$ ($i = 1,\dots,m$), there exists an object $X_0^+$ and a morphism in $S(X_0^+,X_0)$, such that composition with each $s_i$ is an isomorphism $A(X_0^+,X_i) \rightarrow A(X_0^+,X_i^-)$.
}
\end{equation}
This comes with a generalization of \eqref{eq:quotient-hom}:
\begin{equation} \label{eq:quotient-hom-1b}
\mybox{
Let $(X_0,X_1)$ be objects of $A$, such that for any $X_0^+$ and any element of $S(X_0^+,X_0)$, composition with that element is an isomorphism $A(X_0,X_1) \rightarrow A(X_0^+,X_1)$. Assume moreover that $X_0$ admits approximately $S$-projective resolutions. Then, the localisation functor gives an isomorphism $A(X_0,X_1) \rightarrow S^{-1}A(X_0,X_1)$.
}
\end{equation}
The reason for this is fairly straightforward. By definition, any morphism in $S^{-1}A$ comes from a finite quotient, by which we mean the localisation with respect to a finite subset $S_{\mathit{fin}} \subset S$. Similarly, equality of two such morphisms holds iff it holds in some finite quotient. More formally, we have an isomorphism
\begin{equation}
\underrightarrow{\lim}\, S_{\mathit{fin}}^{-1}A \iso S^{-1}A.
\end{equation}
Take a morphism $q \in S^{-1}A(X_0,X_1)$; find a preimage $q_{\mathit{fin}} \in S_{\mathit{fin}}^{-1}A(X_0,X_1)$ for some $S_{\mathit{fin}} = \{s_1,\dots,s_m\}$, and choose $X_0^+$ accordingly as in \eqref{eq:more-projective}. Then, \eqref{eq:projective-object} says that $A(X_0^+,X_1) \rightarrow S_{\mathit{fin}}^{-1}A(X_0^+,X_1)$ is an isomorphism. Hence, if we compose $q_{\mathit{fin}}$ with the image of $s \in S(X_0^+,X_0)$ under the localisation functor to $S_{\mathit{fin}}^{-1}A$, the outcome is the image of an element of $A(X_0^+,X_1)$; by assumption on $X_1$, that element can itself be written as $as$, for some $a \in A(X_0,X_1)$. After passing to $S^{-1}A$ and inverting $s$, it follows that $a$ itself maps to $q$ under the localisation functor, proving surjectivity of the map $A(X_0,X_1) \rightarrow S^{-1}A(X_0,X_1)$.
The argument for injectivity is similar.

The localisation construction has a chain level version, as follows. For an $A_\infty$-category $\scrA$, let $\scrA^{\tw}$ be its formal enlargement by twisted complexes. Given a set of closed morphisms $S$ in $\scrA$ (meaning morphisms satisfying $\mu^1_\scrA(s) = 0$), we consider the full $A_\infty$-subcategory $\scrC_S \subset \scrA^{\tw}$ formed by their mapping cones. Take the $A_\infty$-quotient $\scrA^{\tw}/\scrC_S$, following \cite{lyubashenko-ovsienko06} (which generalized the dg case from \cite{drinfeld02}). We define
\begin{equation} \label{eq:localisation-2}
S^{-1}\scrA \subset \scrA^{\tw}/\scrC_S
\end{equation}
to be the full subcategory whose objects lie in the image of the functor $\scrA \hookrightarrow \scrA^{\tw} \rightarrow \scrA^{\tw}/\scrC_S$. Concretely, morphism spaces in $S^{-1}\scrA$ are of the form
\begin{equation} \label{eq:localised-morphisms}
\begin{aligned}
S^{-1}\scrA(X_0,X_1)  = \scrA(X_0,X_1) \, & \oplus \bigoplus_C \scrA^{\tw}(C,X_1)[1] \otimes \scrA^{\tw}(X_0,C) \\
& \oplus \bigoplus_{C_0,C_1} \scrA^{\tw}(C_1,X_1)[1] \otimes \scrA^{\tw}(C_0,C_1)[1] \otimes \scrA^{\tw}(X_0,C_0) \\
& \oplus \cdots
\end{aligned}
\end{equation}
where the sums are over cones $C,C_0,C_1,\dots$ in $\scrC_S$. The composition maps are inherited from those in $\scrA^{\tw}$ (applied to parts of the expressions in the tensor products), and the $A_\infty$-localisation functor is just the inclusion of the first summand in \eqref{eq:localised-morphisms}. The cohomology level category $H^0(\scrA^{\tw}/\scrC_S)$ is the localisation of the triangulated category $H^0(\scrA^{\tw})$ by $C_S$, in the previously considered sense. Hence, if $\scrA$ was already triangulated (closed under forming cones), $H^0(S^{-1}\scrA)$ is the localisation $S^{-1} H^0(\scrA)$. Indeed, in that case, we could have used mapping cones in $\scrA$ itself rather than in $\scrA^{\tw}$, making the definition \eqref{eq:localisation-2} formally analogous to \eqref{eq:localisation}. The universal property of localisation in the $A_\infty$-context is:
\begin{equation} \label{eq:dg-quotient-property}
\mybox{
Let $\scrB$ be an $A_\infty$-category, and $\mathit{fun}(\scrA,\scrB)$ the $A_\infty$-category of $A_\infty$-functors. Composition with the localisation functor $\scrA \rightarrow S^{-1}\scrA$ induces a cohomologically full and faithful $A_\infty$-functor $\mathit{fun}(S^{-1}\scrA,\scrB) \rightarrow \mathit{fun}(\scrA,\scrB)$, whose essential image are the functors that map elements of $S$ to isomorphisms on the cohomology level. 
}
\end{equation}

\begin{remark}
To spell out the effect of this construction, consider a single morphism $s \in \scrA(X_1,X_0)$ belonging to $S$, with cone $C$. This comes with canonical morphisms
\begin{equation}
i \in \scrA^{\tw}(X_0,C)^0, \quad p \in \scrA^{\tw}(C,X_1)^1, \quad
\mu^1_{\scrA^{\tw}}(i) = 0, \; \mu^1_{\scrA^{\tw}}(p) = 0, \; \mu^2_{\scrA^{\tw}}(p,i) = 0.
\end{equation}
As a consequence,
\begin{equation}
p \otimes i \in \scrA^{\tw}(C,X_1)[1] \otimes \scrA^{\tw}(X_0,C) \subset
S^{-1}\scrA(X_0,X_1)
\end{equation}
is a degree $0$ cocycle in $S^{-1}\scrA$. On cohomology, this is the inverse of the morphism $s$ we started with (by a computation which we will not reproduce here).
\end{remark}

Our discussion of approximately $S$-projective resolutions carries over to the $A_\infty$-context: by first looking at $S^{-1}\scrA^{\mathit{tw}}$, one reduces it to the previously considered triangulated category version.


\section{Basic notions: geometry\label{sec:geometry}}
This section is a brief review of the geometric objects that appear in Lagrangian Floer theory. It also introduces notation used in the rest of the paper. We exclude more advanced topics (such as transversality or Gromov compactness), which will be discussed when the need arises.

\subsection{Riemann surfaces}
Throughout, all Riemann surfaces will be assumed to be connected. The class of surfaces relevant for our (open string) theory is this:
\begin{equation} \label{eq:riemann}
\mybox{
Take a compact Riemann surface with boundary $S^\bullet$, and remove a finite set $\Sigma_S$ of boundary points, arbitrarily divided into two parts $\Sigma_S^{\pm}$. The resulting $S = S^\bullet \setminus \Sigma_S$ will be called a punctured-boundary Riemann surface.
}
\end{equation}
A simple example is the infinite strip $Z = \bR \times [0,1]$, with $\Sigma_Z^{\pm} =\{\pm\infty\}$, whose compactification $Z^\bullet$ is isomorphic to a closed disc. We write $(s,t)$ for the standard coordinates on $Z$.
\begin{equation}
\mybox{
A set of ends for a boundary-punctured Riemann surface $S$ consists of proper holomorphic embeddings of the half-infinite strips $Z^{\pm} = \bR^{\pm} \times [0,1] = \{\pm s \geq 0\} \times [0,1] \subset Z$:
\[
\epsilon_\zeta: 
Z^\pm \longrightarrow S \;\; \text{ for $\zeta \in \Sigma_S^\pm$, with }
\epsilon_{\zeta}^{-1}(\partial S) = \{t  = 0,1\}.
\]
Each embedding is asymptotic to the respective point of $\Sigma_S$, and their images must be pairwise disjoint. 
}
\end{equation}
From the viewpoint of $S^\bullet$, the ends are obtained by taking a local complex coordinate $z$ ($|z| \leq 1$, $\mp \mathrm{im}(z) \geq 0$) near a point of $\Sigma_S^{\pm}$, and then setting $s+it = \mp \log(z)/\pi$.
\begin{equation} \label{eq:bar-compactification}
\mybox{
Take a punctured-boundary Riemann surfaces $S$, and add intervals $\{\zeta\} \times [0,1]$, $\zeta \in \Sigma_S$, so as to obtain a smooth (actually real analytic) surface with corners $|S$, which one can think of as the real oriented blowup of $S^\bullet$ at the set of boundary points $\Sigma_S$. More concretely, a choice of ends yields coordinates on $|S$ near the intervals at infinity, 
\[
\left\{
\begin{aligned}
& \sigma = e^{\mp \pi s} = |z| \geq 0,\\ &  t = \mp \mathrm{arg}(z)/\pi \in [0,1]. 
\end{aligned}
\right.
\]
}
\end{equation}
The smooth structure of $|S$, as well as the parametrization of the intervals at infinity, are canonical. To check that, note that a holomorphic coordinate change $\tilde{z} = z e^{\phi(z)}$ (near the origin in the closed half-plane, with $\phi$ a real analytic function) induces a smooth (and real analytic) coordinate change 
\begin{equation} \label{eq:coordinate-change}
\left\{\begin{aligned}
& \tilde{\sigma} = \sigma \exp(\mathrm{re}(\phi(\sigma e^{\pi it}))), \\ 
& \tilde{t} = t + \mathrm{im}(\phi(\sigma e^{\pi it}))/\pi,
\end{aligned}
\right.
\end{equation}
which reduces to the identity for $\sigma = \tilde{\sigma} = 0$. The signs in \eqref{eq:coordinate-change} have been chosen for a point of $\Sigma_S^-$, but the other case is of course parallel.

Next, let's look at the standard process of gluing together surfaces.
\begin{equation}
\label{eq:gluing}
\mybox{
Suppose that we have surfaces $\bfS_v$ ($v = 1,2$) with ends. Choose points at infinity $\bfzeta_1 \in \Sigma^+_{\bfS_1}$, $\bfzeta_2 \in \Sigma^-_{\bfS_2}$, and denote the coordinates on those ends by $(s_v,t)$. Given a gluing parameter $\gamma \in (0,1)$, or equivalently a gluing length $l = -\log(\gamma)/\pi \in (0,\infty)$, one constructs $S_{\gamma}$ by removing part of the ends, and identifying the rest:
\[
\begin{aligned}
& S_{\gamma} = \big(\bfS_1 \setminus \epsilon_{\bfzeta_1}(\{s_1 > l\})\big) \cup_{\sim} \big(\bfS_2 \setminus \epsilon_{\bfzeta_2}(\{s_2 < -l\})\big), \\
& \qquad \qquad \qquad \text{where  } \epsilon_{\bfzeta_1}(s_1,t) \sim \epsilon_{\bfzeta_2}(s_2,t), \;\; s_2 = s_1-l.
\end{aligned}
\]
}
\end{equation}
The glued surfaces fit into a family $\scrS \rightarrow (0,1)$ parametrized by $\gamma$, and one can partly extend that to $\bar\scrS \rightarrow [0,1)$ by adding the disjoint union of $\bfS_1$ and $\bfS_2$ as a fibre over $\gamma = 0$. Furthermore:
\begin{equation} \label{eq:gluing-2}
\mybox{
There is a fibrewise compactification 
\[
\cornerbar{\scrS} \rightarrow [0,1)
\]
of $\bar{\scrS}$, where the fibre over $\gamma \neq 0$ is the compactification $|S_{\gamma}$, and the fibre over $\gamma = 0$ is the union of $|\bfS_1$ and $|\bfS_2$, with the intervals at infinity $\{\bfzeta_k\} \times [0,1]$ identified. The total space $\cornerbar{\scrS}$ is a smooth three-manifold with corners, and as such, it is independent of the choice of ends (but the fibrewise complex structure depends on that choice). In the same notation as in \eqref{eq:gluing}, local coordinates near the corner locus where the gluing takes place are
\[
\left\{
\begin{aligned}
& \sigma_1 = e^{-\pi s_1} \geq 0, \\ 
& \sigma_2 = e^{\pi s_2} \geq 0, \\
& t \in [0,1],
\end{aligned}
\right.
\]
and in those coordinates, $\gamma = \sigma_1\sigma_2$. 
}
\end{equation}
Let $f$ be a smooth function on $\cornerbar{\scrS}$ which vanishes on the fibre at $0$. This can be written as $f(\sigma_1,\sigma_2,t) = \sigma_1\sigma_2 g(\sigma_1,\sigma_2,t)$ for some smooth function $g$. If we restrict to a nonzero fibre, and use the more symmetric variable $s = s_1-l/2 = s_2+l/2$, we get
\begin{equation} \label{eq:length-decay}
f(s,t) = e^{-\pi l} g(e^{-\pi s- \pi l/2}, e^{\pi s- \pi l/2}, t).
\end{equation}
In particular, on any fixed bounded subset of the ``necks'' $(s,t) \in [-l/2, l/2] \times [0,1]$, we have $f \rightarrow 0$ exponentially as $l \rightarrow \infty$.

\subsection{One-forms}
Let $S$ be any Riemann surface with boundary. 
\begin{equation} 
\label{eq:beta0}
\mybox{
We consider real one-forms $\beta_S \in \Omega^1(S)$ such that $\beta_S|\partial S = 0 \in \Omega^1(\partial S)$, and $d\beta_S$ is compactly supported.
%
}
\end{equation}
In our applications, the following more specific class of one-forms appears.
\begin{equation}
\label{eq:beta1}
\mybox{
Let $S$ be a punctured-boundary Riemann surface. A one-form $\beta_S$ as in \eqref{eq:beta0} is called asymptotically translation-invariant if it extends smoothly to $|S$, and its restriction to the intervals at infinity $\{\zeta\} \times [0,1]$ is $w_\zeta \mathit{dt}$, for some constants $w_\zeta$. In the coordinates given by a choice of ends, this means that for $\pm s \gg 0$,
\[
\epsilon_\zeta^*\beta_S = w_{\zeta} \mathit{dt} + d (\sigma g_\zeta(\sigma,t)) 
= w_\zeta \mathit{dt} + d(e^{\mp \pi s} g_\zeta(e^{\mp \pi s}, t)),
\]
for $\zeta \in \Sigma_S^{\pm}$, where $g_\zeta(\sigma,t)$ is a smooth function that vanishes for $t = 0,1$. Hence, in the coordinates $(s,t)$ on the end, we have exponential decay towards the limit $w_\zeta \mathit{dt}$, which explains the terminology.
}
\end{equation}
There is also a more restrictive version of that property, which depends on the ends:
\begin{equation} \label{eq:restricted-beta}
\mybox{
A one-form \eqref{eq:beta0} is called compatible with a choice of ends if its restrictions to those ends equals $w_\zeta \mathit{dt}$.
}
\end{equation}

\subsection{Discs}
The Riemann surfaces which will be most relevant for us are these:
\begin{equation} \label{eq:punctured-disc}
\mybox{
A boundary-punctured disc is a Riemann surface as in \eqref{eq:riemann}, where $S^\bullet$ is isomorphic to a closed disc, and $\Sigma^-_S = \{\zeta_0\}$ consists of a single point. We number the remaining points $\Sigma^+_S = \{\zeta_1,\dots,\zeta_d\}$ in accordance with the boundary orientation, and use the same numbering for the ends. The components of $\partial S$ are numbered as $I_0,\dots,I_d$, starting with the component between $\zeta_0$ and $\zeta_1$, and continuing in the direction of the boundary orientation.
}
\end{equation}
In this context, we find it convenient to restrict the choices of ends somewhat:
\begin{equation} \label{eq:rational-ends}
\mybox{
Let $S$ be a boundary-punctured disc. A set of ends $(\epsilon_0,\dots,\epsilon_d)$ is called rational if the following holds: $\epsilon_0$ extends to an isomorphism $Z^\bullet \rightarrow S^\bullet$, necessarily taking $-\infty$ to $\zeta_0$; and for $k>0$, $\epsilon_k$ again extends to an isomorphism $Z^\bullet \rightarrow S^\bullet$, which besides taking $+\infty$ to $\zeta_k$, also takes $-\infty$ to $\zeta_0$.
}
\end{equation}
In tune with the classical ``little intervals'' operad, one can think of the situation as follows: 
\begin{equation} \label{eq:upper-half-plane}
\mybox{
Let $H = \{\mathrm{im}(z) \geq 0\}$ be the closed upper half-plane, and $H^\bullet = H \cup \{i\infty\}$ its compactification. Given a boundary-punctured disc with rational ends, one can identify $S^\bullet \iso H^\bullet$ in a unique way, so that $\zeta_0$ corresponds to $i\infty$, and the end $\epsilon_0$ covers the region $\{|z| \geq 1\}$. Then, the other $\zeta_k$ ($k>0$) become ordered points on the real line; and the images of the $\epsilon_k$ are semidiscs centered at the $\zeta_k$, with some radii $\rho_k$.
}
\end{equation}
The gluing process \eqref{eq:gluing} for discs with rational ends can be reformulated accordingly:
\begin{equation} \label{eq:gluing-again}
\mybox{
Let $\bfS_v$ ($v = 1,2$) be discs with $(\bfd_v+1)$ boundary punctures, and rational ends. Identify $\bfS_v^\bullet \iso H^\bullet$ as in \eqref{eq:upper-half-plane}, so that the points of $\Sigma^+_{\bfS_v}$ are $\bfzeta_{v,k} \in \partial H = \bR$. Choose some $1 \leq i \leq \bfd_1$, and a gluing parameter $\gamma$. Let $\bfrho_{1,i}$ be the radius of the semicircle around $\bfzeta_{1,i}$ determined by its strip-like end. Then, the glued surface corresponds to the following collection of $d = (\bfd_1 + \bfd_2- 1)$ points on the real line:
\[
\zeta_k = \begin{cases} 
\bfzeta_{1,k} 
& k<i, \\ 
\bfzeta_{1,i} + \bfrho_{1,i}\, \gamma\, \bfzeta_{2,k-i+1} 
& i \leq k \leq i +\bfd_2 - 1, \\
\bfzeta_{1,k-\bfd_2 +1} 
& k \geq i+\bfd_2.
\end{cases}
\]
}
\end{equation}

\subsection{Cauchy-Riemann equations}
The basic form is this:
\begin{equation} \label{eq:cr}
\mybox{
Let $S$ be a Riemann surface with boundary (whose complex structure we denote by $j$), equipped with a one-form \eqref{eq:beta0}. Take a symplectic manifold $M$ with a Hamiltonian vector field $X$. For each connected component $I \subset \partial S$, let a Lagrangian submanifold $L_I \subset M$ be given. Moreover, we want to have a family $J_S = (J_{S,z})_{z \in S}$ of compatible almost complex structures on $M$. Then, we consider maps
\[
\left\{
\begin{aligned}
& u: S \longrightarrow M, \\
& (Du - X \otimes \beta_S)^{0,1} = \half (Du + J_{S,z} \circ Du \circ j - X \otimes \beta_S - J_{S,z} X \otimes \beta_S \circ j) = 0,  
\\
& u(z) \in L_I \;\; \text{ for $z \in I \subset \partial S$.}
\end{aligned}
\right.
\]
}
\end{equation}
The classical Gromov trick allows one to get rid of the inhomogeneous term:
\begin{equation} \label{eq:graph}
\mybox{
Set $u^*(z) = (z,u(z))$. Then \eqref{eq:cr} is equivalent to
\[
\left\{
\begin{aligned}
& u^*: S \longrightarrow S \times M \text{ a section}, \\
& (Du^*)^{0,1} = \half (Du^* + J_S^* \circ Du^* \circ j) = 0, \\ 
& u^*(z) \in L_I^* \;\; \text{ for $z \in I \subset \partial S$.}
\end{aligned}
\right.
\]
Here, $J_S^*$ is the unique almost complex structure on $S \times M$ which makes projection to $S$ holomorphic, restricts to $J_z$ on the fibres $\{z\} \times M$, and satisfies 
\[
J_S^*(\xi) = j\xi + X\, \beta_S(j\xi) - J_{S,z} X\, \beta_S(\xi) \;\; \text{ for } \xi \in TS;
\]
and $L_I^* = I \times L_I$.
}
\end{equation}
In the special case where $\beta_S$ is exact, one can achieve a similar effect in a different way, without raising the dimension. Namely, let $(\phi^t)$ be the flow of $X$. 
\begin{equation} \label{eq:gauge}
\mybox{Suppose that $\beta_S = db_S$, and write $u(z) = \phi^{b_S(z)}(u^\dag(z))$. Then \eqref{eq:cr} is equivalent to
\[
\left\{
\begin{aligned}
& u^\dag: S \longrightarrow M, \\
& (Du^\dag)^{0,1} = \half(Du^\dag + J_{S,z}^\dag \circ Du^\dag \circ j) = 0, \\
& u^\dag(z) \in L_I^\dag \;\; \text{ for } z \in I \subset \partial S,
\end{aligned}
\right.
\]
where $J^\dag_{S,z} = (\phi^{b_S(z)})^*J_{S,z}$ and $L_I^\dag = (\phi^{b_S|I})^{-1}(L_I)$ (by assumption, $b_S$ is constant on each boundary component).
}
\end{equation}

\subsection{Limits of solutions}
The translation-invariant special case of \eqref{eq:cr}, with $S = Z$ and $\beta_S = w \mathit{dt}$, is Floer's equation: 
\begin{equation} \label{eq:floer}
\left\{
\begin{aligned}
& u: Z \longrightarrow M, \\
& u(s,0) \in L_0, \;\; u(s,1) \in L_1, \\
& \partial_s u + J_t (\partial_t u - wX) = 0. \\
\end{aligned}
\right.
\end{equation}
Translation-invariant solutions are chords of ``length'' $w$:
\begin{equation} \label{eq:chord}
\left\{
\begin{aligned}
& x: [0,1] \longrightarrow M, \\
& x(0) \in L_0, \, x(1) \in L_1, \\
& x'(t) = wX,
\end{aligned}
\right.
\end{equation}
which correspond to points $x(1) \in \phi^w(L_0) \cap L_1$. It therefore makes sense to impose convergence conditions on general solutions of \eqref{eq:floer}, of the form
\begin{equation} \label{eq:floer-limits}
\textstyle \lim_{s \rightarrow -\infty} u(s,\cdot) = x_0, \;\; \lim_{s \rightarrow +\infty} u(s,\cdot) = x_1,
\end{equation}
for $x_0$, $x_1$ as in \eqref{eq:chord}. If the chords are nondegenerate, which means that the intersection $\phi^w(L_0) \cap L_1$ is transverse, convergence in \eqref{eq:floer-limits} is exponentially fast (with rate $e^{\pm \lambda s}$, where $\lambda>0$ depends on the specific chord).

We now generalize from the strip to other geometric situations:
\begin{equation} \label{eq:asymptotically}
\mybox{
Let $S$ be a boundary-punctured Riemann suface. A family of almost complex structures $(J_{S,z})_{z \in S}$ is called asymptotically translation-invariant if it extends smoothly to $|S$. We write $(J_{\zeta,t})_{t \in [0,1]}$ for its values on the intervals at infinity $\{\zeta\} \times [0,1]$.
}
\end{equation}
If we choose ends, this means that there is exponential convergence
\begin{equation} \label{eq:j-decays}
J_{S,\epsilon_\zeta(s,t)} = J_{\zeta,t} + e^{\mp \pi s} O_{\zeta,e^{\mp s}, t},
\end{equation}
where the error term $O_{\zeta,\sigma,t}$ is smooth at all $(\sigma,t) \in [0,1]^2$. This is more flexible than the classical requirement, which would ask for strict compatibility with the ends,
\begin{equation} \label{eq:j-strict}
J_{S,\epsilon_\zeta(s,t)} = J_{\zeta,t}.
\end{equation}
Take a family of almost complex structures as in \eqref{eq:j-decays}, and a one-form with the corresponding property \eqref{eq:beta1}. Then, the resulting equation \eqref{eq:cr} is asymptotic to an equation \eqref{eq:floer} on each end. Hence, it makes sense to consider solutions whose limits 
\begin{equation} \label{eq:zeta-limits}
\textstyle
\lim_{s \rightarrow \pm\infty} u(\epsilon_\zeta(s,\cdot)) = x_\zeta \;\; \text{ for } \zeta \in \Sigma_S^{\pm}
\end{equation}
are chords $x_\zeta$ of length $w_\zeta$, between the Lagrangian submanifolds $(L_{\zeta,0}, L_{\zeta,1})$ associated to points $\epsilon_\zeta(s,0), \epsilon_\zeta(s,1) \in \partial S$. Even though we have used ends in \eqref{eq:zeta-limits}, one could avoid that by reformulating the condition as follows: $u$ extends continuously to $|S$, with values $x_\zeta$ on the intervals at infinity (note that even if the chords are nondegenerate, this extension is not usually differentiable with respect to the manifold structure of $|S$).

Let's consider in particular boundary-punctured discs, and introduce some convenient notation for that case. Take such a disc $S$, equipped with Lagrangian submanifolds $L_0,\dots,L_d$, corresponding to the components of $\partial S$ in the numbering from \eqref{eq:punctured-disc}. Then, the boundary conditions in \eqref{eq:cr} can be written as
$u(I_k) \subset L_k$;
and if $w_0,\dots,w_d$ are the constants governing the behaviour  of $\beta_S$ at infinity, \eqref{eq:zeta-limits} becomes
\begin{equation} \label{eq:disc-limits}
\textstyle
\left\{\begin{aligned}
& \textstyle \lim_{s \rightarrow -\infty} u(\epsilon_0(s,\cdot)) = x_0 \;\; \text{ a chord of length $w_0$ between $(L_0,L_d)$,} \\
& \textstyle \lim_{s \rightarrow +\infty} u(\epsilon_k(s,\cdot)) = x_k \;\; \text{a chord of length $w_k$ between $(L_{k-1},L_k)$, for $k>0$.}
\end{aligned}
\right.
\end{equation}

\subsection{Energy\label{subsec:energy}}
For a solution of \eqref{eq:cr}, one defines the geometric energy as
\begin{equation} \label{eq:e-geom}
E^{\mathit{geom}}(u) = \int_S \half \| Du - X \otimes \beta_S \|^2 = \int_S u^*\omega_M - d(u^*H) \wedge \beta_S,
\end{equation}
where the norm is with respect to the conformal structure on $S$ and the metrics $\omega_M(\cdot,J_{S,z}\cdot)$ on $M$; and $H$ is the Hamiltonian giving rise to $X$. The topological energy is
\begin{equation} \label{eq:e-e}
E^{\mathit{top}}(u) =  \int_S u^*\omega_M - d(u^*H \, \beta_S) = E^{\mathit{geom}}(u) - \int_S u^*H\, d\beta_S.
\end{equation}
Strictly speaking, in this level of generality, the integrals may not converge. Hence, let's restrict to the case of a punctured-boundary Riemann surface, where the one-form and almost complex structures are asymptotically translation-invariant, and consider solutions with nondegenerate limits \eqref{eq:zeta-limits}. Then, the energies are well-defined; and the topological energy remains constant under deformations of $u$ (indeed, it is a topological invariant of a suitable class of maps, without imposing the Cauchy-Riemann equation).
\begin{equation} \label{eq:sub-closed0}
\mybox{
A one-form \eqref{eq:beta0} is called sub-closed if $d\beta_S \leq 0$, with respect to the orientation of $S$.
}
\end{equation}
This condition allows one to estimate the last term in \eqref{eq:e-e}, as follows:
\begin{equation} \label{eq:energy-control}
\mybox{
Consider a solution of \eqref{eq:cr}, \eqref{eq:zeta-limits}. Suppose that $\beta_S$ is sub-closed, and $H(x) \geq C$ everywhere. Then
\[
E^{\mathit{geom}}(u) \leq E^{\mathit{top}}(u) + C \int_S d\beta_S.
\]
}
\end{equation}

Finally, suppose that we are in the following more special geometric situation:
\begin{equation} \label{eq:exact-symplectic}
\mybox{
The symplectic form is exact, $\omega_M = d\theta_M$; and all Lagrangian submanifolds under consideration are exact with respect to that primitive, $\theta_M|L = dK_L$. 
}
\end{equation}
Then, the action of a chord \eqref{eq:chord} is defined as
\begin{equation} \label{eq:action}
A(x) = \int_{[0,1]} x^*(-\theta_M + w H\mathit{dt}) + K_{L_1}(x(1)) - K_{L_0}(x(0)).
\end{equation}
By applying Stokes, one gets
\begin{equation} \label{eq:energy-and-action}
E^{\mathit{top}}(u) = \sum_{\zeta \in \Sigma_S^{\pm}} \mp A(x_\zeta) .
\end{equation}

\subsection{Moving boundary conditions}
This is a mild generalization of our previous setup:
\begin{equation} \label{eq:follows}
\mybox{
We have a family $(L_z)_{z \in \partial S}$ of Lagrangian submanifolds, which is a Hamiltonian isotopy and constant outside a compact subset of $\partial S$. To express this more formally, write $G = \bigcup_z \{z\} \times L_z \subset \partial S \times M$. Then there is a $\chi_G \in \Omega^1(G)$ vanishing on each fibre $L_z$, and which is also zero outside the preimage of a compact subset of $\partial S$, such that
\[
d\chi_G = \omega_M|G.
\]
Concretely, given $\xi \in T(\partial S)_z$, we can evaluate $\chi_G$ at any lift of $\xi$; the outcome, denoted simply by $\chi_G(\xi)$, is a function on $L_z$ describing the infinitesimal Lagrangian deformation $\partial_\xi L_z$.
}
\end{equation}
In \eqref{eq:cr}, we then change the boundary condition to
\begin{equation}
u(z) \in L_z \;\; \text{ for $z \in \partial S$.}
\end{equation}
This affects the equation only on a compact subset of $S$. Hence, it makes sense to impose the same limiting behaviour \eqref{eq:zeta-limits} as before. To our previous definition of the topological energy of solutions, we add an extra term
\begin{equation} \label{eq:corrected-topological-energy}
E^{\mathit{top}}(u) = \int_S u^*\omega_M - d(u^*H \beta_S) \;- \int_{\partial S} u^*\chi_G 
= E^{\mathit{geom}}(u) - \int_S u^*H \, d\beta_S \; - \int_{\partial S} u^*\chi_G.
\end{equation}
Suppose that we are in the exact symplectic context \eqref{eq:exact-symplectic}. Then $\theta_M|G - \chi_G \in \Omega^1(G)$ is a one-form which is globally closed, and exact on each fibre. Hence, possibly after adjusting our choice of $\chi_G$, we can write
\begin{equation} \label{eq:k-gap}
\theta_M|G - \chi_G = dK_G,
\end{equation}
where $K_G \in \smooth(G, \bR)$ is constant in $z \in \partial S$ outside a compact subset. Bearing in mind that $\beta_S|\partial S = 0$, this means that
\begin{equation} \label{eq:modified-energy-2}
E^{\mathit{top}}(u) = \int_S d(u^*\theta_M - u^*H \beta_S) - \int_{\partial S} (u^*\theta_M - u^*H \beta_S) + \int_{\partial S} d(u^*K_G),
\end{equation}
which one can integrate out, leading to the same formula \eqref{eq:energy-and-action} as before, provided that the values of $K_G$ close to each $\zeta$ are used to define the action $A(x_\zeta)$.

\subsection{The complex plane as target space\label{subsec:maps-to-c}}
Let's look at a simple special case of \eqref{eq:cr}, where: the symplectic manifold is the complex plane; the Hamiltonian vector field is infinitesimal rotation; and the Lagrangians are radial lines (in fact half-lines):
\begin{equation} \label{eq:cr-c}
\left\{
\begin{aligned}
& v: S \longrightarrow \bC, \\
& \bar\partial v - 2\pi i v \beta_S^{0,1} = 0, \\
& v(I) \subset e^{2\pi i \alpha_I} \bR^+. \\
\end{aligned}
\right.
\end{equation}
This is a linear $\bar\partial$-equation with totally real boundary conditions, and the standard elliptic theory applies to it. In particular, there are local nonvanishing solutions around any point of $S$; and if $v$ is such a solution, any other local solution can be written as $fv$, where $f$ is a holomorphic function with real nonnegative boundary values. In particular:
\begin{equation} \label{eq:taylor}
\mybox{
Let $v$ be a solution of \eqref{eq:cr-c} which is not identically zero. Then $v^{-1}(0)$ is discrete. Take an interior point of $v^{-1}(0)$, and choose a local complex coordinate $\xi$ centered at that point. Then, the vanishing multiplicity of $v$ is finite, meaning that there is a positive integer $\mu$ such that
\[
v(\xi) = \xi^\mu (c + O(|\xi|)), \;\; \text{ for } c \in \bC^*.
\]
For a boundary point, we similarly have, in a local coordinate with $\mathrm{im}(\xi) \geq 0$,
\[
v(\xi) = \xi^{2\mu} (c + O(|\xi|)), \;\; \text{ for } c \in e^{2\pi i \alpha_I} \bR^{>0}.
\]
We write $\mu = \mu_z(v)$, and extend that to all points of $S$ by setting $\mu_z(v) = 0$ if $v(z) \neq 0$.
}
\end{equation}
The fact that the boundary conditions are half-lines rather than whole lines is relevant here, since it causes the Taylor  exponent in the boundary case to be even (otherwise, boundary multiplicities would be half-integers). More geometrically, the boundary multiplicities can be thought of as follows. Choose a path $\gamma: [0,1] \rightarrow S \setminus \{z\}$ which: remains in a small punctured neighbourhood of $z$; starts at a boundary point to the right of $z$; and ends at a boundary point to the left of $z$. Then $v(\gamma)/|v(\gamma)|$ is a loop in $S^1$, whose winding number is $\mu_z(v)$.

For punctured-boundary Riemann surfaces with one-forms \eqref{eq:beta1}, one can also look at the behaviour near the points at infinity, assuming nondegeneracy of the limits. The model is
\begin{equation} \label{eq:cr-c-strip}
\left\{
\begin{aligned}
& v: Z^{\pm} \longrightarrow \bC, \\
& \partial_s v + i \partial_t v - 2\pi i v (\beta_{Z^\pm}(\partial_s) + i \beta_{Z^\pm}(\partial_t)) = 0, \\
& v(s,0) \in e^{2\pi i \alpha_0}\bR^+, \; v(s,1) \in e^{2\pi i \alpha_1} \bR^+, \\
& \textstyle \lim_{s \rightarrow \pm \infty} v(s,\cdot) = 0.
\end{aligned}
\right.
\end{equation}
Here, the one-form $\beta_{Z^\pm}$ is closed, vanishes when restricted to $\{t = 0,1\}$, and is asymptotic to $w \mathit{dt}$, in the sense of \eqref{eq:beta1}. Moreover, $\alpha_0, \alpha_1 \in \bR$ satisfy $\alpha_1 - \alpha_0 - w \notin \bZ$. Suppose for concreteness that we are on $Z^-$. Then, as a toy model application of \eqref{eq:gauge}, one can write 
\begin{equation} \label{eq:v-dag}
v(s,t)  = \exp(2\pi i (w t + e^{\pi s} g(e^{\pi s},t))) v^\dag(s,t),
\end{equation}
where $g(\sigma,t) \in \bR$ is as in \eqref{eq:beta1}; and the transformed map $v^\dag$ is holomorphic, with the same boundary and convergence conditions as \eqref{eq:cr-c-strip}. The Fourier expansion of $v^\dag$ yields the following asymptotic behaviour as $s \rightarrow -\infty$:
\begin{equation} \label{eq:v-dag-2}
v(s,t) \sim c e^{2\pi (\alpha_1 - \alpha_0 + \nu) s} e^{2\pi i (\alpha_1-\alpha_0+\nu+w) t}, 
\end{equation}
for some integer $\nu > \alpha_0 - \alpha_1$, and $c \in e^{2\pi i \alpha_0} \bR^{>0}$. In the counterpart for $Z^+$, we have $\nu < \alpha_0 - \alpha_1$ instead.

Besides those elementary tools from complex analysis, one can also use more symplectic arguments. For those, we adopt the standard symplectic form $\omega_{\bC}$, so that the Hamiltonian is $H(w) = \pi |w|^2$; and its primitive $\theta_{\bC} = \frac{i}{4} (w\,d\bar{w} - \bar{w} dw)$. Since our Hamiltonian is everywhere nonnegative, we have a particularly simple version of \eqref{eq:energy-control}, namely
\begin{equation} \label{eq:toy-energy-bound} 
E^{\mathit{geom}}(v) \leq E^{\mathit{top}}(v) \;\; \text{ when $\beta_S$ is sub-closed.} 
\end{equation}
We also want to mention a related result, which is a version of the integrated maximum principle \cite[Lemma 7.2]{abouzaid-seidel07}:
\begin{equation}  \label{eq:integrated}
\mybox{
Suppose that $\beta_S$ is sub-closed. Let $v$ be a solution of \eqref{eq:cr-c}. Assume that, for some $r>0$, $v$ intersects the circle $\{|w| = r\}$ transversally, and that the intersection is nonempty. Then, $T = v^{-1}(\{|w| \geq r\}) \subset S$ cannot be compact.
}
\end{equation}
To prove it, suppose on the contrary that $T$ is compact. We consider modified notions of energy, where the integration is over $T \subset S$. The same argument as in \eqref{eq:energy-control} shows that
\begin{equation} \label{eq:model-e-e}
0 \leq E^{\mathit{geom}}(v|T) \leq E^{\mathit{top}}(v|T).
\end{equation}
Write $\partial_{\mathit{in}}T = v^{-1}(\{|w| = r\}) \subset \partial T$, orienting it with the boundary orientation. Then
\begin{equation} \label{eq:t-energy-1}
E^{\mathit{top}}(v|T) = \int_{\partial_{\mathit{in}}T} v^*\theta_{\bC} - \pi r^2 \beta_S.
\end{equation}
At any positively oriented tangent vector $\xi$ to $\partial_{\mathit{in}}T$, we know that $j\xi$ points towards the interior of $T$, hence $\mathit{dv}(j\xi)$ has positive radial component, or equivalently $\theta_{\bC}(i\, \mathit{dv}(j\xi)) > 0$. Hence
\begin{equation} \label{eq:t-energy-2}
\theta_{\bC}(\mathit{dv}(\xi)) = \theta_{\bC}(-i \,\mathit{dv}(j\xi) + 2\pi v \beta_S(j\xi) - 2\pi i v \beta_S(\xi)) <
\pi r^2 \beta_S(\xi).
\end{equation}
It follows from \eqref{eq:t-energy-1} and \eqref{eq:t-energy-2} that $E^{\mathit{top}}(v|T) < 0$, which contradicts \eqref{eq:model-e-e}. 

To conclude this discussion, let's look briefly at the moving boundary condition version of the same situation, which means that we change the relevant part of \eqref{eq:cr-c} to
\begin{equation} \label{eq:moving-radial-lines}
v(z) \in e^{2\pi i \alpha_{\partial S}} \bR^+, \; \text{ for some } \alpha_{\partial S} \in \smooth(\partial S,\bR/\bZ).
\end{equation}
The observations leading to \eqref{eq:taylor} carry over immediately to this context. In the terminology of \eqref{eq:follows}, one can take 
\begin{equation} \label{eq:toy-chi}
\chi_G = H\, d\alpha_{\partial S}. 
\end{equation}
In particular, if one adds the condition that 
\begin{equation} \label{eq:toy-isotopy}
d\alpha_{\partial S} \leq 0,
\end{equation}
then \eqref{eq:toy-energy-bound} still holds. The same additional assumption also makes \eqref{eq:integrated} go through.

\section{Popsicle structures\label{sec:popsicle}}

This section recalls the construction of popsicle moduli spaces, largely following \cite{abouzaid-seidel07}. In brief, popsicles are discs with both boundary and interior marked points, where the interior marked points are constrained to lie on certain curves connecting the boundary marked points. We depart from our source in two respects. In \cite[Section 2d]{abouzaid-seidel07}, the naturally occurring curves (hyperbolic geodesics on the disc) were deformed in a non-canonical way, so as to make the gluing process for popsicles easier to describe. The modified curves partially overlap, which would be unfortunate for our purpose. Therefore, we prefer to stick with the unmodified version, while still using the same terminology. The other difference is that, when it comes to popsicles equipped with integer weights associated to the boundary marked points, what's relevant here are nonpositive weights, and most importantly weights $\{-1,0\}$, as opposed to the nonnegative weights in \cite{abouzaid-seidel07}.

\subsection{Popsicles}
To begin with, we introduce the curves mentioned above.
\begin{equation} \label{eq:popsicle-sticks}
\mybox{
Let $S = S^\bullet \setminus \{\zeta_0,\dots,\zeta_d\}$ be a boundary-punctured disc. For $1 \leq k \leq d$, let $Q_k^\bullet \subset S^\bullet$ be the fixed point set of the unique antiholomorphic involution which fixes $\zeta_0$ and $\zeta_k$, and $Q_k = Q_k^\bullet \cap S$. In the picture from \eqref{eq:upper-half-plane}, this is the vertical line $Q_k = \{\mathrm{re}(z) = \zeta_k\}$. Equivalently, if we consider the interior of $S$ as the hyperbolic disc, $Q_k$ is the geodesic connecting the points at infinity $\zeta_0$ and $\zeta_k$. We call the $Q_k$ popsicle sticks.
}
\end{equation}
The fundamental definition is as follows \cite[Section 2c]{abouzaid-seidel07}:
\begin{equation} \label{eq:popsicle}
\mybox{
Take integers $d \geq 1$ and $m \geq 0$, together with a nondecreasing map $p: \{1,\dots,m\} \rightarrow \{1,\dots,d\}$ (we will often write $|p| = m$ for the size of its domain). A popsicle of type $(d+1,p)$ is a pair $(S,\sigma)$ consisting of a disc $S$ with $(d+1)$ boundary punctures, together with, for each $1 \leq f \leq |p|$, a choice of point $\sigma_f \in Q_{p(f)}$. We call those points (some of which may coincide, if they lie on the same popsicle stick) sprinkles.
}
\end{equation}
Suppose that in addition, $S$ carries rational ends, hence comes with a preferred identification $S^\bullet \iso H^\bullet$ as in \eqref{eq:upper-half-plane}. In that case, the sprinkles can be described by numbers $s_f>0$, namely
\begin{equation} \label{eq:s-values}
\sigma_f = \zeta_{p(f)} + i \,s_f \in H.
\end{equation}

We will also need a degenerate version, called broken popsicles. Take $(d+1,p)$ as in \eqref{eq:popsicle}. Consider a ribbon tree $T$ with $(d+1)$ semi-infinite edges, and with vertices of valency $|v| \geq 2$. One of the semi-infinite edges is singled out as the root,  and the others are called the leaves. Starting with the root and proceeding anticlockwise, we number all semi-infinite edges by $\{0,\dots,d\}$. Similarly, for each vertex $v$, there is a single adjacent edge which either is the root or connects the vertex to the root; starting there, we number all adjacent edges by $\{0,\dots,|v|-1\}$. The additional combinatorial data that enter into a broken popsicle are:
\begin{equation} \label{eq:broken-popsicle-data}
\mybox{ \parskip.5em
Equip each vertex of the ribbon tree $T$ with a nonnegative integer $\bfm_v$, such that $\sum_v \bfm_v = m$, as well as a nondecreasing map $\bfp_v: \{1,\dots,\bfm_v\} \rightarrow \{1,\dots,|v|-1\}$. Additionally, there is an increasing map $\bfr_v: \{1,\dots,\bfm_v\} \rightarrow \{1,\dots,m\}$ for each $v$, and those maps combine to form a bijection $\bfr: \bigsqcup_v \{1,\dots,\bfm_v\} \rightarrow \{1,\dots,m\}$ (equivalently, one could say that $\{1,\dots,m\}$ is divided into subsets of order $\bfm_v$, each of which we then identify with $\{1,\dots,\bfm_v\}$ in the unique order-preserving way). These maps should be such that the $\bfp_v(f)$-th edge adjacent to the vertex $v$ either is the $p(\bfr_v(f))$-th leaf of the tree, or lies on the path connecting $v$ to that leaf.
}
\end{equation}
Then,
\begin{equation} \label{eq:broken-popsicle}
\mybox{
A broken popsicle of type $(T,\bfp,\bfr)$, consists of: for each vertex $v$, a popsicle $(\bfS_v,\bfsigma_v)$ of type $(|v|, \bfp_v)$. 
} 
\end{equation}
For the unique single-vertex tree, this reduces to an ordinary popsicle of type $(d+1,p)$. We spell out the next simplest case:
\begin{equation} \label{eq:two-vertex-tree}
\mybox{
Let $T$ be the tree with two vertices labeled by $v \in \{1,2\}$, where the convention is that the first vertex is closest to the root. Let's denote their valencies by $\bfd_v+1$, and say that the unique finite edge is the $i$-th edge adjacent to the first vertex, for $1 \leq i \leq \bfd_1$. The maps $(\bfp_1,\bfr_1)$ and $(\bfp_2,\bfr_2)$ associated to the vertices must satisfy
\[
p(f) 
\begin{cases}
= \bfp_1(f_1) & \text{if } f = \bfr_1(f_1) \text{ and } \bfp_1(f_1) < i, \\
\in \{i,\dots,i+\bfd_2-1\} & \text{if } f = \bfr_1(f_1) \text{ and } \bfp_1(f_1) = i, \\
= \bfp_1(f_1) + \bfd_2 - 1 & \text{if } f = \bfr_1(f_1) \text{ and }\bfp_1(f_1) > i, \\
= \bfp_2(f_2) + i - 1 & \text{if } f = \bfr_2(f_2).
\end{cases}
\]
}
\end{equation}

Suppose that we have a broken popsicle, each component of which comes equipped with rational ends \eqref{eq:rational-ends}. Suppose additionally that we are given a subset $E$ of the finite edges of $T$, and for each $e \in E$, a gluing parameter $\bfgamma_e \in (0,1)$. The gluing process yields a broken popsicle associated to the tree where all edges of $E$ have been contracted. Restricting to the simplest nontrivial case for simplicity, this is done as follows:
\begin{equation} \label{eq:glue-popsicles}
\mybox{
Take the situation from \eqref{eq:two-vertex-tree}. Given a broken popsicle and parameter $\gamma$, one then glues together its two components $\bfS_1$ and $\bfS_2$ as in \eqref{eq:gluing-again}. The resulting glued surface, identified as before with part of the upper half plane, inherits this popsicle structure:
\[
\sigma_f = 
\begin{cases}
\bfsigma_{1,f_1} & \text{if $f = \bfr_1(f_1)$ and $\bfp_1(f_1) \neq i$}, \\ 
\bfsigma_{1,f_1} + \bfrho_{1,i}\, \gamma\, \bfzeta_{2,p(\bfr_1(f_1)) - i +1}  & \text{if $f = \bfr_{1}(f_1)$ and $\bfp_1(f_1) = i$}, \\
\bfzeta_{1,i} + \bfrho_{1,i} \gamma \bfsigma_{2,f_2} & \text{ if $f = \bfr_2(f_2)$.}
\end{cases}
\]
The first and last case are straightforward: those sprinkles on $\bfS_1$ or $\bfS_2$ carry over to $S$ through the gluing process for surfaces. In the middle case, we have moved the sprinkle $\bfsigma_{1,f_1}$ horizontally so that it will lie on the appropriate popsicle stick for $S$, determined by $\bfr_1(f_1)$: indeed, by \eqref{eq:two-vertex-tree} and \eqref{eq:gluing-again}, 
\[
\mathrm{re}(\bfsigma_{1,f_1} + \bfrho_{1,i}\, \gamma \, \bfzeta_{2,p(\bfr_1(f_1)) - i + 1}) = 
\bfzeta_{1,i} + \bfrho_{1,i\,} \gamma\, \bfzeta_{2,p(\bfr_1(f_1)) - i + 1} = 
\zeta_{p(\bfr_1(f_1))}.
\]
}
\end{equation}

\subsection{Moduli spaces}
For fixed $(d+1,p)$ satisfying the stability condition $d+|p| \geq 2$, there is a moduli space of popsicles, which is a smooth manifold $\scrR^{d+1,p}$ of dimension $d+|p|-2$. It carries an action of the group $\Aut(p)$ of those permutations of $\{1,\dots,m\}$ which preserve $p$. Over the moduli space, there is a universal family of popsicles,
\begin{equation} \label{eq:universal-family}
\scrS^{d+1,p} \longrightarrow \scrR^{d+1,p},
\end{equation}
and the sprinkles give $|p|$ smooth sections of \eqref{eq:universal-family}. The topology of these moduli spaces is easy to understand. For $d \geq 2$, we have a forgetful map to the moduli space of boundary-punctured discs,
\begin{equation}
\scrR^{d+1,p} \longrightarrow \scrR^{d+1},
\end{equation}
which is a fibre bundle with fibre $\bR^{|p|}$. In the remaining case, $\scrR^{2,p}$ can be directly identified with $\bR^{|p|}/\bR = \bR^{|p|-1}$. The outcome is that $\scrR^{d+1,p}$ is always diffeomorphic to $\bR^{d-2+|p|}$. 
\begin{figure}
\begin{picture}(0,0)%
\includegraphics{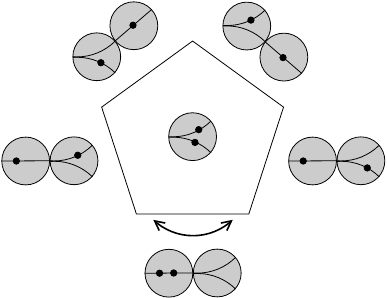}%
\end{picture}%
\setlength{\unitlength}{3355sp}%
\begingroup\makeatletter\ifx\SetFigFont\undefined%
\gdef\SetFigFont#1#2#3#4#5{%
  \reset@font\fontsize{#1}{#2pt}%
  \fontfamily{#3}\fontseries{#4}\fontshape{#5}%
  \selectfont}%
\fi\endgroup%
\begin{picture}(3629,2798)(739,-1879)
\put(3437,441){\makebox(0,0)[lb]{\smash{{\SetFigFont{7}{8.4}{\rmdefault}{\mddefault}{\updefault}{\color[rgb]{0,0,0}1}%
}}}}
\put(4070,-750){\makebox(0,0)[lb]{\smash{{\SetFigFont{7}{8.4}{\rmdefault}{\mddefault}{\updefault}{\color[rgb]{0,0,0}1}%
}}}}
\put(2324,-1567){\makebox(0,0)[lb]{\smash{{\SetFigFont{7}{8.4}{\rmdefault}{\mddefault}{\updefault}{\color[rgb]{0,0,0}2}%
}}}}
\put(2174,-1792){\makebox(0,0)[lb]{\smash{{\SetFigFont{7}{8.4}{\rmdefault}{\mddefault}{\updefault}{\color[rgb]{0,0,0}1}%
}}}}
\put(1381,-488){\makebox(0,0)[lb]{\smash{{\SetFigFont{7}{8.4}{\rmdefault}{\mddefault}{\updefault}{\color[rgb]{0,0,0}2}%
}}}}
\put(2476,-286){\makebox(0,0)[lb]{\smash{{\SetFigFont{7}{8.4}{\rmdefault}{\mddefault}{\updefault}{\color[rgb]{0,0,0}2}%
}}}}
\put(1909,757){\makebox(0,0)[lb]{\smash{{\SetFigFont{7}{8.4}{\rmdefault}{\mddefault}{\updefault}{\color[rgb]{0,0,0}2}%
}}}}
\put(2459,-511){\makebox(0,0)[lb]{\smash{{\SetFigFont{7}{8.4}{\rmdefault}{\mddefault}{\updefault}{\color[rgb]{0,0,0}1}%
}}}}
\put(851,-740){\makebox(0,0)[lb]{\smash{{\SetFigFont{7}{8.4}{\rmdefault}{\mddefault}{\updefault}{\color[rgb]{0,0,0}1}%
}}}}
\put(1551,231){\makebox(0,0)[lb]{\smash{{\SetFigFont{7}{8.4}{\rmdefault}{\mddefault}{\updefault}{\color[rgb]{0,0,0}1}%
}}}}
\put(2980,769){\makebox(0,0)[lb]{\smash{{\SetFigFont{7}{8.4}{\rmdefault}{\mddefault}{\updefault}{\color[rgb]{0,0,0}2}%
}}}}
\put(3555,-753){\makebox(0,0)[lb]{\smash{{\SetFigFont{7}{8.4}{\rmdefault}{\mddefault}{\updefault}{\color[rgb]{0,0,0}2}%
}}}}
\end{picture}%
\caption{\label{fig:pentagon}The moduli space $\bar\scrR^{d,p}$ for $d = 2$ and $p: \{1,2\} \rightarrow \{1,2\}$ the identity map. Note the $\bZ/2$-symmetry on the bottom boundary face (exchanging the two points), which is an instance of the phenomenon described in \eqref{eq:extra-symmetry}.}
\end{figure}

To compactify our moduli space, we allow broken popsicles, subject to the stability condition $|v| + |\bfp_v| \geq 3$ at each vertex. The resulting space is a compact manifold with corners $\bar\scrR^{d+1,p}$. The strata of $\bar\scrR^{d+1,p}$ are indexed by combinatorial data \eqref{eq:broken-popsicle-data}, and each such stratum is a product
\begin{equation} \label{eq:boundary-stratum}
\scrR^{T,\bfp,\bfr} \iso \prod_v \scrR^{|v|,\bfp_v},
\end{equation}
with the single-vertex tree $T$ yielding the interior $\scrR^{d+1,p}$. The gluing process \eqref{eq:glue-popsicles} provides $C^\infty$ charts: the coordinates transverse to each stratum are the gluing parameters $\bfgamma_e \geq 0$, where setting $\bfgamma_e = 0$ means not gluing along that edge. As before, there is a universal family $\bar\scrS^{d+1,p} \rightarrow \bar\scrR^{d+1,p}$, whose fibres are $\bigsqcup_v \bfS_v$, and which comes with $|p|$ canonical sections. There is also a well-behaved fibrewise compactification
\begin{equation} \label{eq:cornerbar-s}
\cornerbar{\scrS}^{d+1,p} \longrightarrow \bar\scrR^{d+1,p}.
\end{equation}
Here, $\cornerbar{\scrS}^{d+1,p}$ is a smooth manifold with corners; the fibres over points of $\scrR^{d+1,p}$ are compactifications $|S$ of popsicles, as in \eqref{eq:bar-compactification}; and the fibres over a boundary point of type $(T,\bfp,\bfr)$ consist of the union of $|\bfS_v$, with the intervals at infinity identified pairwise if they correspond to the same edge. The compactified gluing process \eqref{eq:gluing-2} provides a model for what \eqref{eq:cornerbar-s} looks like in transverse direction to a boundary face. We point out two features of $\bar\scrR^{d+1,p}$:
\begin{equation} \label{eq:extra-symmetry}
\mybox{
Each stratum \eqref{eq:boundary-stratum} carries an action of $\Aut(\bfp) = \prod_v \Aut(\bfp_v)$, and this group can be larger for more degenerate strata. Figure \ref{fig:pentagon} shows an example where the interior has trivial automorphism group, but a boundary stratum has a $\bZ/2$-symmetry.
}
\end{equation}
\begin{equation} \label{eq:shared}
\mybox{
The same product \eqref{eq:boundary-stratum} can appear as boundary stratum for different spaces.
Figure \ref{fig:different-strata} shows one such example.
}
\end{equation}
\begin{figure}
\begin{picture}(0,0)%
\includegraphics{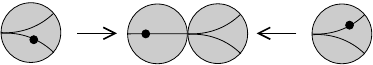}%
\end{picture}%
\setlength{\unitlength}{3355sp}%
\begingroup\makeatletter\ifx\SetFigFont\undefined%
\gdef\SetFigFont#1#2#3#4#5{%
  \reset@font\fontsize{#1}{#2pt}%
  \fontfamily{#3}\fontseries{#4}\fontshape{#5}%
  \selectfont}%
\fi\endgroup%
\begin{picture}(3506,590)(-200,-955)
\end{picture}%
\caption{\label{fig:different-strata}A ``shared boundary point'' of two moduli spaces $\bar\scrR^{d+1,p}$, see \eqref{eq:shared}. On the left, $p: \{1\} \rightarrow \{1,2\}$ has image $1$, and on the right it has image $2$. The broken popsicle appearing as the limit has the same structure in both cases.
}
\end{figure}%

\subsection{Complexification}
Following \cite[Section 6]{abouzaid-seidel07}, we sketch how basic properties of compactified popsicle moduli spaces follow from those of their complex counterparts. Write $\bC^\bullet = \bC \cup \{\infty\} \iso \bC P^1$. We consider genus zero curves $C \iso \bC^\bullet$ with $(d+1)$ distinct marked points $(\zeta_0,\dots,\zeta_d)$. Generally speaking, a set of ends for $C$ would be a collection of holomorphic embeddings of the closed disc into $C$, centered on the marked points.
\begin{equation} \label{eq:rational-ends-2}
\mybox{
A set of ends $(\epsilon_0,\dots,\epsilon_d)$ is called rational if they extend to isomorphisms $\bC^\bullet \rightarrow C$, with the following additional property: if we identify $C = \bC^\bullet$ so that $\zeta_0 = \infty$ and $\epsilon_0(z) = z^{-1}$, then $\epsilon_k(z) = \zeta_k + \rho_k z$ for all $k>0$, with $\zeta_k \in \bC$, $\rho_k \in \bC^*$.
}
\end{equation}
This is analogous to \eqref{eq:rational-ends}. The gluing process for curves with rational ends then has the same form as in \eqref{eq:gluing-again}, with gluing parameter $\gamma \in \bC^*$.

For $(d+1,p)$ as before, consider $X = (\bC^\bullet)^{|p|}$, and the submanifolds $F_0,\dots,F_d \subset X$ given by 
\begin{equation}
\left\{
\begin{aligned}
& F_0 = \{(\infty,\dots,\infty)\}, \\
& F_k = \{ x_f = 0 \; \text{for all $f$ such that $p(f) = k\}$,} \;\; k>0. 
\end{aligned}
\right.
\end{equation}
We then define:
\begin{equation} \label{eq:complex-popsicle}
\mybox{
A complex popsicle is a $C$ as before, together with a map $\phi: C \rightarrow X$ of degree $(1,\dots,1)$, satisfying the incidence condition $\phi(\zeta_k) \in F_k$. If we identify $C \iso \bC^\bullet$ so that $\zeta_0 = \infty$, the components of the map have the form
\[
\phi_f(z) = c_f (z-\zeta_{p(f)}) \;\; \text{ for } c_f \in \bC^*.
\]
The parameter $c_f$, and hence the map $\phi_f$, is uniquely determined by the point $\sigma_f = \phi_f^{-1}(i) \in \bC \setminus \{\zeta_{p(f)}\}$ (here, $i = \sqrt{-1}$ is the complex number, not an integer indexing something).
}
\end{equation}

Assuming $d+|p| \geq 2$, we have a smooth algebro-geometric moduli space $\scrM^{d+1,p}$ of complex popsicles, and its stable map compactification $\bar\scrM^{d+1,p}$, which is again smooth \cite[Lemma 6.1]{abouzaid-seidel07}. Points of the compactified space are represented by stable broken complex popsicles. Their combinatorial structure is indexed by data as in \eqref{eq:broken-popsicle-data}, except for the modifications required by the fact that the trees $T$ no longer have a ribbon structure. To construct charts near a point of the compactification, one equips the components with rational ends \eqref{eq:rational-ends-2}, and then glues them together. Let's consider the simplest case:
\begin{equation} \label{eq:glue-popsicles-2}
\mybox{
Take a tree as in \eqref{eq:two-vertex-tree}, and a corresponding complex popsicle, with rational ends on each component. Spelled out in term of an identification of the components with $\bC^\bullet$ as in \eqref{eq:rational-ends-2}, this data consists of, for $v = 1,2$: points $\bfzeta_{v,1},\dots,\bfzeta_{v,\bfd_v} \in \bC$, with ends parametrized by $\bfrho_{v,1},\dots,\bfrho_{v,\bfd_v} \in \bC^*$; and the maps $\bfphi_{v,1},\dots,\bfphi_{v,\bfm_v}$. After gluing the components for some small value of the parameter $\gamma$, we get maps
\[
\phi_f(z) = 
\begin{cases}
\bfphi_{1,f_1}(z) & \text{if $f = \bfr_1(f_1)$ and $\bfp_1(f_1) \neq i$}, \\ 
\bfphi_{1,f_1}(z - \bfrho_{1,i}\, \gamma\, \bfzeta_{2,p(\bfr_1(f_1))-i+1}) 
& \text{if $f = \bfr_1(f_1)$ and $\bfp_1(f_1) = i$}, \\
\bfphi_{2,f_2}(\bfrho_{1,i}^{-1} \gamma^{-1} (z-\bfzeta_{1,i})) & \text{if $f = \bfr_2(f_2)$.}
\end{cases}
\]
If one thinks in terms of points $\sigma_f$ as in \eqref{eq:complex-popsicle}, this formula turns into that from \eqref{eq:glue-popsicles}.
}
\end{equation}
The important thing for us is that \eqref{eq:glue-popsicles-2} defines a family of complex popsicles with rational dependence on $\gamma$, hence a holomorphic map from a punctured disc $D^*$ to $\scrM^{d+1,p}$. It is easy to see that this extends continuously, and hence holomorphically, to a map $D \rightarrow \bar\scrM^{d+1,p}$, taking $0$ to our original broken popsicle. An analysis of the tangent space to $\bar\scrM^{d+1,p}$ shows that this map is a local transverse slice to the boundary divisor. The same applies to more complicated gluing processes, which yield local holomorphic coordinate charts on $\bar\scrM^{d+1,p}$. To be precise, when gluing is applied to a family of broken popsicles, holomorphicity of the resulting coordinate chart assumes that the ends vary holomorphically in our family. If one allows ends that vary $\smooth$ in the original family, one still gets $\smooth$ charts, for the reason mentioned in \cite[Addendum 6.3]{abouzaid-seidel07}.

The space $\scrM^{d+1,p}$ admits an antiholomorphic involution, where one reverses the complex structure on $C$, and takes the complex conjugate of the maps $\phi_f$. If one thinks of $C = \bC^\bullet$ as in \eqref{eq:complex-popsicle}, this is equivalent to 
\begin{equation}
\left\{ 
\begin{aligned}
& (\zeta_1,\dots,\zeta_d) \longmapsto (\bar\zeta_1,\dots,\bar\zeta_d), \\
& \phi_f(z) \longmapsto \overline{\phi_f(\bar{z})} = \bar{c}_f (z - \bar{\zeta}_{p(f)}).
\end{aligned}
\right.
\end{equation}
The fixed locus has the $\zeta_k$ on the real line, and the $c_f$ nonzero real numbers. Assume that $\zeta_1 < \cdots < \zeta_d$ and $c_f > 0$. Then $\sigma_f = \phi_f^{-1}(i) = \zeta_{p(f)} + i/c_f$ lies on the popsicle stick $Q_{p(f)}$. In that way, $\scrR^{d+1,p}$ is embedded into $(\scrM^{d+1,p})^{\bR}$ as an open subset. The same idea identifies $\bar\scrR^{d+1,p}$ with a subset of $(\bar\scrM^{d+1,p})^{\bR}$, and an analysis of the gluing parameters shows this to be a (codimension zero) submanifold with corners, cut out locally by real analytic inequalities. From that picture, we derive the topology and differentiable structure on $\bar\scrR^{d+1,p}$ used throughout the paper.
\begin{figure}
\begin{centering}
\begin{picture}(0,0)%
\includegraphics{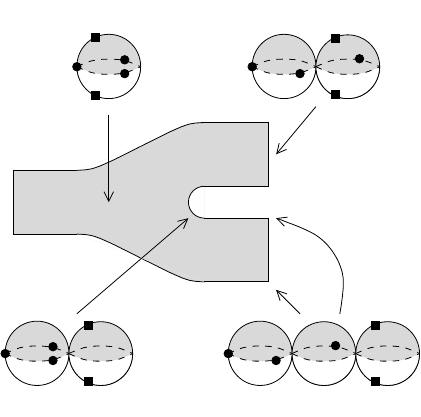}%
\end{picture}%
\setlength{\unitlength}{3355sp}%
\begingroup\makeatletter\ifx\SetFigFont\undefined%
\gdef\SetFigFont#1#2#3#4#5{%
  \reset@font\fontsize{#1}{#2pt}%
  \fontfamily{#3}\fontseries{#4}\fontshape{#5}%
  \selectfont}%
\fi\endgroup%
\begin{picture}(3960,3903)(178,-3730)
\put(2851,-3661){\makebox(0,0)[lb]{\smash{{\SetFigFont{10}{12.0}{\rmdefault}{\mddefault}{\updefault}{\color[rgb]{0,0,0}corner}%
}}}}
\put(451,-3661){\makebox(0,0)[lb]{\smash{{\SetFigFont{10}{12.0}{\rmdefault}{\mddefault}{\updefault}{\color[rgb]{0,0,0}boundary}%
}}}}
\put(901, 14){\makebox(0,0)[lb]{\smash{{\SetFigFont{10}{12.0}{\rmdefault}{\mddefault}{\updefault}{\color[rgb]{0,0,0}interior}%
}}}}
\put(2401, 14){\makebox(0,0)[lb]{\smash{{\SetFigFont{10}{12.0}{\rmdefault}{\mddefault}{\updefault}{\color[rgb]{0,0,0}interval at infinity}%
}}}}
\end{picture}%
\caption{\label{fig:complexify}A compactified surface $|S$, with examples of the stable curves associated to points of $|S$. Here, the round dots are the $\zeta_k$, and the square dots the extra marked points $z_{\pm}$. On the bottom right, note how the endpoints of an interval in $(|S) \setminus S$ have isomorphic associated stable curves.}
\end{centering}
\end{figure}%

Now consider a space $\bar\scrN^{d+1,p}$ which consists of stable maps with the same conditions as before, but where the domain carries two additional marked points $z_+, z_-$. This is a subvariety of the space of stable maps of genus zero with $(d+3)$ marked points, and is again smooth. Take the real locus $(\bar\scrN^{d+1,p})^{\bR}$, where the involution reverses the complex structure as before, and simultaneously exchanges the extra points $z_{\pm}$. There is a continuous map 
\begin{equation} \label{eq:add-plusminus}
\cornerbar{\scrS}^{d+1,p} \longrightarrow (\bar{\scrN}^{d+1,p})^{\bR}.
\end{equation}
On the interior of $\scrS^{d+1,p}$, where we are considering a popsicle $S$ and a point $z$ in its interior, this consists of complexifying the popsicle and then adding $z_+ = z$ and well as its complex conjugate $z_-$ as extra marked points. If $z$ is a boundary point of the popsicle, its image under \eqref{eq:add-plusminus} adds an extra component which will contain $z_{\pm}$; that component is a sphere with three distinct marked points, hence unique up to isomorphism. We can extend that to the intervals at infinity in $|S$, again by adding extra components to the complexified curve. This time, the position of a point $t \in (0,1) \subset (|S) \setminus S$ will determine the position of $z_{\pm}$ on the extra component, which is a sphere with four marked points; one can think of that sphere as the union of $Z^\bullet = \{\pm \infty\} \times (\bR \times [0,1])$ and its complex conjugate, and then $z_+ = (0,t)$. There is a minor hitch: the endpoints of each interval $[0,1] \subset (|S) \setminus S$ give rise to the same complexified curve, so \eqref{eq:add-plusminus} is not injective (see Figure \ref{fig:complexify}). Nevertheless, one can use that map to obtain smooth (and in fact real-analytic) coordinates on $\cornerbar{\scrS}^{d+1,p}$, which make it into a manifold with corners. To explain that,  look at a popsicle $|S$ (not broken). Its image under \eqref{eq:add-plusminus} lies inside a real surface, which is a fibre of $(\bar{\scrN}^{d+1,p})^{\bR}$ under the map forgetting $z_{\pm}$. The image of \eqref{eq:add-plusminus} inside that surface has points where it looks like $\{s_1s_2 \geq 0\}$ in local coordinates $(s_1,s_2) \in \bR^2$. What happens is that there are two corners of $|S$ mapped to the point $\{s_1 = s_2 = 0\}$, and the parts of $|S$ near those corners correspond to the two quadrants $\{s_1,s_2 \geq 0\}$, $\{s_1,s_2 \leq 0\}$. Hence, $(s_1,s_2)$ still give smooth coordinates on $|S$ near the corners (Figure \ref{fig:complexify-2}). A similar, but more involved, observation applies to broken popsicles.
\begin{figure}
\begin{centering}
\includegraphics{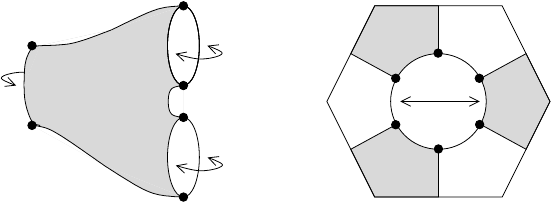}
\caption{\label{fig:complexify-2}The image of the surface $|S$ from Figure \ref{fig:complexify} under \eqref{eq:add-plusminus}. In this simple case, the target space of \eqref{eq:add-plusminus} is a fixed locus of a suitable real involution of Deligne-Mumford space $\bar{\scrM}_{0,5}$, and concretely is a non-orientable surface with Euler characteristic $-1$. We give two representations: as a pair-of-pants with each boundary component closed up to a Moebius band; and as a torus (represented as a hexagon with opposite sides identified by translation), again with a hole closed up to a Moebius band.
The image of $|S$ is the shaded region (on the left, the front half of the pair-of-pants); the dots are the images of the corners of $|S$; and the arrows indicate identifying opposite points on each of the boundary circles.}
\end{centering}
\end{figure}%

\subsection{Orientations\label{subsec:orientations}}
Consider the moduli space $\scrH^{d+1,p}$ of popsicles with a fixed identification $S^\bullet = H^\bullet$, which takes $\zeta_0$ to $i\infty$. By construction,
\begin{equation} \label{eq:g-fibration}
\scrR^{d+1,p} = \scrH^{d+1,p}/G,
\end{equation}
where $G$ is the group of automorphisms of $H^\bullet$ preserving $i\infty$. Identify $\scrH^{d+1,p}$ with an open subset of $\bR^{d+|p|}$ by using the coordinates $\zeta_1,\dots,\zeta_d,-s_1,\dots,-s_{|p|}$, in the notation from \eqref{eq:s-values}. This determines an orientation of that space. We orient the Lie algebra $\frakg$ by taking a basis consisting first of the infinitesimal translation to the right, and then of infinitesimal radial expansion. There is a unique orientation of $\scrR^{d+1,p}$ which is compatible with the choices of orientations and \eqref{eq:g-fibration} (here, the ordering convention is that we choose a splitting $T\scrH^{d+1,p} \iso \frakg \oplus T\scrR^{d+1,p}$ of tangent spaces).

Take the special case of boundary-punctured discs without sprinkles. One can identify $\scrR^{d+1}$ with the slice of the $G$-action on $\scrH^{d+1} = \{\zeta_1 < \cdots < \zeta_d\} \subset \bR^d$ obtained by keeping $\zeta_1$ and $\zeta_2$ fixed. On the infinitesimal level, for a point on that slice, we have a splitting of the short exact sequence
\begin{equation}
\xymatrix{ 
0 \ar[r] & \frakg \ar[rr] && (T\scrH^{d+1})_S \ar[rr] && (T\scrR^{d+1})_S
\ar@/_1.5pc/@{-->}[ll]_{ } \ar[r] & 0.
}
\end{equation}
If we take the previous basis of $\frakg$, followed by the basis of the tangent space to the slice obtained by moving $\zeta_3,\dots,\zeta_d$ to the right, then the outcome yields the same orientation of $(T\scrH^{d+1})_S$ as our previous choice. This shows that the orientation of $\scrR^{d+1}$ defined here is the same as that in \cite[Section 12g]{seidel04}; similarly, the orientation of general spaces $\scrR^{d+1,p}$ agrees with that in \cite[Section 9d]{abouzaid-seidel07}. The effect of these orientation conventions on boundary faces in $\bar\scrR^{d+1,p}$ is expressed by \cite[Equation (9.15)]{abouzaid-seidel07}, which we reproduce here in view of its importance later on:

\begin{lemma} \label{th:boundary-orientations}
Take a codimension one boundary face of the compactified moduli space of popsicles, determined by data as in \eqref{eq:two-vertex-tree}, and which is therefore of the form
\begin{equation} \label{eq:product-face}
\scrR^{\bfd_1+1,\bfp_1} \times \scrR^{\bfd_2+1,\bfp_2} \subset \partial \bar{\scrR}^{d+1,p}.
\end{equation}
With our conventions, the product orientation of that space differs from the induced boundary orientation by a sign $(-1)^\ddag$, where
\begin{equation} \label{eq:boundary-signs}
\begin{aligned}
\ddag = & \bfd_1\bfd_2 + i\bfd_2 + i - 1 \\
& + 
|\bfp_1|\, \bfd_2 \\
& + \# \{ 1 \leq f_1 \leq |\bfp_1|, \, 1 \leq f_2 \leq |\bfp_2| \;:\; \bfr_1(f_1) > \bfr_2(f_2) \}.
\end{aligned}
\end{equation}
\end{lemma}

The first line in \eqref{eq:boundary-signs} comes from the moduli spaces of boundary-punctured discs \cite[Equation (12.22)]{seidel04}. The second line is a Koszul sign: to compare the product orientation with the boundary orientation, the tangent directions which vary the sprinkles belonging to the first vertex have to be moved past those that vary the $(\bfd_2+1)$-punctured disc. The last line is obtained from reordering the sprinkles. Strictly speaking, this recourse to \cite{seidel04} only applies to the situation where $\bfd_1, \bfd_2 \geq 2$, but it is not hard to extend the argument to cover the remaining cases as well.
%

\subsection{Weighted popsicles\label{subsec:weights}}
This is a variant of the corresponding notion in \cite[Section 2g]{abouzaid-seidel07}. Take $(d+1,p)$ as before, and in addition, nonpositive integers $w = (w_0,\dots,w_d)$, satisfying the following conditions:
\begin{align}
& \label{eq:weights-and-sprinkles}
w_0 - w_1 - \cdots - w_d = |p|, \\
& \label{eq:upper-bound}
\# p^{-1}( i) \leq -w_i \;\; \text{ for all $i = 1,\dots,d$.}
\end{align}
We think of the $w_j$ as weights associated to the points $\zeta_j$. We can add weights to the previous notation, without changing the moduli spaces (so we have $\scrS^{d+1,p,w} \rightarrow \scrR^{d+1,p,w}$ and its extension $\bar\scrS^{d+1,p,w} \rightarrow \bar\scrR^{d+1,p,w}$, as well as fibrewise compactification $\cornerbar{\scrS}^{d+1,p,w}$). What's worthwhile mentioning is how one should think of a broken popsicle carrying weights. Namely, take $(d+1,p,w)$ and a broken popsicle of type $(T,\bfp,\bfr)$. Its components inherit weights $\bfw_v = (\bfw_{v,0},\dots,\bfw_{v,|v|-1})$ in the unique way which satisfies the following conditions: the weights at the semi-infinite edges of the tree are the original $(w_0,\dots,w_d)$; the weights on the two ends of each finite edge agree; and the analogue of \eqref{eq:weights-and-sprinkles} at each vertex holds. Slightly less obviously, the analogue of \eqref{eq:upper-bound} will be satisfied at each vertex. To check that, it suffices to consider broken popsicles with two components, since the general case follows by induction. In the notation from \eqref{eq:two-vertex-tree}, where the vertices are $v \in \{1,2\}$, the only potentially nontrivial issue is the size of $\bfp_1^{-1}(i)$. But clearly,
\begin{equation} \label{eq:check-inequality}
\begin{aligned}
\# \bfp_1^{-1}(i) & =  \# p^{-1}(\{i,\dots,i+\bfd_2-1\}) - |\bfp_{2}| 
\\ & 
\leq - w_i - w_{i+1} - \cdots - w_{i+\bfd_2-1} - \bfw_{2,0} + \bfw_{2,1} + \cdots + \bfw_{2,\bfd_2} 
\\ & 
= -\bfw_{2,0} = -\bfw_{1,i}.
\end{aligned}
\end{equation}
Our main concern will be with the case where 
\begin{equation} \label{eq:0-1}
w_0,\dots,w_d \in \{-1,0\}.
\end{equation}
Because of \eqref{eq:weights-and-sprinkles} and \eqref{eq:upper-bound}, this constraint works out as follows. If $w_j = 0$ for some $j >0$, then $p^{-1}(j) = \emptyset$. If $w_j = -1$ for some $j>0$, then $\# p^{-1}(j) \leq 1$; moreover, equality holds either all the time (in which case $w_0 = 0$) or with one exception (in which case $w_0 = -1$). Finally, because $p$ is injective, it is automatically increasing. The condition \eqref{eq:0-1} is not always inherited by the components of a broken popsicle. In codimension one, the cases where this fails reflect the phenomena encountered in Figures \ref{fig:pentagon} and \ref{fig:different-strata}. Namely, one has the following straightforward result:

\begin{lemma} \label{th:wrong-weights}
Take a space $\bar\scrR^{d+1,p,w}$ where the weights lie in $\{-1,0\}$. Consider a boundary stratum of codimension $1$, such that the induced weights for the components of broken popsicles in that stratum do not lie in $\{-1,0\}$. In the notation from \eqref{eq:two-vertex-tree} and \eqref{eq:check-inequality}, this happens when
\begin{equation}
\bfw_{1,i} = \bfw_{2,0} < -1.
\end{equation}
Then, one of the two following applies:
\begin{equation} \label{eq:more-symmetry}
\mybox{$\# \bfp_1^{-1}(i) \geq 2$, or equivalently, $\Aut(\bfp_1)$ is nontrivial.}
\end{equation}
\begin{equation} \label{eq:switch-sprinkle}
\mybox{$\bfw_{1,0} = w_0 =  -1$, $\bfw_{1,i} = \bfw_{2,0} = -2$, and $\bfp_{1}^{-1}(i) = \{f\}$ consists of one element. In that case, $p(\bfr_1(f))$ is one of exactly two values $k_1 < k_2$ in $\{i,\dots,i+\bfd_2-1\}$ for which $w_{k_j} = \bfw_{2,k_j-i+1} = -1$ and $\bfp_{2}^{-1}(k_j-i+1) = \emptyset$.
}
\end{equation}
\end{lemma}

Our next task is to consider weighted popsicles equipped with a suitable class of one-forms.
\begin{equation} \label{eq:restricted-sub-closed}
\mybox{
Let $S$ be a weighted popsicle, with rational ends $(\epsilon_0,\dots,\epsilon_d)$. A sub-closed one-form compatible with the ends is a $\beta_S \in \Omega^1(S)$ as in \eqref{eq:beta0}, \eqref{eq:sub-closed0}, whose behaviour over the ends is dictated by the weights as in \eqref{eq:restricted-beta}.
}
\end{equation}
Suppose that, for each stable weighted popsicle, we choose rational ends $(\epsilon_0,\dots,\epsilon_d)$, as well as a one-form \eqref{eq:restricted-sub-closed}. These choices should be invariant under $\Aut(p)$, and vary smoothly in the moduli space of popsicles. Let's call this a universal choice of ends and compatible one-forms. One can always find such choices, since for both ends and one-forms, the space of possibilities is contractible.
\begin{equation} \label{eq:consistent-ends-and-one-forms}
\mybox{
A universal choice of ends and compatible sub-closed one-forms is called consistent if the following holds. Consider a stable broken popsicle with weights. Each component $\bfS_v$ is itself a stable weighted popsicle, and we equip it with the ends and one-forms which are part of our universal choice. If we glue along a subset of the finite edges, as in \eqref{eq:glue-popsicles} or its generalizations, the outcome inherits ends and a one-form, and for sufficiently small values of the gluing parameters, these should agree with those given by the universal choice.
}
\end{equation}
This condition can be achieved by an inductive construction. Besides the previously mentioned contractibility property, the main ingredient is  the ``associativity'' of the gluing process \eqref{eq:glue-popsicles} (see e.g.\ \cite{qin11} for the terminology), which ensures that the constraints coming from different boundary strata of $\bar\scrR^{d+1,p,w}$ do not contradict each other. 

There is an alternative and less restrictive approach, where one works with the following class:
\begin{equation} \label{eq:sub-closed}
\mybox{
Let $S$ be a weighted popsicle. An asymptotically translation-invariant sub-closed one-form is a $\beta_S \in \Omega^1(S)$ as in \eqref{eq:beta0}, \eqref{eq:sub-closed0}, whose asymptotic behaviour is dictated by the weights as in \eqref{eq:restricted-beta}.
}
\end{equation}
Note that, unlike the previously considered setup \eqref{eq:restricted-sub-closed}, no ends need to be specified. One can choose such a one-form for each stable weighted popsicle, as before satisfying $\Aut(p)$-symmetry and smoothness with respect to moduli. The last-mentioned condition can be encoded by asking for a smooth fibrewise one-form on the space $|\scrS^{d+1,p,w}$, which is the fibrewise compactification of $\scrS^{d+1,p,w}$ in the sense of \eqref{eq:bar-compactification}. Let's call this a universal choice of asymptotically translation-invariant sub-closed one-forms. The same is true for the analogue of \eqref{eq:consistent-ends-and-one-forms}, which is:
\begin{equation} \label{eq:asymptotically-consistent-one-forms}
\mybox{
A universal choice of aymptotically translation-invariant sub-closed one-forms is called asymptotically consistent if the following holds. For every $(d+1,p,w)$, our choice can be extended to a smooth one-form on $\cornerbar{\scrS}^{d+1,p,w}$, which is closed in a neighbourhood of each interval at infinity. Moreover, if we take that one-form and restrict it to any component $\bfS_v$ of a broken popsicle, the outcome agrees with that prescribed by the universal choice for weighted popsicles of type $(|v|,\bfp_v,\bfw_v)$.
}
\end{equation}
To understand the meaning of this, it is helpful to think of the model gluing setup from \eqref{eq:gluing-again}. Let the surfaces $\bfS_v$ carry one-forms $\beta_{\bfS_v}$, with the same $w$ at the points at infinity being glued together. Explicitly, in coordinates on the ends, this means that for $|s| \gg 0$,
\begin{equation} \label{eq:beta-gluing-0}
\left\{
\begin{aligned}
& \beta_{\bfS_1} = w\, \mathit{dt} + d(e^{-\pi s_1} \bfg_1(e^{-\pi s_1},t)), \\
& \beta_{\bfS_2} = w\, \mathit{dt}  + d(e^{\pi s_2} \bfg_2(e^{\pi s_2},t)).
\end{aligned}
\right.
\end{equation}
Asymptotic consistency means that the one-form $\beta_{S_l}$ on the glued surfaces have the following shape on the neck region $[-l/2,l/2] \times [0,1]$, with first coordinate $s = s_1-l/2 = s_2+l/2$ as in \eqref{eq:length-decay}:
\begin{equation} \label{eq:beta-gluing}
\beta_{S_l} = \beta_{\bfS_1} + \beta_{\bfS_2} - w \,\mathit{dt} 
+ e^{-\pi l} d (g(e^{-\pi s- \pi l/2}, e^{\pi s- \pi l/2}, t)),
\end{equation}
where $g(\sigma_1,\sigma_2,t)$ vanishes for $t = 0,1$. As before, one can satisfy \eqref{eq:asymptotically-consistent-one-forms} by an inductive construction; this is a little more complicated because, on the gluing regions, partitions-of-unity arguments have to be applied to primitives of the one-forms involved, which means the functions $\bfg_v$, $g$ in \eqref{eq:beta-gluing-0}, \eqref{eq:beta-gluing}.

The weaker condition \eqref{eq:asymptotically-consistent-one-forms} is sufficient for our applications (and it also serves as a model for how we will address other consistency issues later on); however, readers who prefer a simple explicit description for how sub-closed one-forms behave under degenerations of popsicles may want to stick with the more classical approach \eqref{eq:consistent-ends-and-one-forms}.

\section{Fibrations with singularities\label{sec:first-construction}}
This section defines a noncommutative divisor associated to any exact symplectic fibration with singularities (we will assume that the first Chern class of the total space vanishes, but except for the obvious grading issue, that is unnecessary). This serves as a preparation for developments later in the paper, which will be formally analogous but geometrically a bit trickier. The pseudo-holomorphic map equations involved are of the same kind as in \cite{abouzaid-seidel07}, even though the algebraic structure we build on those foundations is different. 
%
%
%

\subsection{Executive summary\label{subsec:fantasy}}
Let's start by explaining what the definition would look like if one were allowed to wish away important technical points. The issues that are being swept under the carpet include: transversality, down to the basic fact that only having morphisms between transverse Lagrangian submanifolds does not constitute an $A_\infty$-category structure; energy and $C^0$-bounds for pseudoholomorphic maps; bubbling off of holomorphic discs or spheres; and issues of signs and grading in Lagrangian Floer cohomology. On the resulting formal level, the specific symplectic geometry situation becomes irrelevant; and the statements, if read literally, range from questionable to hair-raisingly wrong. Nevertheless, such an outline is hopefully useful in conveying the overall shape of the construction (besides, some notation introduced here will carry over to the later parts).

Take a symplectic manifold $M$, together with a Hamiltonian vector field $X$ and its flow $(\phi^t)$. Moreover, fix a compatible almost complex structure (which will be used in all Cauchy-Riemann equations that appear). Take $(L_0,L_1)$ and a nonpositive integer $w$, such that all chords \eqref{eq:chord} are nondegenerate. Then, the Floer cochain space is the free abelian group generated by them:
\begin{equation} \label{eq:floer-cochain}
\mathit{CF}^*(L_0,L_1;w) = \bigoplus_x \bZ x.
\end{equation}
The differential $d$ counts solutions of Floer's equation \eqref{eq:floer}, \eqref{eq:floer-limits}, which are isolated up to translation, with suitable signs:
\begin{equation} \label{eq:count-differential}
 dx_1 = \textstyle \sum_{x_0} \sum_u \pm x_0.
 \end{equation}
Let $\mathit{HF}^*(L_0,L_1;w)$ be its cohomology. In the special case $w = 0$, we usually omit $w$ from the notation. The main role in our story will be played by the Floer complexes for $w \in \{-1,0\}$ together with the continuation map between them, which is a chain map
\begin{equation} \label{eq:continuation-map}
\gamma: \mathit{CF}^*(L_0,L_1;-1) \longrightarrow \mathit{CF}^*(L_0,L_1).
\end{equation}
To define that, take the surface $S = Z$, equipped with a one-form \eqref{eq:sub-closed} which corresponds to setting $w_0 = 0$ at the end $s \rightarrow -\infty$, respectively $w_1 = -1$ as $s \rightarrow +\infty$. This gives rise to a special case of \eqref{eq:cr} (the continuation map equation).  One obtains $\gamma$ by counting solutions of that equation, in the same way as in \eqref{eq:count-differential}. Denote by $\mathit{HF}^*(L_0,L_1;\mathit{cont})$ the cohomology of the mapping cone of \eqref{eq:continuation-map}. By construction, this fits into a long exact sequence
\begin{equation} \label{eq:hf-cont}
\cdots \rightarrow \mathit{HF}^*(L_0,L_1;-1) \xrightarrow{H(\gamma)} \mathit{HF}^*(L_0,L_1) \longrightarrow \mathit{HF}^*(L_0,L_1;\mathit{cont}) \rightarrow \cdots 
\end{equation}
This is of interest only for non-closed symplectic manifolds and Lagrangian submanifolds: if $L_0$ or $L_1$ is closed, the continuation map is a quasi-isomorphism, hence $\mathit{HF}^*(L_0,L_1;\mathit{cont}) = 0$.

Our aim is to define an $A_\infty$-category $\scrF$ whose objects are Lagrangian submanifolds, and whose cohomology level morphisms are $\mathit{HF}^*(L_0,L_1;\mathit{cont})$. Assume that, for stable weighted popsicles, an asymptotically consistent choice of one-forms \eqref{eq:asymptotically-consistent-one-forms} has been made. Suppose that we have such a popsicle $(S,\sigma)$, of type $(d+1,p,w)$, and Lagrangian submanifolds $(L_0,\dots,L_d)$, such that the intersections $\phi^{w_0}(L_0) \cap L_d$, $\phi^{w_k}(L_{k-1}) \cap L_k$ are transverse. We then consider solutions $u$ of the associated equation \eqref{eq:cr}, with limits \eqref{eq:disc-limits}. Denote the space of triples $(S,\sigma,u)$ by
\begin{equation} \label{eq:r-maps}
\scrR^{d+1,p,w}(x_0,\dots,x_d) \longrightarrow \scrR^{d+1,p,w}.
\end{equation}
To make the notation more homogeneous, we include spaces of non-stationary solutions of Floer's equation mod translation as  $\scrR^{2,p,w}(x_0,x_1)$, with $|p| = 0$. Putting Floer trajectories on an entirely equal footing with \eqref{eq:r-maps} would require a proper notion of ``the $(-1)$-dimensional space $\scrR^2 = \mathit{point}/\bR$''. We skip that and, given that Floer trajectories are the most familiar part of our entire story, will mostly leave their treatment to the reader (more details are given in \cite{abouzaid-seidel07}). The spaces \eqref{eq:r-maps} have Gromov compactifications 
\begin{equation} \label{eq:r-maps-2}
\bar\scrR^{d+1,p,w}(x_0,\dots,x_d) \longrightarrow \bar{\scrR}^{d+1,p,w}.
\end{equation}
In the limit, weighted popsicles break into pieces. Those pieces are either themselves stable, or else are strips which carry Floer trajectories; the last-mentioned components are forgotten under the map in \eqref{eq:r-maps-2}. As usual, the codimension of each stratum in the compactification is the number of components, minus one.

Counting points in zero-dimensional spaces \eqref{eq:r-maps} yields maps
\begin{equation} \label{eq:mu-p}
\begin{aligned}
& \mu^{d,p,w}: \mathit{CF}^*(L_{d-1},L_d;w_d) \otimes \cdots \otimes \mathit{CF}^*(L_0,L_1;w_1)
\longrightarrow \mathit{CF}^{*+2-d-|p|}(L_0,L_d;w_0), \\
& \mu^{d,p,w}(x_d,\dots,x_1) = \textstyle \sum_{x_0} \sum_{(S,\sigma,u)} \pm x_0.
\end{aligned}
\end{equation}
Not all maps \eqref{eq:mu-p} are nontrivial:
\begin{equation} \label{eq:involution}
\mybox{
Suppose that $p$ is not injective. Then $\mu^{d,p,w} = 0$. Indeed, pick an element of $\Aut(p)$ which switches exactly two sprinkles. This acts on $\scrR^{d+1,p,w}$ as an orientation-reversing involution, with a codimension $1$ fixed point set. The action lifts to $\scrR^{d+1,p,w}(x_0,\dots,x_d)$, with the same properties. Generically, no isolated points in our moduli space lie over the fixed point set, and all the other isolated points appear in pairs exchanged by the involution, whose contributions cancel. 
}
\end{equation}
We set 
\begin{equation} \label{eq:cone-floer-cochains}
\scrF(L_0,L_1) = \mathit{CF}^*(L_0,L_1) \oplus \mathit{CF}^{*+1}(L_0,L_1;-1),
\end{equation}
and define the $A_\infty$-operation $\mu^d_{\scrF}$ to be the sum over all $\mu^{p,d,w}$ such that the weights lie in $\{-1.0\}$. The $A_\infty$-associativity equations \eqref{eq:associativity} for $\scrF$ are, as always, derived from a study of boundary points of one-dimensional spaces $\bar\scrR^{d+1,p,w}(x_0,\dots,x_d)$. Besides splitting off of Floer trajectories on the ends, this means looking at the codimension one strata of $\bar\scrR^{d+1,p,w}$, which correspond to broken stable popsicles with two components. In principle, these components can carry weights $<-1$, which puts us potentially in danger of seeing boundary contributions which are not part of the $A_\infty$-structure. However, based on Lemma \ref{th:wrong-weights}, one can see that those contributions turn out to be zero:
\begin{equation} \label{eq:cancel-1}
\mybox{
In the case of \eqref{eq:more-symmetry}, the analysis from \eqref{eq:involution} applies to the first component. The cancellation observed there carries over to the pair consisting of both components, and from there to the counting of such boundary points in $\bar\scrR^{d+1,p,w}(x_0,\dots,x_d)$.
}
\end{equation}
\begin{equation} \label{eq:cancel-2}
\mybox{
In the situation of \eqref{eq:switch-sprinkle}, we get one boundary point of each of two moduli spaces $\bar\scrR^{d+1,p,w}(x_0,\dots,x_d)$, which differ only in the map $p$. More precisely, what switches is the unique $k \in \{i,\dots,i+\bfd_2-1\}$ such that $k-i+1$ does not lie in the image of $\bfp_2$, but $k$ lies in the image of $p$; in the notation from \eqref{eq:switch-sprinkle}, $k = p(\bfr_1(f))$, and the two possibilities are $k = k_1$ or $k_2$. The definition of $\mu^d_{\scrF}$ adds up the two resulting $\mu^{d,p,w}$, and the boundary contributions again cancel.
}
\end{equation}
To round off the discussion, let's point out that $\mu^1_{\scrF}$ consists of the Floer differentials on each summand, plus an off-diagonal entry, which comes from ($d = 1$, $|p| = 1$, $w_0 = 0$, $w_1 = -1$). Inspection of the definition shows that this entry is the continuation map \eqref{eq:continuation-map}.

\subsection{Target space geometry\label{subsec:target}}
We now embark on the task of making the construction rigorous, in a specific geometric context. Let's recall a version of the well-known definition of an exact symplectic fibration with singularities. Take an exact symplectic manifold with boundary $(M^{2n}, \omega_M = d\theta_M)$, together with a proper map 
\begin{equation} \label{eq:pi}
\Pi: M \longrightarrow \bC.
\end{equation}
Equip the base with its standard symplectic form $\omega_{\bC}$.
\begin{equation}
\mybox{
We say that $\Pi$ is symplectically trivial at $x \in M$ if the following holds. $x$ is a regular point; $\mathit{TM}_x^v = \mathit{ker}(D\Pi_x) \subset \mathit{TM}_x$ is a symplectic subspace; and the restriction of $\omega_M$ to its symplectic orthogonal complement $\mathit{TM}_x^h = (\mathit{TM}_x^v)^\perp$ is equal to the pullback of $\omega_{\bC}$. If $x \in \partial M$, we additionally require that $\mathit{TM}_x^h \subset T(\partial M)_x$.
}
\end{equation}
With that at hand, we can formulate the following:
\begin{equation} \label{eq:symplectic-fibration}
\mybox{
An exact symplectic fibration with singularities is a map \eqref{eq:pi}, such that: $\Pi$ is symplectically trivial at every point which satisfies $|\Pi(x)| \geq 1$, and also in a neighbourhood of $\partial M$.
}
\end{equation}
This implies that the restriction of $\Pi$ to $\{|w| \geq 1\} \subset \bC$ is a locally trivial exact symplectic fibration: in particular, all those fibres $M_w$ are isomorphic exact symplectic manifolds. A similar remark applies to the geometry near $\partial M$: if we take $F = M_1$ as a reference fibre, parallel transport yields a symplectic isomorphism
\begin{equation} \label{eq:partial-trivialization}
\xymatrix{
(\text{neighbourhood of } \partial M) \ar[rr]^-{\iso} \ar@{^{(}->}[d] 
&& 
\bC \times (\text{neighbourhood of } \partial F) \ar@{^{(}->}[d]
\\
M \ar[r]_-{\Pi}
&\bC & 
\bC \times F. \ar[l]^-{\;\text{projection}} 
}
\end{equation}
We also impose the following convexity condition:
\begin{equation} \label{eq:convexity}
\mybox{
A neighbourhood of the boundary in $F = M_1$ should carry a preferred compatible almost complex structure $I_F$, such that $\partial F$ is weakly $I_F$-convex (has nonnegative Levi form).}
\end{equation}
On a neighbourhood of the boundary in $M$, this determines a unique almost complex structure $I_M$ which, under \eqref{eq:partial-trivialization}, corresponds to the product of the complex structure on $\bC$ and $I_F$. Finally, we require:
\begin{equation} \label{eq:complex-volume-form}
\mybox{$c_1(M) = 0$, and we choose a trivialization of its canonical bundle.}
\end{equation}

All the other geometric objects are chosen to be compatible with the specific nature of our symplectic manifolds:
\begin{equation}
\label{eq:hamiltonian}
\mybox{
We fix some $H \in \smooth(M,\bR)$ with the following properties. Outside the preimage of the open unit disc, it equals $\pi  |\Pi|^2$. Moreover, in a neighbourhood of $\partial M$, it is the pullback of a function on $\bC$. The associated vector field $X$ then satisfies $D\Pi(X_x) = 2\pi i w\partial_w$ at all points with $|\Pi(x)| \geq 1$. Similarly, near $\partial M$, $X$ is a vector field on the $\bC$ factor, extended trivially in $F$-direction, with respect to \eqref{eq:partial-trivialization}. Finally, we make a technical assumption, namely, that the zero-set $X^{-1}(0)$ has no interior points (its complement is open and dense).
}
\end{equation}
Clearly, $X$ has a well-defined flow $(\phi^t)$.
\begin{equation}
\label{eq:almost-complex}
\mybox{
We use compatible almost complex structures $J$ on $M$ with the following properties. First, $D\pi_x$ is $J$-holomorphic all points with $|\Pi(x)| \geq 1$. Moreover, there is a compatible complex structure $I_{\bC}$ on the complex plane such that, with respect to \eqref{eq:partial-trivialization}, $J_M$ agrees with $I_{\bC} \times I_F$ to first order; this means that $J_M - (I_{\bC} \times F)$ is zero on $\partial M$, and the derivative in normal direction to $\partial M$ is also zero.
}
\end{equation}
This class of almost complex structures is preserved by $(\phi^t)$.
\begin{equation}
\label{eq:lagrangian}
\mybox{
We consider connected exact Lagrangian submanifolds $L \subset M$, disjoint from $\partial M$. Each such submanifold must come with an ``angle'' $\alpha_L \in (-\half,\half)$, such that $\Pi(L) \subset \bC$ is contained in the union of the open unit disc and the ray $e^{2\pi i \alpha_L} \bR^+$. It should also come with a grading with respect to \eqref{eq:complex-volume-form}, and hence an orientation, as well as a {\em Spin} structure.
}
\end{equation}
Less formally, Lagrangian submanifolds are allowed to go infinity (with ends modelled on a radial half-ray times some closed Lagrangian submanifold in a fibre) in any direction except the forbidden one $\bR^- \subset \bC$. Hence, this class of submanifolds is preserved by $(\phi^t)$ only as long as one doesn't hit the forbidden direction. 

This completes the definition of our geometric setup. For the rest of the section, we  work with a fixed $M$ satisfying \eqref{eq:symplectic-fibration}, \eqref{eq:convexity}, \eqref{eq:complex-volume-form}; a fixed Hamiltonian \eqref{eq:hamiltonian} on it; and any almost complex structures or Lagrangian submanifolds will belong to the classes \eqref{eq:almost-complex}, \eqref{eq:lagrangian}. 

\subsection{Controlling solutions\label{subsec:control}}
The simplest use of the various geometric assumptions is to prevent solutions of Cauchy-Riemann equations from escaping to infinity, or from reaching the boundary.

\begin{lemma} \label{th:boundary}
If a solution of an equation \eqref{eq:cr} touches $\partial M$, it must be entirely contained in it. 
\end{lemma}

\begin{proof}
Note that $\partial S$ is irrelevant for this argument: because of the boundary conditions, only interior points of $S$ can map to $\partial M$. Suppose, for the moment, that the almost complex structures $J_{S,z}$ used in \eqref{eq:cr} are products $I_{\bC,z} \times I_F$ not just to first order, but on some neighbourhood of $\partial M$. Then, near $\partial M$, the equation splits into base and fibre components, the latter of which is just the equation for an $I_F$-holomorphic map $S \rightarrow F$, with no inhomogeneous term; and the desired result follows from the weak convexity property of $I_F$. Let's reformulate this observation slightly differently. Suppose that, following \eqref{eq:graph}, we rewrite our equation as one for pseudo-holomorphic sections $u^*: S \rightarrow S \times M$, for an almost complex structure $I_{S \times M}^*$ on $S \times M$. Then, near $S \times \partial M$, we can split $I_{S \times M}^*$ as the product of an almost complex structure on $S \times \bC$ (determined by $I_{\bC,z}$, $\beta_S$, and $H$) and the almost complex structure $I_F$ on $F$. It then follows from \eqref{eq:convexity} that $S \times \partial M$ is weakly $I_{S \times M}^*$-convex.

If we use a general family of almost complex structures and carry out the same construction, we end up with an almost complex structure $J^*_{S \times M}$ which, along $S \times \partial M$, agrees with $I_{S \times M}^*$, and the first derivatives are also the same. In particular, the previous weak convexity property still holds (because the Levi forms are the same). By \cite[Corollary 4.7]{diederich-sukhov08}, it follows that $(u^*)^{-1}(S \times \partial M) = u^{-1}(\partial M)$ is open and closed.
\end{proof}

\begin{lemma} \label{th:compactness-2}
Let $u$ be a solution of an equation \eqref{eq:cr} on a noncompact surface $S$, such that $|\Pi(u)| \leq 1$ outside a compact subset. Then $|\Pi(u)| \leq 1$ everywhere.
\end{lemma}

\begin{proof}
Suppose that this is false. By assumption, $T = \{|\Pi(u(z))| \geq r\}$ is compact for all $r>1$, and $\{|\Pi(u(z))| = r\}$ is nonempty if $r$ is sufficiently close to $1$. Since the map $v = \Pi(u)$ satisfies \eqref{eq:cr-c} on the region $\{|v(z)| > 1\}$, we get a contradiction to \eqref{eq:integrated}.
\end{proof}

\begin{lemma} \label{th:injective}
Let $u$ be a solution of an equation \eqref{eq:cr}, such that $U = \{|\Pi(u(z))| < 1\} \subset S$ is nonempty. Then, either there is a point in $U$ where $Du - X \otimes \beta_S \neq 0$; or else, $U = S$ and $Du = X \otimes \beta_S$ everywhere.
\end{lemma}

\begin{proof}
Let's exclude the first possibility, hence assume that $Du = X \otimes \beta_S$ on $U$, and therefore on $\bar{U}$ as well. 
For any path $c: [0,1] \rightarrow \bar{U}$, we therefore have 
\begin{equation} \label{eq:c-path}
u(c(1)) = \phi^b(u(c(0))), \;\; \text{ where } b = \textstyle \int_c \beta_S.
\end{equation}
Suppose that $U \neq S$. We can then find a path which starts in $U$ but is not contained in it. By making the domain smaller, we can find a path such that $c(r)$ lies in $U$ for $r<1$, while $c(1) \notin U$; that path is obviously contained in $\bar{U}$. Since the point $x = u(c(1))$ satisfies $|\Pi(x)| \geq 1$, the same holds for its entire $X$-flow line, hence for all $u(c(r))$ by \eqref{eq:c-path}, which is a contradiction.
\end{proof}

Finally, note that by assumption, $H$ is bounded below, say $H \geq C$. Hence, we have an a priori energy bound for solutions, because of \eqref{eq:energy-control} and \eqref{eq:energy-and-action}:

\begin{lemma} \label{th:apriori}
Consider an equation \eqref{eq:cr}, where $S$ is a punctured-boundary Riemann surface, and $\beta_S$ is a sub-closed one-form as in \eqref{eq:beta1}. For all solutions of that equation with limits \eqref{eq:zeta-limits},
\begin{equation}
E^{\mathit{geom}}(u) \leq \sum_{\zeta \in \Sigma_S^{\pm}} \mp A(x_\zeta) - C \int_S d\beta_S.
\end{equation}
\end{lemma}

\subsection{Transversality\label{subsec:transversality}}
The class of Cauchy-Riemann equations \eqref{eq:cr} relevant for us is subject to several constraints. We stick to a fixed Hamiltonian \eqref{eq:hamiltonian}, hence can't vary the inhomogeneous term. Moreover, the almost complex structures \eqref{eq:almost-complex} are constrained on part of the target manifold. In spite of this, regularity of the moduli spaces still holds generically within that class, in a suitable sense. The basic argument is classical, following \cite{floer-hofer-salamon94}; hence, the discussion here will be kept brief.

\begin{lemma} \label{th:transversality-1}
Take $(L_0,L_1)$ and $w$, such that
\begin{equation} \label{eq:diff-angles}
\alpha_{L_0} \neq \alpha_{L_1}, \text{ and the intersection $\phi^w(L_0) \cap L_1$ is transverse;}
\end{equation}
Then, a generic choice of almost complex structures $(J_t)$, taken as usual from the class \eqref{eq:almost-complex}, makes all solutions of Floer's equation \eqref{eq:floer}, \eqref{eq:floer-limits} regular.
\end{lemma}

\begin{proof}[Sketch of proof]
Because of \eqref{eq:diff-angles}, the limits of $u$ must be contained in $\Pi^{-1}(\{|w| < 1\}) \subset M$. Hence, for any Floer trajectory $u$, the values $u(s,t)$ for $|s| \gg 0$ lie in a region where the almost complex structures are not constrained by the conditions in \eqref{eq:almost-complex}. Suppose that $u$ is not translation-invariant. Then, among those $(s,t)$, one can find ``regular'' ones in the sense of \cite[Theorem 4.3]{floer-hofer-salamon94}. This suffices to make the rest of the argument from \cite{floer-hofer-salamon94} go through.
\end{proof}

In the same spirit:

\begin{lemma} \label{th:transversality-2}
Let $S$ be a boundary-punctured disc, with an asymptotically translation-invariant sub-closed one-form $\beta_S$, whose behaviour at infinity is described by $(w_0,\dots,w_d)$. Equip this with Lagrangian boundary conditions $(L_0,\dots,L_d)$ satisfying the following general position requirements:
\begin{equation} \label{eq:pairwise}
\mybox{
$\alpha_{L_0} \neq \alpha_{L_d}$, and $\phi^{w_0}(L_0) \cap L_d$ is transverse. Similarly, for $k>0$, $\alpha_{L_{k-1}} \neq \alpha_{L_k}$, and $\phi^{w_k}(L_{k-1}) \cap L_k$ is transverse. 
}
\end{equation}
\begin{equation} \label{eq:multi-transverse}
\mybox{
If $d\beta = 0$, $\phi^{w_1+\cdots+w_d}(L_0) \cap \phi^{w_2+\cdots+w_d}(L_1) \cap \cdots \cap L_d$ is empty; respectively, if $d\beta \neq 0$, the intersection $\phi^{w_0}(L_0) \cap \phi^{w_1+\cdots+w_d}(L_0) \cap \phi^{w_2+\cdots+w_d}(L_1) \cap \cdots \cap L_d$ is disjoint from the zero-set of $X$.
}
\end{equation}
Fix families $J_k = (J_{k,t})_{t \in [0,1]}$ of almost complex structures, for $k = 0,\dots,d$. Then, for a generic family $(J_{S,z})$ of almost complex structures which is asymptotically translation-invariant and converges to the $J_k$, the space of solutions of \eqref{eq:cr}, with limits \eqref{eq:disc-limits}, is regular.
\end{lemma}

\begin{proof}
The assumption on the angles $\alpha_{L_k}$ ensures that all limiting chords are contained in $\{|\Pi(x)| < 1\}$. Hence, the subset $U$ from Lemma \ref{th:injective} is certainly nonempty. If there is a point in that subset where $Du - X \otimes \beta_S$ is nonzero, we can achieve transversality by varying $J_{S,z}$ near $u(z)$. 

Let's suppose then that $Du = X \otimes \beta_S$ everywhere. In particular, $u|\partial S$ is locally constant. By looking at its values on boundary components, one sees that the value of $u$ on the $d$-th boundary component lies in the intersection from \eqref{eq:multi-transverse} (this is actually the same intersection in both cases, the only difference being that the term $\phi^{w_0}(L_0)$ is redundant for $d\beta = 0$, since then $w_0 = w_1 + \cdots + w_d$). If $d\beta = 0$, this is enough to reach a contradiction. In the remaining case, observe that \eqref{eq:c-path} holds for all paths in $S$, which means that $u(S)$ is contained in a single orbit of $X$. By \eqref{eq:multi-transverse}, that orbit is necessarily nonconstant. One can find contractible loops $c$ along which $\int_c \beta_S$ is an arbitrarily small nonzero number, and that leads to a contradiction with \eqref{eq:c-path}.
\end{proof}

\subsection{Consistent choices\label{subsec:choice}}
So far, we have considered a single equation \eqref{eq:cr} in isolation. To build a suitable hierarchy of moduli spaces of solutions, one needs to make sure that the data which define those spaces are suitably correlated. Our first step is to make a suitable selection of Lagrangian submanifolds, within the class \eqref{eq:lagrangian}:
\begin{equation} 
\label{eq:collection}
\mybox{ 
A sufficiently large collection of Lagrangian submanifolds is a countable set $\bfL$ of submanifolds, such that the following holds. Given any $L$ from \eqref{eq:lagrangian} and a number $a<\half$, there is an $L^+ \in \bfL$ which has $\alpha_{L^+} \geq a$ and which is isotopic to $L$; as usual, we mean that the isotopy should remain in \eqref{eq:lagrangian}.
}
\end{equation}
It is unproblematic to show that such collections exist, the basic fact being that two Lagrangian submanifolds which are sufficiently $C^1$-close can be deformed into each other (compare \cite[Remark 3.29 and Definition 3.34]{ganatra-pardon-shende17}). Moreover, by generic perturbations, one can achieve the following:
\begin{equation}
\label{eq:collection-2}
\mybox{
A countable set $\bfL$ of Lagrangian submanifolds is said to be in general position if it has the following properties. For any $L_0,L_1 \in \bfL$ with $\alpha_{L_0} > \alpha_{L_1}$ and any $w$, the intersection $\phi^w(L_0) \cap L_1$ is transverse, and also disjoint from the zero-set of $X$; moreover, for any $L_0,L_1,L_2 \in \bfL$ with $\alpha_{L_0} > \alpha_{L_1} > \alpha_{L_2}$ and any $w_1,w_2$, 
\[
\phi^{w_1+w_2}(L_0) \cap \phi^{w_2}(L_1) \cap L_2 = \emptyset.
\]
}
\end{equation}
We assume from now on that an $\bfL$ as in \eqref{eq:collection}, \eqref{eq:collection-2} has been fixed.
\begin{equation}
\label{eq:floer-2}
\mybox{
Take Lagrangian submanifolds $(L_0,L_1)$ in $\bfL$, such that $\alpha_{L_0} > \alpha_{L_1}$. For any such pair and any $w$, we choose a family of almost complex structures $(J_t) = (J_{L_0,L_1,w,t})$ parametrized by $t \in [0,1]$.
}
\end{equation}
Assume that one-forms as in \eqref{eq:asymptotically-consistent-one-forms} have been chosen on all stable weighted popsicles. When it comes to almost complex structures, we proceed in a parallel way. 
\begin{equation} \label{eq:choice-2}
\mybox{
Let $(S,\sigma)$ be a stable weighted popsicle of type $(d+1,p,w)$. Suppose also that we have Lagrangian submanifolds $(L_0,\dots,L_d)$ in $\bfL$, such that $\alpha_{L_0} > \cdots > \alpha_{L_d}$. For each such datum, we choose a family of almost complex structures $(J_{S,z}) = (J_{S,\sigma,w,L_0,\dots,L_d,z})$ parametrized by $z \in S$, in the class \eqref{eq:almost-complex}, and which is asymptotically translation-invariant \eqref{eq:asymptotically}, with limits on the intervals at infinity equal to $J_{L_0,L_d,w_0}$ ($k = 0$) and $J_{L_{k-1},L_k,w_k}$ ($k>0$). 
}
\end{equation}
We require that these families of almost complex structures should be invariant under the action of $\Aut(p)$, and vary smoothly in the moduli space of popsicles. The latter condition can be expressed by saying that they form a smooth family on the fibrewise compactification $|\scrS^{d+1,p,w}$. 
Let's call this a universal choice of almost complex structures. 
\begin{equation} \label{eq:cr-2}
\mybox{
A universal choice of almost complex structures is asymptotically consistent if the following holds. For every $(d+1,p,w)$, the relevant family can be extended smoothly to $\cornerbar{\scrS}^{d+1,p,w}$. Moreover, if we restrict that extension to any component $\bfS_v$ of a broken popsicle, the outcome agrees with that prescribed by the universal choice for weighted popsicles of type $(|v|,\bfp_v,\bfw_v)$.
}
\end{equation}
In the same spirit as \eqref{eq:beta-gluing-0}, \eqref{eq:beta-gluing}, consider a gluing process involving a broken popsicle two pieces. The almost complex structures that enter into the gluing can be written as
\begin{equation} \label{eq:j-gluing-0}
\left\{
\begin{aligned}
& J_{\bfS_1,s_1,t} = J_t + e^{-\pi s_1} \bfO_{1,\exp(-\pi s_1),t}, \\
& J_{\bfS_2,s_2,t} = J_t + e^{\pi s_2} \bfO_{2,\exp(\pi s_2),t};
\end{aligned}
\right.
\end{equation}
and the glued surface carries a family almost complex structure which, on the neck, has the form
\begin{equation} \label{eq:on-the-neck}
J_{S_l,s,t} = J_{\bfS_1,s+l/2,t} + J_{\bfS_2,s-l/2,t} - J_t + O_{\exp(-\pi s - \pi l/2), \exp(\pi s - \pi l/2), t},
\end{equation}
where the correction term is a smooth family $O_{\sigma_1,\sigma_2,t}$. In particular, if we consider a piece of the neck of a fixed size, then on that piece, $J_{S_l} \rightarrow J_t$ exponentially as $l \rightarrow \infty$. Families satisfying the asymptotically consistency condition can be constructed inductively.

\begin{remark}
In principle, one could instead work with consistent choices \eqref{eq:consistent-ends-and-one-forms}; and similarly, demand that the almost complex structures should be strictly translation-invariant on the ends \eqref{eq:j-strict}, as well as on the necks of glued surfaces (in the notation from \eqref{eq:on-the-neck}, this means equality $J_{S_l,s,t} = J_t$ for $l \gg 0$). However, this leads to complications with transversality, and we therefore prefer the asymptotic version of consistency.
\end{remark}

Having fixed almost complex structures as in \eqref{eq:floer-2}--\eqref{eq:cr-2}, we consider the following moduli spaces (updating the definitions from Section \ref{subsec:fantasy}). For $(L_0,L_1)$ and $w$ as in \eqref{eq:floer-2}, we use the almost complex structures chosen there to define the space $\scrR^{2,p,w}(x_0,x_1)$, $|p| = 0$, of Floer trajectories. For other $(d+1,p,w)$ and $(L_0,\dots,L_d)$ as in \eqref{eq:choice-2}, $\scrR^{d+1,p,w}(x_0,\dots,x_d)$ is the space of triples $(S,\sigma,u)$ where $(S,\sigma)$ is a stable weighted popsicle of type $(d+1,p,w)$, and $u$ is a solution of the associated equation \eqref{eq:cr}, with limits \eqref{eq:disc-limits}. The arguments from Lemmas \ref{th:boundary} and \ref{th:compactness-2} apply to all equations appearing in the definition of our moduli spaces. Hence, solutions are always contained in $\Pi^{-1}(\{|w| \leq 1\}) \setminus \partial M \subset M$. A priori energy bounds are provided by \eqref{eq:e-e}, \eqref{eq:energy-and-action}; and bubbling is of course impossible in an exact situation. With that in mind, the construction of the compactification $\bar\scrR^{d+1,p,w}(x_0,\dots,x_d)$ follows the standard strategy. 

The transversality result says that for generic universal choices of almost complex structures which are asymptotically consistent, as in \eqref{eq:floer-2}--\eqref{eq:cr-2}, all the moduli spaces $\scrR^{d+1,p,w}(x_0,\dots,x_d)$ are regular. By a parallel transversality argument, if we fix a conjugacy class $C$ of subgroups in $\Aut(p)$, then the subspace $\scrR^{d+1,p,w}(x_0,\dots,x_d)_C$ where the underlying popsicle has isotropy exactly $C$ is also generically regular. As a concrete consequence, the combination of this and compactness shows that generically, zero-dimensional spaces $\scrR^{d+1,p,w}(x_0,\dots,x_d)$ are finite sets, equipped with a free action of $\Aut(p)$. In brief, the transversality argument works as follows (see \cite[Sections 8c and 8d]{abouzaid-seidel07} for more details). For Floer trajectories, we use Lemma \ref{th:transversality-1} directly. For all other moduli spaces, the assumptions from Lemma \ref{th:transversality-2} are satisfied. Namely, if $d>1$, we use emptiness of triple intersections in \eqref{eq:collection-2}; and in the remaining case $d = 1$, we necessarily have $|p|>0$, so $d\beta_S \neq 0$, and we appeal to the part of \eqref{eq:collection-2} which says that intersections are disjoint from the zero-set of $X$.  Lemma \ref{th:transversality-2} only states generic regularity of the moduli space for a fixed domain; but the standard proof of that fact, which relies on the surjectivity of the ``universal linearized operator'' which includes deformations of the almost complex structure, ensures that the corresponding result for families holds as well. The requirement of $\Aut(p)$-equivariance, which is part of our notion of universal choice, still allows one to vary the almost complex structure on any single popsicle freely; the choices for different popsicles are correlated, but that does not interfere with the transversality argument (see \cite[Remark 8.10]{abouzaid-seidel07}).

\subsection{Grading and signs\label{subsec:grading-and-signs}}
Recall some elementary terminology: given a real vector space $K$, we write $\mathit{or}(K)$ for the free abelian group generated by the two orientations of $K$, with the relation that those generators add to zero. Hence, a choice of isomorphism $\mathit{or}(K) \iso \bZ$ is the same as an orientation of $K$. We actually want to consider $\mathit{or}(K)$ as a graded abelian group, placed in degree $\mathrm{dim}(K)$. For a Fredholm operator $D$, we write $\mathit{or}(D)$ for the orientation space associated to the virtual index space (kernel minus cokernel). Its degree is the Fredholm index $\mathrm{ind}(D)$.

One standard approach to orientations in Floer theory goes as follows \cite[Section 12b]{seidel04}. To every nondegenerate chord $x \in \mathit{ch}(L_0,L_1;w)$ one associates a Fredholm operator $D_x$, and sets
\begin{equation} \label{eq:orientation-x}
\mathit{or}(x) = \mathit{or}(D_x);
\end{equation}
the degree of this graded group is the Maslov index $\mathrm{ind}(x) = \mathrm{ind}(D_x)$. Note that $D_x$ is not unique; however, our assumption that the Lagrangian submanifolds are graded and {\em Spin} ensures that the space of such operators is connected, and that the local system $\mathit{or}(D_x)$ on it is trivial. Hence, \eqref{eq:orientation-x} is well-defined up to canonical isomorphism. Replacing our earlier definition \eqref{eq:floer-cochain} (to which it  reduces if one is willing to pick orientations), we set
\begin{equation} \label{eq:floer-cochain-2}
\mathit{CF}^*(L_0,L_1;w) = \bigoplus_x \mathit{or}(x).
\end{equation}

In the situation of Lemma \ref{th:compactness-2}, write $D_u$ for the operator that linearizes the Cauchy-Riemann equation at $u$. The gluing theory for Fredholm operators yields
\begin{equation} \label{eq:index-formula}
\mathrm{ind}(D_u) + \mathrm{ind}(x_d) + \cdots + \mathrm{ind}(x_1) = \mathrm{ind}(x_0),
\end{equation}
as well as a canonical isomorphism
\begin{equation} \label{eq:or-1}
\mathit{or}(D_u) \otimes \mathit{or}(x_d) \otimes \cdots \otimes \mathit{or}(x_1) \iso \mathit{or}(x_0).
\end{equation}
Turning more specifically to the case of interest here, consider a regular point $(S,\sigma,u)$ in a moduli space $\scrR^{d+1,p,w}(x_0,\dots,x_d)$. From \eqref{eq:index-formula}, we get
\begin{equation} \label{eq:dimension-formula}
\begin{aligned}
\mathrm{dim}(\scrR^{d+1,p,w}(x_0,\dots,x_d)) & = \mathrm{dim}(\scrR^{d+1,p}) + \mathrm{ind}(D_u) \\ & = 
d-2+|p| + \mathrm{ind}(x_0) - \mathrm{ind}(x_d) - \cdots - \mathrm{ind}(x_1).
\end{aligned}
\end{equation}
By a deformation argument for linearized operators, one obtains a canonical isomorphism
\begin{equation} \label{eq:or-2}
\mathit{or}(T\scrR^{d+1,p}) \otimes \mathit{or}(D_u) \iso \mathit{or}(T\scrR^{d+1,p,w}(x_0,\dots,x_d)).
\end{equation}
Here, the tangent space on the right hand side is at the point $(S,\sigma,u)$, and that on the left at the projection $[S,\sigma]$. Combining \eqref{eq:or-1} and \eqref{eq:or-2} yields a canonical isomorphism
\begin{equation} \label{eq:or-3}
\mathit{or}(T\scrR^{d+1,p,w}(x_0,\dots,x_d)) \otimes \mathit{or}(x_d) \otimes \cdots \otimes \mathit{or}(x_1) \iso
\mathit{or}(T\scrR^{d+1,p}) \otimes \mathit{or}(x_0).
\end{equation}
If $(S,\sigma,u)$ is an isolated point in its moduli space, and we use the orientation of $\scrR^{d+1,p}$ chosen in Section \ref{subsec:orientations}, then \eqref{eq:or-3} becomes
\begin{equation} \label{eq:sign-contributions}
\mathit{sign}(S,\sigma,u): \mathit{or}(x_d) \otimes \cdots \otimes \mathit{or}(x_1) \stackrel{\iso}{\longrightarrow}
\mathit{or}(x_0).
\end{equation}

At this point, we assume that all choices of almost complex structures have been made generically, as discussed at the end of Section \ref{subsec:choice}. Take a zero-dimensional space $\scrR^{d+1,p,w}(x_0,\dots,x_d)$, 
which by compactness is a finite set. We count points in it to define a map as in \eqref{eq:mu-p}, but with the following sign specification:
\begin{equation} \label{eq:sign-contributions-2}
\mu^{d,p,w}\,\big|\,\mathit{or}(x_d) \otimes \cdots \otimes \mathit{or}(x_1) = (-1)^{\ast+\heartsuit+\diamondsuit} \textstyle
\bigoplus_{x_0} \sum_{(S,\sigma,u)} \mathit{sign}(S,\sigma,u),
\end{equation}
where
\begin{align}
\label{eq:gah-1}
& \textstyle \ast = \sum_{k=1}^d (k + w_1 + \cdots + w_{k-1}) \mathrm{ind}(x_k), \\
\label{eq:gah-2}
& \textstyle \heartsuit = \sum_{k=1}^d (d-k)w_k, \\
\label{eq:gah-3}
& \textstyle \diamondsuit = \sum_k \# p^{-1}(\{k+1,\dots,d\}) (w_k + \# p^{-1}(k)).
\end{align}
The story for spaces of Floer trajectories is analogous, but of course more familiar, which is why we omit it here. 

\begin{remark} \label{th:signsign}
Signs of type \eqref{eq:gah-1} always appear when one uses the ``reduced degree'' algebraic conventions for $A_\infty$-operations: the calculus of determinant lines more naturally leads to different conventions involving unreduced degrees, and the difference has to be bridged by these ad hoc sign changes. Next, recall that we had chosen orientations of $\scrR^{d+1,p}$ in a way that was simple to explain but somewhat arbitrary. For those orientations, there is no particular reason why one would expect the $A_\infty$-relations to hold; and indeed, one needs to insert the additional signs \eqref{eq:gah-2}, \eqref{eq:gah-3} to make that work.
\end{remark}

\subsection{The ordered categories\label{subsec:ordered-categories}}
Let's formalize the progress made so far. We construct an $A_\infty$-category $\scrF^{\mathit{ord}}$ whose objects are the Lagrangian submanifolds $L \in \bfL$, with morphisms
\begin{equation} \label{eq:f-ord-morphisms}
\scrF^{\mathit{ord}}(L_0,L_1) = \begin{cases} 
\mathit{CF}^*(L_0,L_1) \oplus \mathit{CF}^{*+1}(L_0,L_1;-1) & \alpha_{L_0} > \alpha_{L_1}, \\
\bZ \, e & L_0 = L_1, \\
0 & \text{otherwise.}
\end{cases}
\end{equation}
The elements $e$ are by definition identity morphisms. The nontrivial $A_\infty$-operations $\mu^d_{\scrF^{\mathit{ord}}}$ involve objects $(L_0,\dots,L_d)$ with $\alpha_{L_0} > \cdots > \alpha_{L_d}$. We define them to be the sum of \eqref{eq:sign-contributions-2} over all weights in $\{-1, 0\}$. If we only take the first summand in the nontrivial morphism spaces \eqref{eq:f-ord-morphisms}, we get an $A_\infty$-subcategory $\scrA^{\mathit{ord}} \subset \scrF^{\mathit{ord}}$. In this case, the $A_\infty$-operations use only boundary-punctured discs $S$ with no sprinkles and weights $0$, for which one can choose $\beta_S = 0$ as part of \eqref{eq:asymptotically-consistent-one-forms}; hence, $\scrA^{\mathit{ord}}$ is an ordered version of the classical Fukaya $A_\infty$-structure.

With the signs \eqref{eq:gah-1}--\eqref{eq:gah-3}, the $A_\infty$-associativity relations \eqref{eq:associativity} is satisfied. The sign computation is the same as in \cite[Section 9]{abouzaid-seidel07}, and will not be repeated here (see Remark \ref{th:signsign} for motivation). Still, there are two aspects that seem important enough to warrant a brief discussion, namely the cancellation of boundary contributions postulated in \eqref{eq:cancel-1}, \eqref{eq:cancel-2}:
\begin{equation}
\mybox{
The (free, because of the genericity assumption) action of $\Aut(p)$ on a zero-dimensional space $\scrR^{d+1,p,w}(x_0,\dots,x_d)$ does not change the pseudo-holomorphic map, or the signs \eqref{eq:gah-1}--\eqref{eq:gah-3}. Hence, its only effect on sign contributions comes from the action on $\scrR^{d+1,p}$. Consider boundary points of a one-dimensional moduli space $\bar\scrR^{d+1,p,w}(x_0,\dots,x_d)$, where the underlying popsicle splits as in \eqref{eq:more-symmetry}. Those boundary points come in pairs, exchanged by the action of $\Aut(\bfp_1) = \bZ/2$ on the moduli space $\scrR^{\bfd_1+1,\bfp_1}$ for the first component, and that action is orientation-reversing. Hence, the contribution from each pair indeed cancels.
}
\end{equation}
\begin{equation} \label{eq:cancel-2b}
\mybox{
Now consider the situation from \eqref{eq:cancel-2}, which gives rise to two boundary points of distinct moduli spaces. These again consist of the same maps, but their combinatorial structure is different, and that affects the signs. Taking $k_1 < k_2$ as in \eqref{eq:switch-sprinkle}, inspection of Lemma \ref{th:boundary-orientations} shows that the orientation discrepancies \eqref{eq:boundary-signs} for the two boundary points differ by $(-1)^\spadesuit$, where 
\[
\spadesuit = \# \bfp_2^{-1}(\{k_1+1,\dots,k_2-1\}).
\]
We also have to take into account the difference between the signs \eqref{eq:gah-1}--\eqref{eq:gah-3} for the two moduli spaces. In fact only \eqref{eq:gah-3} changes, by $(-1)^{\spadesuit + 1}$. Hence, the two boundary points indeed yield opposite signs when one adds up the $\mu^{d,p,w}$ to form $\mu^d_{\scrF^{\mathit{ord}}}$.
}
\end{equation}

\begin{remark} \label{th:abouzaid}
It may be useful to compare the algebraic structure constructed here with that in \cite{abouzaid-seidel07}. Suppose that we took the construction from \cite{abouzaid-seidel07} and adapted it to the present geometric situation. The outcome would be an $A_\infty$-category $\scrC^{\mathit{ord}}$ whose nontrivial morphism spaces, for $\alpha_{L_0} > \alpha_{L_1}$, are
\begin{equation} \label{eq:c-category}
\scrC^{\mathit{ord}}(L_0,L_1) = \bigoplus_{w \geq 0} \mathit{CF}^*(L_0,L_1;w) \oplus \mathit{CF}^*(L_0,L_1;w)q,
\end{equation}
where the formal variable $q$ has degree $-1$. A moduli space of weighted popsicles (now having nonnegative weights) with injective $p$ yields multilinear maps 
\begin{equation}
\label{eq:q-shift}
\mathit{CF}^*(L_{d-1},L_d;w_d) q^{\delta_d} \otimes \cdots 
\otimes \mathit{CF}^*(L_0,L_1;w_1) q^{\delta_1} \longrightarrow \mathit{CF}^*(L_0,L_d;w_0) q^{\delta_0}[2-d], 
\end{equation}
for all $\# p^{-1}(k) \leq \delta_k \leq 1$, $k = 1,\dots,d$, such that $\delta_0 = \delta_1 + \cdots + \delta_d - |p| \leq 1$.
To define $\mu_{\scrC^{\mathit{ord}}}^d$, one takes those operations and adds another term in the differential, namely $\partial_q = \mathit{id}: \mathit{CF}^*(L_0,L_1;w)q \rightarrow \mathit{CF}^*(L_0,L_1;w)$. To put \eqref{eq:f-ord-morphisms} on the same footing, one would think of it as obtained by setting $\delta_k = -w_k$ in \eqref{eq:q-shift}, meaning that
\begin{equation}
\scrF^{\mathit{ord}}(L_0,L_1) = \mathit{CF}^*(L_0,L_1;0) \oplus \mathit{CF}^*(L_0,L_1;-1)q.
\end{equation}
With that in mind, the signs in \eqref{eq:sign-contributions-2} agree with those in \cite[Equations (3.43) and (3.44)]{abouzaid-seidel07}, even though they are written slightly differently.
%
\end{remark}

\subsection{Continuation cocycles\label{subsec:continuation}}
Ultimately, one wants the isomorphism class of a Lagrangian submanifold in the Fukaya category to be constant under isotopies within the class \eqref{eq:lagrangian}, something that's obviously not true in $\scrF^{\mathit{ord}}$ because of the ordering condition. Following  \cite{abouzaid-seidel11b, ganatra-pardon-shende17} we will restore the ``missing morphisms'' by categorical localisation with respect to continuation cocycles. The situation is as follows:
\begin{equation} \label{eq:positive-perturbation}
\mybox{
Take $(L^+,L)$ with $\alpha_{L^+} > \alpha_L$, which intersect transversally. Suppose that they are isotopic within the class \eqref{eq:lagrangian}, and make a choice of isotopy, which is a family $(L_s)_{s \in \bR}$ with $L_s = L^+$ for $s \ll 0$, $L_s = L$ for $s \gg 0$, and such that $\alpha_{L_s}$ is a nonincreasing function of $s$ (one can always modify a given isotopy so that it satisfies the last-mentioned condition, by using the flow of $X$). We then say that $L^+$ is a positive perturbation of $L$.
}
\end{equation}
Consider the closed upper half-plane as a Riemann surface $S = \{\mathrm{im}(z) \geq 0\}$ with one negative point at infinity, $\Sigma_S^- = \{i\infty\}$. We equip this with $\beta_S = 0$ and the moving boundary condition $G = \bigcup_s \{s\} \times L_s$. If we take the one-form $\chi_G$ from \eqref{eq:follows} and evaluate it along a lift of $\partial_s$, the outcome is 
\begin{equation} \label{eq:fibrewise-chi}
\chi_G(\partial_s)\,|\, L_s \cap \{|\pi(x)| \geq 1\} = (\partial_s \alpha_{L_s}) (H|L_s) + K_s,
\end{equation}
where $K_s$ describes the fibrewise part of the deformation $\partial_s L_s$. Since that part is the same in the fibre over any point of $e^{2\pi i \alpha_{L_z}} \bR^{\geq 1}$, it follows that $\chi_G(\partial_s)$ is bounded above. A bound on the topological energy \eqref{eq:corrected-topological-energy} therefore implies a bound on the geometric energy of solutions (of the associated Cauchy-Riemann equation with moving boundary conditions); we have seen this argument before, in a toy model situation \eqref{eq:toy-isotopy}. This yields the analogue of Lemma \ref{th:apriori}. There is also a counterpart of Lemma \ref{th:compactness-2}, obtained by replacing \eqref{eq:integrated} with its analogue for \eqref{eq:moving-radial-lines}. For transversality, we can use the same argument as in Lemma \ref{th:transversality-2}, bearing in mind that $L^+ \cap L$ is disjoint from the zero-locus of $X$, by \eqref{eq:collection-2}.

After these preliminaries, we can count isolated solutions and obtain a cocycle
\begin{equation} \label{eq:continuation-cocycle}
c_{L^+,L} \in \mathit{CF}^0(L^+,L),
\end{equation}
called the continuation cocycle. We will need a few well-known properties of such cocycles.
\begin{equation}
\mybox{
The cohomology class $[c_{L^+,L}] \in \mathit{HF}^0(L^+,L)$ remains unchanged if we deform the isotopy in \eqref{eq:positive-perturbation}, keeping the endpoints fixed.
}
\end{equation}
\begin{equation} \label{eq:chain-perturbation}
\mybox{
Let $L^+$ be a positive perturbation of $L$, and $L$ a positive perturbation of $L^-$. Assuming suitable transverse intersection assumptions, we then find that $L^+$ is a positive perturbation of $L^-$ (in the sense that we can concatenate the isotopies); and the associated continuation cocycles are related by
\[
[c_{L^+,L^-}] = [\mu^2(c_{L,L^-}, c_{L^+,L})] \in \mathit{HF}^*(L^+,L^-).
\]
}
\end{equation}
\begin{equation} \label{eq:above}
\mybox{
Suppose that $L_1$ is a positive perturbation of $L_1^-$. If $L_0$ is another Lagrangian submanifold such that $\alpha_{L_0} > \alpha_{L_1}$, satisfying suitable transverse intersection assumptions, the product
\[
\mu^2(c_{L_1,L_1^-},\cdot): \mathit{CF}^*(L_0,L_1;w) \stackrel{\htp}{\longrightarrow} \mathit{CF}^*(L_0,L_1^-;w)
\]
is a quasi-isomorphism.
}
\end{equation}
\begin{equation} \label{eq:below}
\mybox{
Suppose that $L_0^+$ is a positive perturbation of $L_0$. If $L_1$ is another Lagrangian submanifold such that $\alpha_{L_1} < \alpha_{L_0}$, satisfying suitable transverse intersection assumptions, the product
\[
\mu^2(\cdot, c_{L_0^+,L_0}): \mathit{CF}^*(L_0,L_1;w) \stackrel{\htp}{\longrightarrow} \mathit{CF}^*(L_0^+,L_1;w)
\]
is a quasi-isomorphism.
}
\end{equation}
Here, all the $\mu^2$ are triangle products (with no sprinkles). We omit the proofs, which are by standard gluing arguments (for \eqref{eq:above} and \eqref{eq:below}, the gluing argument shows that the products are homotopic to continuation maps; and one then applies a geometric argument to those maps).

\subsection{Fukaya categories by localisation\label{subsec:af}}
Our first application is to the $A_\infty$-category $\scrA^{\mathit{ord}}$. Given a positive perturbation $L^+$ of $L$, both lying in $\bfL$, we take the cocycle $c_{L^+,L}$, and collect all of them in a set $S$. Consider the localisation, in the sense of \eqref{eq:localisation-2},
\begin{equation}
\scrA \stackrel{\mathrm{def}}{=} S^{-1} \scrA^{\mathit{ord}}.
\end{equation}
This is the definition of the Fukaya category of the Lefschetz fibration from \cite{abouzaid-seidel11b}. To explain its basic properties, we use the algebraic tools from Section \ref{subsec:quotient}. First, note that any object $L_0$ in $\scrA^{\mathit{ord}}$ admits approximately $S$-projective resolutions. Namely, given elements of $S(L_i,L_i^-)$ ($i = 1,\dots,m$), choose a positive perturbation $L_0^+$ of $L_0$ such that $\alpha_{L_0^+} > \alpha_{L_i}$. The desired property \eqref{eq:more-projective} follows directly from \eqref{eq:above}. With that in mind, \eqref{eq:below} and \eqref{eq:quotient-hom} imply the following:
\begin{equation} \label{eq:a-hom}
\mybox{
If $\alpha_{L_0} > \alpha_{L_1}$, the localisation functor gives a quasi-isomorphism $\scrA^{\mathit{ord}}(L_0,L_1) \htp \scrA(L_0,L_1)$; and hence, $H^*(\scrA(L_0,L_1)) \iso \mathit{HF}^*(L_0,L_1)$. For general $(L_0,L_1)$, we therefore have $H^*(\scrA(L_0,L_1)) \iso \mathit{HF}^*(L_0^+,L_1)$, where $L_0^+$ is a positive perturbation of $L_0$ such that $\alpha_{L_0^+} > \alpha_{L_1}$.
}
\end{equation}

The same strategy applies to $\scrF^{\mathit{ord}}$. By definition of $\mu^1_{\scrF^{\mathit{ord}}}$,
\begin{equation} \label{eq:cone-continuation}
(c_{L^+,L},0) \in \scrF^{\mathit{ord},0}(L^+,L)
\end{equation}
is a cocycle. In the situation of \eqref{eq:chain-perturbation}, these cocycles satisfy
\begin{equation} \label{eq:chain-perturbation-2}
[(c_{L^+,L^-},0)] = [\mu^2_{\scrF^{\mathit{ord}}}( (c_{L,L^-},0) , (c_{L^+,L},0))] \in H^0(\scrF^{\mathit{ord}}(L^+,L^-)),
\end{equation}
again by definition of $\mu^2_{\scrF^{\mathit{ord}}}$. By using the long exact sequence associated to 
\begin{equation}
0 \rightarrow \mathit{CF}^*(L_0,L_1;0) \longrightarrow \scrF^{\mathit{ord}}(L_0,L_1) \longrightarrow \mathit{CF}^{*+1}(L_0,L_1;-1) \rightarrow 0,
\end{equation}
one sees that in the situations from \eqref{eq:above} and \eqref{eq:below}, there are quasi-isomorphisms
\begin{align} \label{eq:above-2}
& \mu^2_{\scrF^{\mathit{ord}}}((c_{L_1,L_1^-},0),\cdot): \scrF^{\mathit{ord}}(L_0,L_1) \stackrel{\htp}{\longrightarrow} \scrF^{\mathit{ord}}(L_0,L_1^-), \\
& \label{eq:below-2}
\mu^2_{\scrF^{\mathit{ord}}}(\cdot,(c_{L_0^+,L_0},0)): \scrF^{\mathit{ord}}(L_0,L_1) \stackrel{\htp}{\longrightarrow} \scrF^{\mathit{ord}}(L_0^+,L_1).
\end{align}
This time, take $S = S(\scrF^{\mathit{ord}})$ to be the set of cocycles \eqref{eq:cone-continuation}, and 
\begin{equation}
\scrF \stackrel{\mathrm{def}}{=} S^{-1} \scrF^{\mathit{ord}}.
\end{equation}
Strictly in parallel with \eqref{eq:a-hom}, the properties \eqref{eq:above-2}, \eqref{eq:below-2} imply the following:
\begin{equation} \label{eq:f-success}
\mybox{
If $\alpha_{L_0} > \alpha_{L_1}$, the localisation functor gives a quasi-isomorphism $\scrF^{\mathit{ord}}(L_0,L_1) \htp \scrF(L_0,L_1)$;  and hence, $H^*(\scrF(L_0,L_1)) \iso \mathit{HF}^*(L_0,L_1;\mathit{cont})$ in the notation of \eqref{eq:hf-cont}. For general $(L_0,L_1)$, we therefore have $H^*(\scrF(L_0,L_1)) \iso \mathit{HF}^*(L_0^+,L_1;\mathit{cont})$, where $L_0^+$ is a positive perturbation of $L_0$ such that $\alpha_{L_0^+} > \alpha_{L_1}$.
}
\end{equation}
By the universal property \eqref{eq:dg-quotient-property}, the inclusion $\scrA^{\mathit{ord}} \hookrightarrow \scrF^{\mathit{ord}}$ induces an essentially unique $A_\infty$-functor $\scrA \rightarrow \scrF$. On the cohomology level, this yields the map $\mathit{HF}^*(L_0^+,L_1) \rightarrow \mathit{HF}^*(L_0^+,L_1;\mathit{cont})$.

\begin{remark}
Let's return to the discussion initiated in Remark \ref{th:abouzaid}, but this time taking a wider view. One could introduce a category $\scrZ^{\mathit{ord}}$ whose morphism spaces, in the nontrivial case $\alpha_{L_0} > \alpha_{L_1}$, are
\begin{equation} \label{eq:p-morphisms}
\begin{aligned}
\scrZ^{\mathit{ord}}(L_0,L_1) = & \Big( \prod_{w<0} \mathit{CF}^*(L_0, L_1;w) \oplus \mathit{CF}^*(L_0,L_1;w)q \Big)  \\
\oplus &\Big( \bigoplus_{w \geq 0} \mathit{CF}^*(L_0,L_1;w) \oplus \mathit{CF}^*(L_0,L_1;w)q \Big).
\end{aligned}
\end{equation}
The $A_\infty$-structure would involve all weighted popsicles with injective maps $p$, together with an additional $\partial_q$ term. The direct product on the $w<0$ part of \eqref{eq:p-morphisms} has been introduced so that the decreasing filtration obtained by putting upper bounds on the sum of $w$ and the power of $q$ is complete. The subspaces where that filtration is $\leq 0$ form an $A_\infty$-subcategory, which we denote by $\scrN^{\mathit{ord}}$:
\begin{equation}
\begin{aligned}
\scrN^{\mathit{ord}}(L_0,L_1) = \prod_{w \leq 0} \mathit{CF}^*(L_0, L_1;w) \oplus \prod_{w \leq -1} \mathit{CF}^*(L_0,L_1;w)q.
\end{aligned}
\end{equation}
The projection $\scrN^{\mathit{ord}} \rightarrow \scrF^{\mathit{ord}}$ is an $A_\infty$-functor. Moreover, if we take \eqref{eq:c-category}, then the inclusion $\scrC^{\mathit{ord}} \rightarrow \scrZ^{\mathit{ord}}$ is an $A_\infty$-functor.
Let's localize all these $A_\infty$-categories with respect to continuation cocycles. The $A_\infty$-functors induced by inclusions and projection form a diagram
\begin{equation} \label{eq:5}
\xymatrix{
& \scrA\ar[d] \ar[r] &\ar[d] \scrC \\ 
\scrF & \ar[l] \scrN \ar[r] & \scrZ.
}
\end{equation}

As a formal algebro-geometric analogue (roughly corresponding to \eqref{eq:5} under mirror symmetry), one can think of the situation where we have a divisor $F$ inside a projective variety $A$, both assumed to be smooth for simplicity. Let $C = A \setminus F$ be the complement of the divisor; $N$ its formal neighbourhood in $A$; and $Z$ the punctured formal neighbourhood ($Z$ is not particularly well-behaved within classical algebraic geometry, but one can dodge that bullet by considering only coherent sheaves on it that extend to $N$, which is what we will do here). Then, the associated derived categories sit in a diagram 
\begin{equation} \label{eq:5b}
\xymatrix{
& D^b\mathit{Coh}(A) \ar[r] \ar[d] & \ar[d] D^b\mathit{Coh}(C) \\
D^b\mathit{Coh}(F) & \ar[l] D^b\mathit{Coh}(N) \ar[r] & D^b\mathit{Coh}(Z). \\
}
\end{equation}
It is an open question how far this analogy goes: meaning, to what extent \eqref{eq:5} has the properties of \eqref{eq:5b}, and whether there is an actual correspondence in cases where mirror symmetry applies.
\end{remark}

\subsection{A noncommutative divisor\label{subsec:divisor-1}}
By slightly modifying our previous approach, we will construct a noncommutative divisor $\scrD$ with ambient space $\scrA$, and whose fibre is quasi-equivalent to the previous $\scrF$. The additional step is a transfer argument, which may appear to add unnecessary complexity; but that piece of extra work will pay off later, when considering the analogous issue for noncommutative pencils.

We work over the polynomial ring $\bZ[v]$, and define a homogeneous version of  $\scrF^{\mathit{ord}}$, which is an $A_\infty$-category $\scrD^{\mathit{ord}}$ having the same objects as before, and with \eqref{eq:f-ord-morphisms} replaced by
\begin{equation} \label{eq:d-ord-morphisms}
\scrD^{\mathit{ord}}(L_0,L_1) = \begin{cases} 
\mathit{CF}^*(L_0,L_1)[v] \oplus \mathit{CF}^{*+1}(L_0,L_1;-1)[v] & \alpha_{L_0} > \alpha_{L_1}, \\
\bZ[v] \, e & L_0 = L_1, \\
0 & \text{otherwise.}
\end{cases}
\end{equation}
Following the general design from Section \ref{subsec:nc-linear-system}, we equip this with a weight grading, where $\mathit{CF}^*(L_0,L_1)v^r$ has weight $-r$, and $\mathit{CF}^{*+1}(L_0,L_1;-1)v^r$ weight $-r-1$. The $A_\infty$-operations $\mu_{\scrD^{\mathit{ord}}}^d$ are defined by adding up the $\bZ[v]$-linear extensions of $v^{|p|} \mu^{d,p,w}$. For instance, the piece of $\mu^1_{\scrD^{\mathit{ord}}}$ which comes from $(d = 1,\, |p| = 1,\, w_0 = 0, \, w_1 = -1)$ maps $\mathit{CF}^*(L_0,L_1;-1)v^r$ ($r \geq 1$) to $\mathit{CF}^*(L_0,L_1)v^{r+1}$. One can consider the continuation cocycles \eqref{eq:cone-continuation} as cocycles in \eqref{eq:d-ord-morphisms} of weight $0$. Therefore, their mapping cones are elements of the weight-graded version of the category of twisted complexes over $\scrD^{\mathit{ord}}$. As a result, the localised category $S^{-1}\scrD^{\mathit{ord}}$, formed over $\bZ[v]$, inherits the weight grading, as one can see from \eqref{eq:localised-morphisms}. Inspection of that formula also shows that the positive weight part of the morphism spaces is zero. We can spell that out concretely:
\begin{equation} \label{eq:spell}
\begin{aligned}
& S^{-1}\scrD^{\mathit{ord}}(L_0,L_1)  \\ & = \textstyle \scrD^{\mathit{ord}}(L_0,L_1) \oplus 
\bigoplus_C \big( \scrD^{\mathit{ord},\mathit{tw}}(C,L_1)[1] \otimes_{\bZ[v]} \scrD^{\mathit{ord},tw}(L_0,C) \big) \oplus \cdots \\
& \textstyle = \scrD^{\mathit{ord}}(L_0,L_1) \oplus \bigoplus_{(L,L^+)} \big( 
\scrD^{\mathit{ord}}(L,L_1) \otimes_{\bZ[v]} \scrD^{\mathit{ord}}(L_0,L)
\\ & \qquad
\oplus \scrD^{\mathit{ord}}(L^+,L_1) \otimes_{\bZ[v]} \scrD^{\mathit{ord}}(L_0,L)[-1]
\oplus \scrD^{\mathit{ord}}(L,L_1) \otimes_{\bZ[v]} \scrD^{\mathit{ord}}(L_0,L^+)[1]
\\ & \qquad
\oplus \scrD^{\mathit{ord}}(L^+,L_1) \otimes_{\bZ[v]} \scrD^{\mathit{ord}}(L_0,L^+) \big)
\oplus \cdots
\end{aligned}
\end{equation}
where the sum is over $C$ which are mapping cones of a continuation cocycle $L^+ \rightarrow L$.
Given any two objects $(L_0,L_1)$, consider the subcomplex 
\begin{equation} \label{eq:cut-weight}
\scrD(L_0,L_1) \subset S^{-1}\scrD^{\mathit{ord}}(L_0,L_1)
\end{equation}
generated over $\bZ[v]$ in weights $\{-1,0\}$, as in \eqref{eq:weight01}. Concretely, if we take the expression in \eqref{eq:spell} and write the nontrivial morphism spaces in it as a sum of two pieces as in \eqref{eq:d-ord-morphisms}, then $\scrD(L_0,L_1)$ contains those tensor product in which at most one Floer cochain group with $w = -1$ appears as a factor. Choose a positive perturbation $L_0^+$ such that $\alpha_{L_0^+} > \alpha_{L_1}$. Consider the diagram
\begin{equation} \label{eq:extract-weight}
\xymatrix{
\scrD(L_0,L_1) 
\ar[r]^-{\htp}
\ar@{^{(}->}[d]
&
\scrD(L_0^+,L_1) 
\ar@{^{(}->}[d]
&
\ar[l]_-{\htp} 
\scrD^{\mathit{ord}}(L_0^+,L_1)
\ar@{=}[d]
\\
S^{-1}\scrD^{\mathit{ord}}(L_0,L_1)
\ar[r]^-{\htp}
&
S^{-1}\scrD^{\mathit{ord}}(L_0^+,L_1)
&
\ar[l]_-{\htp}
\scrD^{\mathit{ord}}(L_0^+,L_1).
}
\end{equation}
On the bottom row, $\rightarrow$ is given by composition with the image of $(c_{L_0^+,L_0},0)$ in $S^{-1}\scrD^{\mathit{ord}}$. By construction, it is a homotopy equivalence, homogeneous with respect to weights. The $\leftarrow$ next to it is part of the localisation functor, which has the same homotopy equivalence property by the $\bZ[v]$-linear version of \eqref{eq:f-success}. To get the top row, we extract the subcomplexes generated in weights $\{-1,0\}$. Since that process, applied to modules which are trivial in positive weights, is functorial, we again get homotopy equivalences. Diagram-chasing shows that the inclusions \eqref{eq:cut-weight} are homotopy equivalences, again in a weight-homogeneous sense. Let $\scrD$ be the noncommutative divisor obtained by applying \eqref{eq:weight01} to \eqref{eq:cut-weight}. 

\begin{proposition} \label{th:divisor-1}
The ambient space of $\scrD$ is $\scrA$, and its fibre is quasi-equivalent to the previously constructed $\scrF$.
\end{proposition}

\begin{proof}
The weight zero part of \eqref{eq:spell}, which by definition is the ambient space of the divisor, is given by
\begin{equation}
\begin{aligned}
& \scrD(L_0,L_1)^{(0)} = S^{-1}\scrD^{\mathit{ord}}(L_0,L_1)^{(0)} \\ 
&= \textstyle \scrA^{\mathit{ord}}(L_0,L_1) \oplus
\bigoplus_C \big( \scrA^{\mathit{ord},\mathit{tw}}(C,L_1)[1] \otimes \scrA^{\mathit{ord},tw}(L_0,C) \big) \oplus \cdots \\
& \textstyle = \scrA^{\mathit{ord}}(L_0,L_1) \oplus \bigoplus_{(L,L^+)} \big( 
\scrA^{\mathit{ord}}(L,L_1) \otimes \scrA^{\mathit{ord}}(L_0,L)
\\ & \qquad
\oplus \scrA^{\mathit{ord}}(L^+,L_1) \otimes \scrA^{\mathit{ord}}(L_0,L)[-1]
\oplus \scrA^{\mathit{ord}}(L,L_1) \otimes \scrA^{\mathit{ord}}(L_0,L^+)[1]
\\ & \qquad
\oplus \scrA^{\mathit{ord}}(L^+,L_1) \otimes \scrA^{\mathit{ord}}(L_0,L^+) \big)
\oplus \cdots \\
& = \scrA(L_0,L_1).
\end{aligned}
\end{equation}
Moreover, as noted in \eqref{eq:weight01}, the $A_\infty$-structure $\mu_{\scrD}^*$ agrees with $\mu_{S^{-1}\scrD^{\mathit{ord}}}^*$ on the weight zero part, hence is equal to $\mu_{\scrA}^*$. 

Next, we can take the $A_\infty$-functor $\scrD \rightarrow S^{-1}\scrD^{\mathit{ord}}$ provided by the transfer theorem, and specialize it to $v = 1$. This yields an $A_\infty$-functor between the fibre of $\scrD$ and $\scrF = S^{-1}\scrF^{\mathit{ord}}$, which by construction is a quasi-isomorphism on each morphism space.
\end{proof}

\section{The restriction functor\label{sec:restriction}}

We will now show that the $A_\infty$-category $\scrF$ from Section \ref{sec:first-construction} is quasi-equivalent to a full subcategory of the classical Fukaya category of the fibre of the fibration. The fact that the Fukaya category of the fibre appears in such a context (of a noncommutative divisor) is not fundamentally new (it can be traced back to \cite{seidel12b} and further, even if not in the same form), but it emerges here in a particularly natural way.

\subsection{Geometric setup\label{subsec:restriction-setup}}
We revisit the situation from Section \ref{subsec:target}, introducing some technical restrictions, which do not detract from its generality in any important way.
\begin{figure}
\begin{centering}
\begin{picture}(0,0)%
\includegraphics{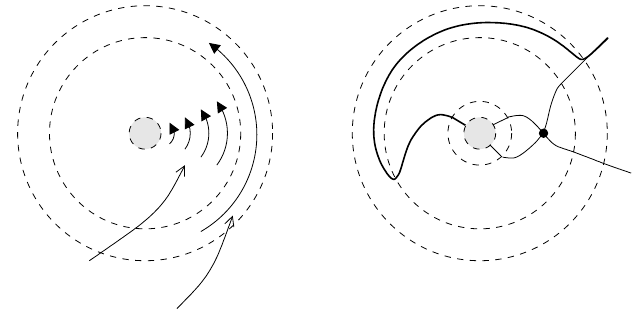}%
\end{picture}%
\setlength{\unitlength}{3355sp}%
\begingroup\makeatletter\ifx\SetFigFont\undefined%
\gdef\SetFigFont#1#2#3#4#5{%
  \reset@font\fontsize{#1}{#2pt}%
  \fontfamily{#3}\fontseries{#4}\fontshape{#5}%
  \selectfont}%
\fi\endgroup%
\begin{picture}(5952,3112)(-1214,-3875)
\put(3782,-2270){\makebox(0,0)[lb]{\smash{{\SetFigFont{8}{9.6}{\rmdefault}{\mddefault}{\updefault}{\color[rgb]{0,0,0}$3/5$}%
}}}}
\put(-1199,-3436){\makebox(0,0)[lb]{\smash{{\SetFigFont{10}{12.0}{\rmdefault}{\mddefault}{\updefault}{\color[rgb]{0,0,0}rotates with speed $\pi$}%
}}}}
\put(-1199,-3811){\makebox(0,0)[lb]{\smash{{\SetFigFont{10}{12.0}{\rmdefault}{\mddefault}{\updefault}{\color[rgb]{0,0,0}radial increase of speed to $2\pi$}%
}}}}
\put(4190,-2393){\makebox(0,0)[lb]{\smash{{\SetFigFont{8}{9.6}{\rmdefault}{\mddefault}{\updefault}{\color[rgb]{0,0,0}$L_1$}%
}}}}
\put(4175,-1580){\makebox(0,0)[lb]{\smash{{\SetFigFont{8}{9.6}{\rmdefault}{\mddefault}{\updefault}{\color[rgb]{0,0,0}$L_0$}%
}}}}
\put(751,-886){\makebox(0,0)[lb]{\smash{{\SetFigFont{8}{9.6}{\rmdefault}{\mddefault}{\updefault}{\color[rgb]{0,0,0}$|w| = 1$}%
}}}}
\put(2574,-2458){\makebox(0,0)[lb]{\smash{{\SetFigFont{8}{9.6}{\rmdefault}{\mddefault}{\updefault}{\color[rgb]{0,0,0}$\phi^{-1}(L_0)$}%
}}}}
\put(-524,-1561){\makebox(0,0)[lb]{\smash{{\SetFigFont{8}{9.6}{\rmdefault}{\mddefault}{\updefault}{\color[rgb]{0,0,0}$|w| = 4/5$}%
}}}}
\put(-224,-2311){\makebox(0,0)[lb]{\smash{{\SetFigFont{8}{9.6}{\rmdefault}{\mddefault}{\updefault}{\color[rgb]{0,0,0}$|w| = 1/5$}%
}}}}
\put(3076,-1636){\makebox(0,0)[lb]{\smash{{\SetFigFont{8}{9.6}{\rmdefault}{\mddefault}{\updefault}{\color[rgb]{0,0,0}$|w| = 2/5$}%
}}}}
\end{picture}%
\caption{\label{fig:restriction}On the left, the Hamiltonian vector field of the function from \eqref{eq:restriction-hamiltonian}. On the right, two examples of Lagrangian submanifolds from \eqref{eq:restriction-lagrangian}, with $\alpha_{L_0} > \alpha_{L_1}$; and the effect of applying the time $(-1)$ map of the Hamiltonian.}
\end{centering}
\end{figure}
\begin{equation} \label{eq:restriction-geometry}
\mybox{
We take a symplectic fibration with singularities \eqref{eq:pi}, but now assume that it is symplectically locally trivial over a larger subset of the base, namely $\{|w| \geq 1/5\} \subset \bC$. Our reference fibre for the following considerations will be $F = M_{3/5}$.
}
\end{equation}
\begin{equation} \label{eq:restriction-almost-complex}
\mybox{When considering almost complex structures as in \eqref{eq:almost-complex}, we additionally require that $D\Pi$ should be $J$-holomorphic on $\Pi^{-1}(\{ 1/5 \leq |w| \leq 2/5\})$.
}
\end{equation}
\begin{equation} \label{eq:restriction-lagrangian}
\mybox{
Fix, once and for all, a function $a: [1/5,\infty) \rightarrow [-1,1]$ which satisfies
\[
\left\{
\begin{aligned}
& a|[1/5,2/5] = -1, \\
& a(3/5) = 0, \\
& a|[4/5,\infty) = 1, \\
& a'(3/5) > 0, \\
& a' \geq 0\; \text{everywhere.}
\end{aligned}
\right.
\]
Within the class of Lagrangian submanifolds $L \subset M$ from \eqref{eq:lagrangian}, we restrict the angles to (the right half-circle) $\alpha_L \in (-\quarter,\quarter)$, and additionally require that 
\[
\Pi(L) \cap \{|w| \geq 1/5\} = \{|w| \geq 1/5, \; \mathrm{arg}(w) = 2\pi \alpha_L a(|w|)\}.
\]
Hence, $L_F = L \cap F$ is itself a (closed exact) Lagrangian submanifold of $F$.
}
\end{equation}
\begin{equation} \label{eq:restriction-hamiltonian}
\mybox{
We take a Hamiltonian $H \in \smooth(M,\bR)$, belonging to the class \eqref{eq:hamiltonian}, and such that 
\[
H(x) = h(\pi |\Pi(x)|^2) \; \text{ on $\Pi^{-1}(\{|w| \geq 1/5\})$},
\]
where
\[
\left\{
\begin{aligned}
& h'(p) = 1/2 \; \text{ for $\sqrt{p/\pi} \in [1/5, 4/5]$,} \\
& h'' \geq 0 \; \text{ everywhere,} \\
& h(p) = p \; \text{ for $\sqrt{p/\pi} \geq 1$}
\end{aligned}
\right.
\]
(see again Figure \ref{fig:restriction}). We also retain the technical condition that the critical point set of $H$ should have no interior points (which is compatible with the other requirements).
}
\end{equation}


\begin{lemma} \label{th:intersect-in-fibre}
Let $(L_0,L_1)$ be as in \eqref{eq:restriction-lagrangian}, with $\alpha_{L_0} > \alpha_{L_1}$. Then
\begin{equation} \label{eq:l-lf}
L_0 \cap L_1 \subset \Pi^{-1}(\{|w| < 1/5\}) \cup F.
\end{equation}
We call intersection points that lie in $\Pi^{-1}(\{|w| < 1/5\})$ interior ones, and the others exterior ones. An exterior point is a transverse intersection point in $M$ iff it is transverse as intersection point of $L_{0,F}$ and $L_{1,F}$ inside $F$.

\end{lemma}

\begin{proof}
By definition, $(L_0,L_1)$ can intersect in $\Pi^{-1}(\{|w| = r\})$, $r \geq 1/5$, only if 
$\alpha_{L_0} a(r) \equiv \alpha_{L_1} a(r) \; \text{mod } \bZ$.
Since $\alpha_{L_0} - \alpha_{L_1} \in (0,1/2)$ and $a(r) \in [-1,1]$, this can happen only when $a(r) = 0$, which by definition means $r = 3/5$. This singles out the fibre at $3/5$. Since $\alpha_{L_0} \neq \alpha_{L_1}$ and $a'(3/5) \neq 0$, the two manifolds are, locally near $3/5$, fibered over transversally intersecting paths. This proves the equivalence of transverse intersections in $M$ and $F$. 
\end{proof}

\begin{lemma} \label{th:do-not-intersect-in-fibre}
Let $(L_0,L_1)$ be as in Lemma \ref{th:intersect-in-fibre}. Then
\begin{equation} \label{eq:l-lf2}
\phi^{-1}(L_0) \cap L_1 \subset \Pi^{-1}(\{|w| < 1/5\}).
\end{equation}
In the same terminology as in the previous Lemma, only interior intersection points (in other words, interior chords) exist in this case.
\end{lemma}

\begin{proof}
In $\{1/5 \leq |w| \leq 4/5\}$, the argument of the path $\Pi(\phi^{-1}(L_0))$, divided by $2\pi$, lies in $[-1,1]\alpha_{L_0} + 1/2 \subset (1/4,3/4)$, while that of $\Pi(L_1)$ lies in $[-1,1]\alpha_{L_1} \subset (-1/4,1/4)$; obviously, there can't be any intersections between the two. 

In $\{4/5 \leq |w| \leq 1\}$, the argument of $\Pi(\phi^{-1}(L_0))$, again divided by $2\pi$, lies in $\alpha_{L_0} + [-1,-1/2] = [\alpha_{L_0},\alpha_{L_0} + 1/2]$; while that of $\Pi(L_1)$ remains constant at $\alpha_{L_1}$. By assumption, $-1/4 < \alpha_{L_1} < \alpha_{L_0} < 1/4$, so these are once more disjoint.
\end{proof}

\subsection{Energy and transversality arguments\label{subsec:partial-energy}}
Let's return to the familiar Cauchy-Riemann equation \eqref{eq:cr}, where:
\begin{equation} \label{eq:restriction-cr}
\mybox{
$S$ is a boundary-punctured disc with a sub-closed one-form $\beta_S$ as in \eqref{eq:beta1}, such that the weights associated to the ends lie in $\{-1,0\}$. The almost complex structures are taken from \eqref{eq:restriction-almost-complex}; the boundary conditions from \eqref{eq:restriction-lagrangian}, assuming moreover that $\alpha_{L_0} > \cdots > \alpha_{L_d}$, and that all chords which can appear as limits are nondegenerate; and the Hamiltonian is as in \eqref{eq:restriction-hamiltonian}.
}
\end{equation}
The aim is to constrain the behaviour of solutions $u$.
\begin{equation} \label{eq:tilde}
\mybox{
Let $\tilde\omega_{\bC} \in \Omega^2(\bC)$ be a rotationally invariant form which is everywhere nonnegative; positive exactly in $\{1/5 < |w| < 2/5\}$; and has integral $1$. Write it as $\tilde\omega_{\bC} = d\tilde\theta_{\bC}$, where $\tilde\theta_{\bC}$ is again rotationally invariant, and vanishes in $\{|w| \leq 1/5\}$. Pull both back by \eqref{eq:pi} to get $\tilde\omega_M$, $\tilde\theta_M$.
}
\end{equation}
\begin{equation}
\mybox{
Lagrangian submanifolds from \eqref{eq:restriction-lagrangian} are isotropic with respect to $\tilde\omega_M$, and exact with respect to $\tilde\theta_M$. Given such an $L$, we choose the primitive $\tilde{K}_L$ for $\tilde\theta_M|L$ which vanishes on $\Pi^{-1}(\{|w| \leq 1/5\})$; this can also be constructed by integrating $\tilde\theta_{\bC}$ along $\Pi(L) \cap \{|w| \geq 1/5\}$, and then pulling back the outcome to $L$.}
\end{equation}
\begin{equation}
\mybox{
The vector field $X$ which generates $(\phi^t)$ has an associated function $\tilde{H} \in \smooth(M,\bR)$, such that $d\tilde{H} = \tilde\omega_M(\cdot,X)$. Again, this is pulled back from $\bC$, and we'll choose it so that it vanishes on $\Pi^{-1}(\{|w| \leq 1/5\})$. By looking at radial derivatives, one sees that $\tilde{H} \geq 0$ everywhere, and $\tilde{H} = \half \int_{\bC} \tilde\omega_{\bC} = \half$ on $\Pi^{-1}(\{|w| \geq 2/5\})$.
}
\end{equation}
The concepts from Section \ref{subsec:energy} can be modified as follows (this is similar to the idea underlying the integrated maximum principle from Section \ref{subsec:maps-to-c}, and its application in Lemma \ref{th:compactness-2}).
\begin{equation} \label{eq:partial-e-geom}
\mybox{
One defines the partial geometric energy to be
\[
\tilde{E}^{\mathit{geom}}(u) = \int_S u^*\tilde\omega_M - d(u^*\tilde{H}) \wedge \beta_S
\]
At all points where the integrand is nonzero, $D\Pi$ is $J$-holomorphic, and $D\Pi(X) = \pi i w\partial_w$. Hence, the partial geometric energy can also be written as the energy of $v = \Pi(u)$ weighted with respect to $\tilde\omega_{\bC}$:
\[
\tilde{E}^{\mathit{geom}}(u) = \int_S |Dv - \pi i v \beta_S |^2 \, v^*(\tilde\omega_{\bC}/\omega_{\bC}).
\]
This shows that $\tilde{E}^{\mathit{geom}}(u) \geq 0$. Equality can hold only if the part of $v$ lying in $\{1/5 < |w| < 2/5\}$ satisfies $Dv = \pi i v \beta_S$. If we assume that that part is nonempty, we get a contradiction with the limiting behaviour, by an argument parallel to the proof of Lemma \ref{th:injective}. The conclusion is that $\tilde{E}^{\mathit{geom}}(u) = 0$ iff $u$ never enters the region $\Pi^{-1}(\{1/5 < |w| < 2/5\})$.}
\end{equation}
\begin{equation}
\mybox{
There is a similar partial version of the topological energy,
\[
\tilde{E}^{\mathit{top}}(u) = \int_S u^*\tilde{\omega}_M - d(u^*\tilde{H}\, \beta_S) =
\tilde{E}^{\mathit{geom}}(u) - \int_S (u^*\tilde{H})\, d\beta_S \geq \tilde{E}^{\mathit{geom}}(u).
\]
Chords have partial actions 
\[
\tilde{A}(x) = \int x^*(-\tilde{\theta}_M + w \tilde{H} \mathit{dt}) \,+ \tilde{K}_{L_1}(x(1)) - \tilde{K}_{L_0}(x(0)), 
\]
with respect to which the analogue of \eqref{eq:energy-and-action} holds. 
}
\end{equation}

\begin{lemma} \label{th:interior-exterior-action}
Take $(L_0,L_1)$ with $\alpha_{L_0} > \alpha_{L_1}$. For $w \in \{0,1\}$, the interior intersection points (or chords) have $\tilde{A}(x) = 0$. For $w = 0$, the exterior intersection points have $\tilde{A}(x) = \alpha_{L_1} - \alpha_{L_0} < 0$.
\end{lemma}

\begin{proof}
The first part is clear: all ingredients that enter into the definition of $\tilde{A}(x)$ vanish on $\Pi^{-1}(\{|w| < 1/5\})$. For the second part, we consider a map
\begin{equation}
\left\{
\begin{aligned} 
& v: [1/5,3/5] \times [0,1] \longrightarrow \{|w| \geq 1/5\} \subset \bC, \\
& |v(1/5,t)| = 1/5, \\
& v(3/5,t) = 3/5, \\
& v(s,0) \in \Pi(L_0), \\
& v(s,1) \in \Pi(L_1).
\end{aligned}
\right.
\end{equation}
By Stokes, 
\begin{equation}
\int  v^*\tilde\omega_{\bC} = \int_{[1/5,3/5] \times \{0\}} v^*\tilde\theta_{\bC} - \int_{[1/5,3/5] \times \{1\}} v^*\tilde\theta_{\bC}
= \tilde{K}_{L_0}(x) - \tilde{K}_{L_1}(x) = -\tilde{A}(x). 
\end{equation}
This is independent of the choice of map; one could choose
\begin{equation}
v(s,t) = s  \exp\big(2 \pi i a(s) (t\alpha_{L_1} + (1-t)\alpha_{L_0})\big).
\end{equation}
Note that $\int v^*\tilde{\omega}_{\bC}$ is nonzero only on $[1/5,2/5] \times [0,1]$. The image of that region under $v$ is bounded by the radial axes $\mathrm{arg}(w) = 2\pi \alpha_{L_0}$, $\mathrm{arg}(w) = 2\pi \alpha_{L_1}$, and segments of the circles $|w| = 1/5$, $|w| = 2/5$. Because $\tilde{\omega}_{\bC}$ is rotationally invariant, the relevant integral is an $(\alpha_{L_1}-\alpha_{L_0})$ fraction of the integral of $\tilde{\omega}_{\bC}$ over the entire annulus, which is $1$ by definition (see Figure \ref{fig:restriction}).
\end{proof}

The following two results are direct consequences:

\begin{lemma} \label{th:interior-interior}
Let $u$ be a solution, such that all limits $(x_0,\dots,x_d)$ are interior intersection points. Then the map $u$ remains in $\Pi^{-1}(\{|w| \leq 1/5\})$.
\end{lemma}

\begin{lemma} \label{th:interior-exterior}
Suppose that $w_0 = 0$. Let $u$ be a solution, such that the limit $x_0$ is an exterior intersection point. Then all the other limits $(x_1,\dots,x_d)$ must be exterior intersection points, which means that necessarily $w_1 = \cdots = w_d = 0$. Moreover, the map $u$ remains in $\Pi^{-1}(\{|w| \geq 2/5\})$.
\end{lemma}

We can also consider a generalization to moving boundary conditions. Let's focus on the special case which is primarily of interest to us: the equation underlying the definition of \eqref{eq:continuation-cocycle}, meaning that $S$ is the upper half-plane with $\beta_S = 0$, and the boundary conditions are as in \eqref{eq:positive-perturbation}, but within the class \eqref{eq:restriction-lagrangian}. We then have an analogue of \eqref{eq:follows} on the graph $G$ of the isotopy, which is a one-form $\tilde{\chi}_G$ satisfying $d\tilde{\chi}_G = \tilde{\omega}_M|G$. One can write this explicitly, following the same idea as in \eqref{eq:toy-chi} and \eqref{eq:fibrewise-chi}:
\begin{equation} \label{eq:tilde-chi}
\tilde\chi_G(\partial_s)\,|\,L_s = -2(\tilde{H}|L) \partial_s\alpha_s,
\end{equation}
where the $-2$ factor reflects the fact that, over the annulus on which $\tilde{\omega}_M$ is supported, $X$ generates a rotation with speed $\pi$, whereas changing $\alpha$ has the effect of rotating that part of our Lagrangian submanifolds with speed $-2\pi$; in other words, this reflects the relation of $h'|[1/5,2/5] = 1/2$ in \eqref{eq:restriction-hamiltonian} with $a|[1/5,2/5] = -1$ in \eqref{eq:restriction-lagrangian}. As in \eqref{eq:corrected-topological-energy}, we can adjust the topological energy to take the moving boundary condition into account, in this case modifying the formula in \eqref{eq:partial-e-geom} to
\begin{equation}
\tilde{E}^{\mathit{top}}(u) = \int_S u^*\tilde{\omega}_M - d(u^*\tilde{H}\, \beta_S) - \int_{\partial S} u^*\tilde{\chi}_G.
\end{equation}
In view of \eqref{eq:tilde-chi}, one then has
\begin{equation} \label{eq:eeee}
\tilde{E}^{\mathit{geom}}(u) = \tilde{E}^{\mathit{top}}(u) + \int_{\partial S} u^*\tilde{\chi}_G
\leq \tilde{E}^{\mathit{top}}(u) - \int_{\bR} \partial_s \alpha_s = \tilde{A}(x) + \alpha_{L^+} - \alpha_{L},
\end{equation}
where $x$ is the limit of $u$. If $x$ is an exterior point, Lemma \ref{th:interior-exterior-action} says that $\tilde{A}(x) = \alpha_L - \alpha_{L^+}$, hence the right hand side of \eqref{eq:eeee} is zero. As an immediate consequence, one gets:

\begin{lemma} \label{th:exterior-cont}
Let $u$ be a solution of the equation underlying the definition of \eqref{eq:continuation-cocycle}, such that the unique limit $x$ is an exterior point of $L^+ \cap L$. Then, the map remains within $\Pi^{-1}(\{|w| \geq 2/5\})$.
\end{lemma}

%

Let's turn briefly to transversality arguments. Compared to the situation from Section \ref{sec:first-construction}, the significant restriction is on the almost complex structures \eqref{eq:restriction-almost-complex}. However, we can compensate for that by proving a version of Lemma \ref{th:injective} where the relevant subset is $U = \{z \in S \;:\; |\Pi(u(z))| \in [0,1/5) \cup (2/5,1)\}$. With that in mind, the proofs of Lemmas \ref{th:transversality-1} and \ref{th:transversality-2} go through as before.

\subsection{Filtering the morphism spaces\label{subsec:q-category}} 
For $(L_0,L_1$) with $\alpha_{L_0} > \alpha_{L_1}$ and $w \in \{-1,0\}$, consider the subgroup 
\begin{equation} \label{eq:interior-only} 
\mathit{CF}^*_{\mathit{int}}(L_0,L_1;w) \subset \mathit{CF}^*(L_0,L_1;w)
\end{equation}
spanned by interior intersection points (for the sake of brevity, we will not explicitly mention the transverse intersection conditions which are required to make the Floer cochain groups well-defined). As a consequence of Lemma \ref{th:interior-exterior}, the subgroups \eqref{eq:interior-only} are subcomplexes, and are preserved by the continuation map \eqref{eq:continuation-map}. Moreover:
\begin{equation} \label{eq:tilde-cont2}
\mybox{
For $\alpha_{L_0} > \alpha_{L_1}$, the continuation map restricts to a homotopy equivalence
\[
\mathit{CF}^*_{\mathit{int}}(L_0,L_1;-1) \stackrel{\htp}{\longrightarrow} \mathit{CF}^*_{\mathit{int}}(L_0,L_1).
\]
}
\end{equation}
Here, Lemma \ref{th:interior-interior} applies, showing that \eqref{eq:tilde-cont2} counts only maps which remain inside the subset $\Pi^{-1}(\{|w| \leq 1/5\})$. With that in mind, the statement is standard, since during the isotopy $\phi^t(L_0)$, $t \in [-1,0]$, no intersection points with $L_1$ enter or leave that subset.
We will need some more statements of a similar kind, which are filtered versions of \eqref{eq:above} and \eqref{eq:below}:
\begin{equation} \label{eq:above2}
\mybox{
Suppose that $L_1$ is a positive perturbation of $L_1^-$, and $L_0$ another Lagrangian submanifold such that $\alpha_{L_0} > \alpha_{L_1}$. Then, the product with the continuation cocycle preserves the subspaces generated by interior points, and induces a quasi-isomorphism
\[
\mu^2(c_{L_1,L_1^-},\cdot): \mathit{CF}^*(L_0,L_1)/\mathit{CF}^*_{\mathit{int}}(L_0,L_1) \stackrel{\htp}{\longrightarrow} \mathit{CF}^*(L_0,L_1^-)/\mathit{CF}^*_{\mathit{int}}(L_0,L_1^-).
\]
}
\end{equation}
\begin{equation} \label{eq:below2}
\mybox{
Suppose that $L_0^+$ is a positive perturbation of $L_0$, and $L_1$ another Lagrangian submanifold such that $\alpha_{L_1} < \alpha_{L_0}$. Then, the product with the continuation cocycle preserves the subspaces generated by interior points, and induces a quasi-isomorphism
\[
\mu^2(\cdot, c_{L_0^+,L_0}): \mathit{CF}^*(L_0,L_1)/\mathit{CF}^*_{\mathit{int}}(L_0,L_1) \stackrel{\htp}{\longrightarrow} \mathit{CF}^*(L_0^+,L_1)/\mathit{CF}^*_{\mathit{int}}(L_0^+,L_1).
\]
}
\end{equation}
The fact that the products from \eqref{eq:above2}, \eqref{eq:below2} preserve the interior subspaces is an application of Lemma \ref{th:interior-exterior}. To see that they induce quasi-isomorphisms on the quotients, one proceeds roughly as follows. To fix the notation, let's suppose that we are considering \eqref{eq:below2}. As a consequence of Lemma \ref{th:exterior-cont}, the image of $c_{L_0^+,L_0}$ in $\mathit{CF}^*(L_0^+,L_0)/\mathit{CF}^*_{\mathit{int}}(L_0^+,L_0)$ involves only maps whose image remains in $\Pi^{-1}(\{|w| \geq 2/5\})$. This, together with Lemma \ref{th:interior-exterior}, tells us the same thing for the map
\begin{equation} \label{eq:q-map}
\mu^2(\cdot, c_{L_0^+,L_0}): \mathit{CF}^*(L_0,L_1) / \mathit{CF}^*_{\mathit{int}}(L_0,L_1)\longrightarrow \mathit{CF}^*(L_0^+,L_1) / \mathit{CF}^*_{\mathit{int}}(L_0^+,L_1).
\end{equation}
By a gluing argument, that map is chain homotopic to a kind of continuation map (still defined by considering moving boundary conditions). There is an analogue of Lemma \ref{th:exterior-cont} for that map, meaning that it can be described using only $\Pi^{-1}(\{|w| \geq 2/5\})$. At this point, a geometric argument similar to that mentioned above for \eqref{eq:tilde-cont2} can be brought to bear, based on the fact that during isotopy from $L_0$ to $L_0^+$, all the exterior intersection points with $L_1$ remain in $F = M_{3/5}$.

Let's put those kinds of considerations on a more systematic footing. Define categories $\scrA^{\mathit{ord}}$ and $\scrF^{\mathit{ord}}$ as in Section \ref{subsec:ordered-categories}, but obviously taking into account our more specific geometric context. From \eqref{eq:interior-only}, we get subgroups $\scrA^{\mathit{ord}}_{\mathit{int}}(L_0,L_1)$ and $\scrF^{\mathit{ord}}_{\mathit{int}}(L_0,L_1)$ (defined to be zero if $L_0 = L_1$). Let's write
\begin{equation} \label{eq:q-morphisms}
\begin{aligned}
\scrQ^{\mathit{ord}}(L_0,L_1) &\; = \scrA^{\mathit{ord}}(L_0,L_1)/\scrA^{\mathit{ord}}_{\mathit{int}}(L_0,L_1)
\\ &\; = \scrF^{\mathit{ord}}(L_0,L_1) / \scrF^{\mathit{ord}}_{\mathit{int}}(L_0,L_1),
\end{aligned}
\end{equation}
Thanks to Lemma \ref{th:do-not-intersect-in-fibre}, it makes no difference whether one thinks of the quotient in terms of $\scrA^{\mathit{ord}}$ or $\scrF^{\mathit{ord}}$. As a direct consequence of Lemma \ref{th:interior-exterior}, the spaces \eqref{eq:q-morphisms} inherit an $A_\infty$-structure, which admits the following interpretation: generators (in the nontrivial case $\alpha_{L_0} > \alpha_{L_1}$) are exterior intersection points, and the nontrivial $A_\infty$-products count only solutions of the relevant Cauchy-Riemann equation which are contained in $\Pi^{-1}(\{|w| \geq 2/5\})$. As a consequence of \eqref{eq:tilde-cont2}, projection is a homotopy equivalence
\begin{equation} \label{eq:project-to-q}
\scrF^{\mathit{ord}}(L_0,L_1) \stackrel{\htp}{\longrightarrow} \scrQ^{\mathit{ord}}(L_0,L_1).
\end{equation}
Moreover, in the situations of \eqref{eq:above2} and \eqref{eq:below2}, the continuation cocycle, projected to the appropriate $\scrQ^{\mathit{ord}}$ space, induces homotopy equivalences
\begin{align}
& \mu^2(c_{L_1,L_1^-},\cdot): \scrQ^{\mathit{ord}}(L_0,L_1) \stackrel{\htp}{\longrightarrow} \scrQ^{\mathit{ord}}(L_0,L_1^-), \\
& \mu^2(\cdot,c_{L_0^+,L_0}): \scrQ^{\mathit{ord}}(L_0,L_1) \stackrel{\htp}{\longrightarrow} \scrQ^{\mathit{ord}}(L_0^+,L_1).
\end{align}

As before, let $S$ be the set of continuation cocycles. To be more precise, there are three versions of this, $S(\scrA^{\mathit{ord}})$, $S(\scrF^{\mathit{ord}})$, and $S(\scrQ^{\mathit{ord}})$, where the latter is obtained from one of the previous two by projection to \eqref{eq:q-morphisms}. Let $\scrA$, $\scrF$, and $\scrQ$ be the localisations with respect to the appropriate $S$. The first two categories are familiar from Section \ref{subsec:af}, and $\scrQ$ shares their essential properties. By this, we mean that morphism spaces can be computed, up to chain homotopy, by using the product with a continuation cocycle together with the quotient functor:
\begin{equation}
\scrQ(L_0,L_1) \stackrel{\htp}{\longrightarrow} \scrQ(L_0^+,L_1) \stackrel{\htp}{\longleftarrow}
\scrQ^{\mathit{ord}}(L_0^+,L_1),
\end{equation}
where $L_0^+$ is a positive perturbation with $\alpha_{L_0^+} > \alpha_{L_1}$. From that and \eqref{eq:project-to-q}, we immediately obtain the following:

\begin{lemma} \label{th:f-q}
Projection $\scrF^{\mathit{ord}} \rightarrow \scrQ^{\mathit{ord}}$ induces a quasi-equivalence $\scrF \rightarrow \scrQ$.
\end{lemma}

\subsection{Geometry of the $A_\infty$-category $\scrQ$\label{subsec:more-q}}
At this point, we have shown that $\scrF$ is quasi-equivalent to the category $\scrQ$, which is defined as a version of the ordinary Fukaya category (using boundary-punctured discs with no popsicles and $\beta_S = 0$) on the part of $M$ lying over the annulus $\{2/5 \leq |w| \leq 1\}$. In fact, by looking at the position of the Lagrangian submanifolds and intersection points, one sees that we could lift everything to the universal cover of that annulus. Then, in a version of the language of \cite{sylvan16}, what we are looking at is (a subcategory of) the partially wrapped Fukaya category of $[0,1] \times \bR \times F$ with ``stops'' at $\{(0,0)\} \times F$ and $\{(1,0)\} \times F$, which is the ``stabilization of $F$'' as in \cite[Definition 2.11]{sylvan16}. This is related to the standard Fukaya category of $F$ by a special case of the Kunneth theorem \cite[Theorem 1.5]{ganatra-pardon-shende18}. The outcome of these considerations is:

\begin{lemma} \label{th:gps}
$\scrQ$ is quasi-equivalent to the full subcategory of the Fukaya category of $F$ whose objects are $L_F$, for $L$ a Lagrangian submanifold in our collection $\bfL$ (with the induced grading and {\em Spin} structure).
\end{lemma}

By combining this with Lemma \ref{th:f-q}, one arrives at the main result of this section:

\begin{prop} \label{th:restriction-functor}
$\scrF$ is quasi-equivalent to a full subcategory of the Fukaya category of $F$, defined as in Lemma \ref{th:gps}.
\end{prop}

As explained above, Lemma \ref{th:gps} is a version of a known result. Nevertheless, in the interest of keeping the discussion self-contained, we outline a proof which uses only classical techniques \cite{seidel04, seidel12b}. For that, we further specialize \eqref{eq:restriction-almost-complex} by imposing the following additional condition on almost complex structures:
\begin{equation} \label{th:restriction-almost-complex-2}
\mybox{Over a fixed small closed disc $\Delta \subset \bC$ centered at $3/5$, the projection $\Pi$ is $J$-holomorphic. In particular, $F$ itself is an almost complex submanifold.
}
\end{equation}
For maps as in \eqref{eq:restriction-cr}, this implies the following additional property:

\begin{lemma} \label{th:degree-f}
Suppose that $\beta_S = 0$, and let $u$ be a solution all of whose limits are exterior intersection points. Then, the map $u$ must remain in $F$.
\end{lemma}

\begin{proof}
In this situation, we are looking at an ordinary pseudo-holomorphic map equation (no inhomogeneous term). We can carry out a partial energy argument as before, but where now the relevant two-form is pulled back from one supported on $\Delta$. All our solutions have zero topological and hence geometric energy, which means that $u$ must be locally constant on the open subset $u^{-1}\Pi^{-1}(\Delta \setminus \partial \Delta)$. Because of the limit condition, this implies the desired result.
%
%
\end{proof}

Transversality for solutions contained in $F$ requires separate considerations of fibre and base components, where the latter involves the ordering condition on the $\alpha_{L_k}$; we refer to \cite[Lemma 14.7]{seidel04} (where instead of an ordering condition, a more general assumption \cite[Equation (14.10)]{seidel04} on the inhomogeneous term is used) or \cite[Lemma 7.1]{seidel12b}. The comparison of gradings and signs between $M$ and $F$ is given by \cite[Lemma 14.9]{seidel04}. These results show that $\scrQ^{\mathit{ord}}$ can be identified with an ordered version of the Fukaya category of $F$. Lemma \ref{th:degree-f} has an analogue for continuation cocycles, and in the resulting localisation is obviously quasi-equivalent to a subcategory of the Fukaya category of $F$ defined in the more conventional way (say, following \cite{seidel04}).

\section{Ample hypersurfaces\label{sec:anticanonical}}

We now turn to a slightly different geometric situation, namely, closed monotone symplectic manifolds equipped with an anticanonical (hence ample) hypersurface, and a section of the normal bundle to that hypersurface. The issues most relevant to us play out in the neighbourhood of the hypersurface. The behaviour of pseudo-holomorphic curves in such a situation is a well-explored topic, see e.g.\ \cite{cieliebak-mohnke07, charest-woodward17, tonkonog17}; our approach borrows from \cite[Section 7]{seidel16} as well. The outcome will be another noncommutative divisor, which encodes pseudo-holomorphic maps that intersect the hypersurface. Note that, in spite of the notational overlaps with Section \ref{sec:first-construction}, the algebraic objects in this section are not the same (they play a parallel role, but their definition is different).

\subsection{Local geometric setup\label{subsec:local-geometry}}
Let $D$ be a closed $(2n-2)$-dimensional symplectic manifold, with the following additional data:
\begin{equation}
\mybox{
$P \rightarrow D$ is a complex line bundle, such that $c_1(P)$ is an integral lift of $[\omega_D]$. It comes with a section $Z$ which transversally cuts out a symplectic hypersurface $B = Z^{-1}(0) \subset D$; and the symplectic orientation of $B$ should agree with that inherited as a zero-locus. We identify $D$ with the zero-section in $P$. We write $D_c \subset P$ for the image of $cZ$, $c \in \bC$, so that $D_0 = D$.
}
\end{equation}
From elementary symplectic geometry, recall that a metric $\|\cdot\|$ and hermitian connection $\nabla$ on $P$ determine a closed two-form $\omega_P$ on its total space. Explicitly, if we work in a local trivialization of $P$ compatible with $\|\cdot\|$, and write $\nabla = d+iA$ for $A \in \Omega^1(D)$, then
\begin{equation} \label{eq:omega-p}
\omega_P = \omega_{\bC} + \omega_D + d(\half |y|^2 A),
\end{equation}
where $y \in \bC$ is the fibre coordinate. In a neighbourhood of the zero-section, $\omega_P$ is symplectic. On the complement of the zero-section, it is exact. More concretely, the curvature of the connection is $2\pi i (d\theta_D - \omega_D)$ for some $\theta_D \in \Omega^1(D)$, and the choice of $\theta_D$ determines a primitive $\theta_{P \setminus D}$ for $\omega_P\,|\,(P \setminus D)$. In a local trivialization as before, this means $\omega_D = -dA/2\pi + d\theta_D$, and 
\begin{equation} \label{eq:theta-p}
\theta_{P \setminus D} = \textstyle \theta_{\bC} + \theta_D + \half |y|^2 A - \frac{1}{2\pi}(A + d\mathrm{arg}(y)),
\end{equation}
where $\theta_{\bC} = \frac{i}{4}(y\, d\bar{y} - \bar{y} \, dy)$. 
We actually want to adopt a particular choice of metric and connection:
\begin{equation} \label{eq:z-trivialization}
\mybox{
Fix $\|\cdot\|$ and $\nabla$ so that, outside an open neighbourhood $U_B \subset D$ of $B$, $Z$ has norm $1$ and is covariantly constant. Hence, if we use $Z$ to partially trivialize $P$,
\[
P\,|\, (D \setminus U_B) \iso \bC \times (D \setminus U_B), 
\]
then $\omega_P = \omega_{\bC} + \omega_D$ in that trivialization; and similarly, $\theta_{P \setminus D} = \theta_{\bC} + \theta_D - \frac{1}{2\pi} d\mathrm{arg}(y)$, as well as $d\theta_D = \omega_D$ (everything outside $U_B$, of course).
}
\end{equation}
Other classes of geometric objects are adjusted accordingly:
\begin{equation} \label{eq:local-hamiltonian}
\mybox{
We use a function $H$ on $P$ which has vanishing value and first derivative along the zero-section. Moreover, its restriction to the partial trivialization should be $H = -\pi \|y\|^2$, so that the associated Hamiltonian vector field is clockwise rotation there.
}
\end{equation}
Really, the notion of Hamiltonian vector field makes sense only in a neighbourhood of the zero-section, where $\omega_P$ is symplectic. The same caveat will apply tacitly at many points throughout the following discussion.
\begin{equation} 
\label{eq:local-lagrangian}
\mybox{
We consider submanifolds with boundary $L \subset P$ of the following form. Pick a Lagrangian submanifold $L_D \subset D$ which is disjoint from $\bar{U}_B$, and some $\alpha_L \in (-\half,\half)$. Over each point $x \in L_D$, take the ray $\bR^+ e^{-2\pi i\alpha_L} Z_x$, and let $L$ be their union. Equivalently, in the partial trivialization from \eqref{eq:z-trivialization}, $L$ is the product $\bR^+ e^{-2\pi i\alpha_L} \times L_D$. Hence, it is an $\omega_P$-Lagrangian submanifold, with boundary $\partial L = L_D$. 
}
\end{equation}
\begin{equation} \label{eq:local-almost-complex}
\mybox{
We consider almost complex structures $J$, defined on a neighbourhood of the zero-section in $P$, and compatible with $\omega_P$, such that the following holds. Both $D$ and $B$ are almost complex submanifolds with respect to $J$. Moreover, in the partial trivialization from \eqref{eq:z-trivialization}, projection to the $\bC$ factor is $J$-holomorphic.
}
\end{equation}

\begin{lemma} \label{th:local-transversality}
For a generic $J$ as in \eqref{eq:local-almost-complex}, every simple (not multiply-covered) $J$-holomorphic sphere which is not contained in $D$ is regular.
\end{lemma}

\begin{proof}
The $J$-holomorphic sphere can't be contained in the closed subset where \eqref{eq:z-trivialization} applies, since the symplectic form is exact there. On the other hand, it can intersect $D$ only at finitely many points. It follows that there is a point lying on it which is neither inside \eqref{eq:z-trivialization} nor in $D$. Near that point, the almost complex structure can be varied freely, which is enough to make the standard argument go through.
\end{proof}

The restriction that the curve should not be contained in $D$ could be dropped, because of the ``positivity'' of the normal bundles of $D \subset P$ and $B \subset D$; but we will not need that.

\begin{remark} \label{th:local-almost-complex-2}
It is possible (after a preliminary deformation of $Z$ near $B$) to find almost complex structures which, in addition to the conditions in \eqref{eq:local-almost-complex}, make all the $D_c$ into almost complex submanifolds. Namely, set $\omega_B = \omega_P|B$; choose a connection $A_B$ on $P|B$; and consider the total space of the rank two bundle $Q = P|B \oplus P|B$, with its associated symplectic form
\begin{equation}
\omega_Q = \omega_{\bC^2} + \omega_B + d(\half |y_1|^2 A_B + \half |y_2|^2 A_B).
\end{equation}
One can identify a neighbourhood of the zero-section in $Q$ with a neighbourbood of $B$ inside the total space of $P$. 
After deforming $Z$ locally near $B$ (within the class of sections that vanish transversally along $B$), we may assume that with respect to this identification, 
\begin{equation} \label{eq:local-d-c}
D_c = \{y_2 = cy_1\} \subset Q.
\end{equation}
Any compatible almost complex structure $J_B$ induces an almost complex structure $J_Q$ on $Q$, which is the standard complex structure on each $\bC^2$ fibre, and is then characterized by
\begin{equation}
(J_Q)_{x,y_1,y_2}(\xi,\, -iy_1 A_B(\xi),\, -iy_2 A_B(\xi)) = (J_B\xi,\, -iy_1 A_B(J_B\xi),\, -iy_2 A_B(J_B \xi)).
\end{equation}
With respect to any such almost complex structure, all the \eqref{eq:local-d-c} are almost complex submanifolds.
\end{remark}

\subsection{Intersection numbers\label{subsec:local-intersection}}
Consider solutions of an equation \eqref{eq:cr}, where the target space is (a neighbourhood of the zero-section inside) $M = P$, and the remaining data are as in \eqref{eq:local-hamiltonian}-\eqref{eq:local-almost-complex}. In the region where the partial trivialization from \eqref{eq:z-trivialization} applies, we can project to the $\bC$-factor, and get solutions of a slightly modified version of \eqref{eq:cr-c}. Concretely, this is
\begin{equation} \label{eq:v-equation}
\left\{
\begin{aligned}
& v: S \longrightarrow \bC, \\
& \bar\partial v +  2\pi i v \beta_S^{0,1}= 0, \\
& v(I) \subset e^{-2\pi i \alpha_I} \bR^+;
\end{aligned}
\right.
\end{equation}
which differs from \eqref{eq:cr-c} in the sign of the inhomogeneous term (and of $\alpha_I$). The sign change means that instead of the maximum principle \eqref{eq:integrated}, we have a minimum principle (for $|v|$, assuming that never vanishes); the other results from Section \ref{subsec:maps-to-c} are unaffected. 
%
%

\begin{lemma} \label{th:discrete-2}
For a solution $u$ which is not contained in $D$, $u^{-1}(D)$ is discrete. 
\end{lemma}

\begin{proof}
Since $X$ vanishes along $D$, $S \times D \subset S \times P$ is an almost complex submanifold with respect to the almost complex structure $J^*_S$ from \eqref{eq:graph}. Hence, at interior points of $S$, the result follows from \eqref{eq:graph} and standard pseudo-holomorphic curve theory. If $u$ meets $D$ at some point $z \in I \subset \partial S$, then $u(z) \in L_I \cap D \subset D \setminus \bar{U}_B$. Projecting to $\bC$ as in \eqref{eq:v-equation} reduces the problem to \eqref{eq:taylor}.
\end{proof}

Take $u$ as in Lemma \ref{th:discrete-2}. We associate to any point $z \in S$ a nonnegative integer $\mu_z(u) \geq 0$, as follows. If $u(z) \notin D$, $\mu_z(u) = 0$. For interior points of $u^{-1}(D)$, $\mu_z(u)$ is the local intersection multiplicity of $u$ and $D$, which is the same as that of $u^*$ and $S \times D$ in the associated equation \eqref{eq:graph}, hence finite and positive by standard pseudo-holomorphic curve theory. For a boundary point, we define it as the order of vanishing $\mu_z(v)$ of the associated map \eqref{eq:v-equation}, in the sense of \eqref{eq:taylor}. Assuming that $u^{-1}(D)$ is finite, we set $\mu(u) = \sum_z \mu_z(u)$. Since the zero-section is Poincar\'{e} dual to the symplectic class, these intersection multiplicities can also be expressed through integral formulae, of which we give two ``local'' versions for specific domains.
\begin{equation} \label{eq:energy-defect-1}
\mybox{
Let $T \subset \bC$ be the closed unit disc, and $u: T \rightarrow P$ be a solution of our Cauchy-Riemann equation, but without Lagrangian boundary conditions; instead, we require that $u(\partial T) \cap D = \emptyset$. Then
\[
\mu(u) = \int_T u^*\omega_P - \int_{\partial T} u^*\theta_{P \setminus D}.
\]
}
\end{equation}
\begin{equation} \label{eq:energy-defect-2}
\mybox{
Take the closed half-disc $T^+ = \{ z \in \bC \;:\; |z| \leq 1, \, \mathrm{im}(z) \geq 0 \}$, and write $\partial_{\mathit{in}} T^+ = \{ |z| = 1 \} \cap T^+$, $\partial_{\mathit{out}} T^+ = \{\mathrm{im}(z) = 0 \} \cap T^+$. Suppose that $L \subset P$ is as in \eqref{eq:local-lagrangian}, and that its interior is exact, $\theta_{P \setminus D} | (L \setminus \partial L) = dK_L$. Let $u: T^+ \rightarrow P$ be a solution of our Cauchy-Riemann equation, with the partial boundary condition $u(\partial_{\mathit{out}}T^+) \subset L$, and such that $u(\partial_{\mathit{in}}T^+) \cap D = \emptyset$. Then
\[
\mu(u) = \int_{T^+} u^*\omega_P\; - \int_{\partial_{\mathit{in}} T^+} u^*\theta_{P \setminus D} \; - K_L(u(1)) + K_L(u(-1)).
\]
}
\end{equation}
Topologically, it is sometimes simpler to consider the intersection with $D_c$, $c<0$, instead of $D = D_0$, since those hypersurfaces are disjoint from our Lagrangian submanifolds. Here is an application of this idea, again formulated only for a particular kind of domain:

\begin{lemma} \label{th:bound-intersection}
Take a half-strip $Z^{\pm}$, with a one-form $\beta_{Z^{\pm}}$ which is closed, vanishes on $\{t = 0,1\}$, and is asymptotic to $w \mathit{dt}$ in the sense of \eqref{eq:beta1}. As boundary conditions, use Lagrangian submanifolds $(L_0,L_1)$ as in \eqref{eq:local-lagrangian}, which intersect transversally and satisfy $\alpha_{L_0} > \alpha_{L_1}$. Finally, choose a family $(J_{s,t})$ of almost complex structures \eqref{eq:local-almost-complex} which is asymptotically translation-invariant. Let $u$ be a solution of the associated equation \eqref{eq:cr} with boundary conditions $u(s,0) \in L_0$, $u(s,1) \in L_1$; such that $u(0,t) \notin D$ for all $t$; and whose limit is a point of $L_0 \cap L_1 = L_{D,0} \cap L_{D,1} \subset D$. Then $u^{-1}(D)$ is finite; and for all sufficiently small $c<0$, we have the intersection number inequality
\begin{equation} \label{eq:minimum-intersection}
u \cdot D_c \geq \mu(u)  + \begin{cases} -w & \text{in the $Z^-$ case,} \\
w+1 & \text{in the $Z^+$ case.}
\end{cases}
\end{equation}
\end{lemma}

\begin{proof}
Let's compactify the domain by adding a point $\{s = \pm \infty\}$. If $c<0$ is small, the restriction of $u$ to the boundary of the compactified domain is disjoint from $D_c$. Therefore, the intersection number $u \cdot D_c$ is well-defined, and independent of $c$ as long as that remains sufficiently small.

In the limit $c \rightarrow 0$, points of $u^{-1}(D_c)$ converge to points of $u^{-1}(D) \cup \{\pm\infty\}$. Hence, $u \cdot D_c$ can be written as a sum of contributions associated to such points. If $z \in u^{-1}(D)$ is an interior point, the contribution is obviously $\mu_z(u)$. For a boundary point, we can locally project to $\bC$ and get a solution of \eqref{eq:v-equation} which vanishes exactly to order $2\mu_z(u)$, and has boundary values on $e^{2\pi i \alpha_{L_0}}\bR^+$ or $e^{2\pi i \alpha_{L_1}} \bR^+$. It therefore has degree $\mu_z(u)$ over every small negative point, which implies that the contribution is again $\mu_z(u)$.

Finally, we have to consider the contribution at infinity. Let's look at the case of $Z^-$ for concreteness, and again focus on the $\bC$-projection $v$. Expanding it as in \eqref{eq:v-dag}, \eqref{eq:v-dag-2}, but taking into account the sign in \eqref{eq:v-equation}, yields the asymptotic expression 
\begin{equation} \label{eq:asymp-minus}
v(s,t) \sim c\, e^{2\pi (\alpha_{L_0} - \alpha_{L_1} + \nu) s} e^{2\pi i(\alpha_{L_0} - \alpha_{L_1} + \nu - w) t}
\end{equation}
for some integer $\nu \geq \alpha_{L_1}-\alpha_{L_0}$, which given our assumptions is equivalent to $\nu \geq 0$, and $c \in e^{-2\pi i \alpha_{L_0}} (\bR^+ \setminus \{0\})$. From this and its counterpart for $Z^+$, one concludes that the path $t \mapsto v(s,t) \in \bC^*$, for a fixed $\pm s \gg 0$, goes through an angle $2\pi(\alpha_{L_0}-\alpha_{L_1}+\nu-w)$, where $\nu \geq 0$ in the $Z^-$ case, and $\nu \leq -1$ in the $Z^+$ case. Consider the loop in $\bC$ that goes radially from $0$ to $v(s,0)$, then along $v(s,t)$, and finally radially back from $v(s,1)$ to $0$. Then, what we have just shown is that this loop has winding number $\nu-w$ around any sufficiently small $c<0$. Hence, the contribution of $\pm\infty$ to $u \cdot D_c$ is $\mp (\nu-w)$, which yields \eqref{eq:minimum-intersection}.
\end{proof}

Let's look briefly at Cauchy-Riemann equations with moving boundary conditions, along the same lines as in Section \ref{subsec:continuation}. On the region where \eqref{eq:z-trivialization} applies, projection to the $\bC$-factor leads to a version of \eqref{eq:v-equation} where the boundary condition has been modified as in \eqref{eq:moving-radial-lines}:
\begin{equation} \label{eq:moving-negative-lines}
\mybox{
$v(z) \in e^{-2\pi i \alpha_{\partial S}(z)} \bR^+$ for $z \in \partial S$.
}
\end{equation}
As already pointed out at the end of Section \ref{subsec:maps-to-c}, the local analysis from \eqref{eq:taylor} carries over, hence so does our definition of $\mu_z(u)$. 

\subsection{Global geometric setup\label{subsec:global-geometry}}
Let $M^{2n}$ be a closed symplectic manifold which is monotone, \eqref{eq:monotone}. Let $K_M^{-1}$ be its anticanonical bundle, which is the complex line bundle $\Lambda^n_{\bC}(\mathit{TM})$ with respect to some compatible almost complex structure. We assume that:
\begin{equation}
\mybox{
$K_M^{-1}$ comes with a section $Y$ which transversally cuts out a symplectic hypersurface $D \subset M$; and the symplectic orientation of $D$ should agree with that inherited as a zero-locus. Moreover, $K_M^{-1}|D$ comes with another section $Z$, which cuts out a symplectic hypersurface $B \subset D$ in the same way.
}
\end{equation}
Since $Y$ trivializes $K_M^{-1}$ away from $D$, we have $c_1(M \setminus D) = 0$. Similarly,  one can use the first derivative of $Y$ along $D$ to give an identification of the normal bundle (as an oriented real plane bundle) $\nu_D \iso K_M^{-1}|D$,
and this determines a trivialization (unique up to homotopy) of $K_D^{-1} \iso (K_M^{-1}|D) \otimes \nu_D^{-1}$; hence, $c_1(D) = 0$.
\begin{equation} \label{eq:normal-bundle}
\mybox{
Fix a metric and connection on $P = K_M^{-1}$, as in Section \ref{subsec:local-geometry}. Equip the total space of $P$ with the symplectic form \eqref{eq:omega-p}, and fix a symplectic diffeomorphism, for small $\epsilon>0$,
\[
\{\text{closed neighbourhood $\bar{U}_D \subset M$}\} \iso \{ \|y\| \leq \epsilon\} \subset P.
\]
}
\end{equation}
The existence of such a neighbourhood of $D$ follows from the normal form theorem for symplectic submanifolds. We use \eqref{eq:normal-bundle} to introduce a family of hypersurfaces $D_c$ deforming $D_0 = D$, for small $c \in \bC$, carried over from those of the same name in Section \ref{subsec:local-geometry}. In particular, $D_c \cap D = B$ for all $c \neq 0$. 

On $K_M^{-1}$ choose a metric, and a connection whose curvature is $-2\pi i \omega_M$. Outside the zero-section, this connection is determined by a one-form whose exterior derivative is the pullback of $\omega_M$. Let's use the section $Y$ to pull back that one-form to $\theta_{M \setminus D} \in \Omega^1(M \setminus D)$. This will satisfy $d\theta_{M \setminus D} = \omega_M$, and moreover:
\begin{equation} \label{eq:poincare}
\mybox{
If $S$ is a compact oriented surface with boundary, and $u: S \rightarrow M$ a map such that $u(\partial S) \cap D = \emptyset$, then 
\[
\int_S u^*\omega_M = (u \cdot D) + \int_{\partial S} \theta_{M \setminus D}.
\]
}
\end{equation}
Let's restrict $\theta_{M \setminus D}$ to $\bar{U}_D \setminus D$, and compare that under \eqref{eq:normal-bundle} with the one-form $\theta_{P \setminus D_0}$ from \eqref{eq:theta-p}. The difference between the two is a closed one-form $\delta$ on $\{ 0 < \|y\| \leq \epsilon\} \subset P$, which moreover has the property that its integral over any fibre circle vanishes. The last-mentioned property implies that the cohomology class of $\delta$ is pulled back from $H^1(D;\bR)$ by projection. By changing the choice of $\theta_D$ used to define $\theta_P$, one can achieve that $\delta$ is exact. Having done that, we modify $\theta_{M \setminus D}$ so that it agrees with $\theta_{P \setminus D_0}$ under \eqref{eq:normal-bundle}, while retaining \eqref{eq:poincare}; which is the primitive we will work with from now on. Combining this construction with \eqref{eq:z-trivialization} yields:
\begin{equation}
\label{eq:elsewhere}
\mybox{
Let $V_B \subset \bar{U}_D$ be the neighbourhood of $B$ which, in the tubular neighbourhood from \eqref{eq:normal-bundle}, corresponds to the part of $P$ lying over $U_B$ (see \eqref{eq:z-trivialization} for this notation). Then, we get a diffeomorphism
\[
\bar{U}_D \setminus V_B \iso \{|y|\leq \epsilon\} \times (D \setminus U_B) \subset \bC \times (D \setminus U_B),
\]
under which $\omega_M$ corresponds to $\omega_{\bC} + \omega_D$, and $\theta_{M \setminus D}$ to $\theta_{\bC} + \theta_D - \frac{1}{2\pi} d\mathrm{arg}(y)$.
}
\end{equation}

\begin{remark} \label{th:lefschetz-pencil}
The usual notion of a symplectic pencil involves a family $(D_c)_{c \in \bC \cup \{\infty\}}$ of symplectic hypersurfaces, which intersect along $B$. The Lefschetz condition says that all but finitely many $D_c$ are smooth, and prescribes the local structure of the singularities of the remaining ones \cite[Definition 1.4]{gompf04}. Since here we only consider the $D_c$ for $c \approx 0$, we can get away with assuming them to be smooth, and do not have to discuss singularities at all. Re-imposing the Lefschetz condition would be useful for further developments, since that jumpstarts the construction of Lagrangian submanifolds as Lefschetz thimbles, and thereby provides a source of ``sufficiently many'' objects in the Fukaya category (which is important for properties such as homological smoothness).
\end{remark}

As for the rest of the data:
\begin{equation} \label{eq:global-hamiltonian}
\mybox{
We fix an $H \in \smooth(M,\bR)$ whose restriction to $\bar{U}_D$ is of type \eqref{eq:local-hamiltonian} with respect to the identification \eqref{eq:normal-bundle}. For the associated vector field $X$, we additionally require that $X^{-1}(0) \subset M$ should have no interior points.
}
\end{equation}
\begin{equation} \label{eq:global-almost-complex}
\mybox{
We consider compatible almost complex structures $J$ on $M$ which, on $\bar{U}_D$, are of the class introduced in \eqref{eq:local-almost-complex}.
}
\end{equation}
\begin{equation} \label{eq:global-lagrangian}
\mybox{
We consider Lagrangian submanifolds $L \subset M$ with boundary $L_D = \partial L \subset D$. These should be such that  $L \cap \bar{U}_D$ is disjoint from $V_B$ and, under \eqref{eq:elsewhere}, corresponds to $[0,\epsilon] e^{-2\pi i \alpha_L} \times L_D$, for some $\alpha_L \in (-\half,\half)$. The interior $L \setminus \partial L$ must be exact with respect to $\theta_{M \setminus D}$, graded with respect to the given trivialization of $K_{M \setminus D}^{-1}$, and carry a {\em Spin} structure.
}
\end{equation}

\subsection{Intersection numbers revisited}
Let's look at the resulting class of equations \eqref{eq:cr}. Our setting is as follows:
\begin{equation} \label{eq:global-equation}
\mybox{
Let $S$ be a punctured-boundary Riemann surface with sub-closed one-form $\beta_S$ as in \eqref{eq:beta1}, with integral weights $w_\zeta \leq 0$ at the ends. The almost complex structures are taken from \eqref{eq:global-almost-complex}; the boundary conditions from \eqref{eq:global-lagrangian}, assuming moreover that $\alpha_{L_{\zeta,0}} > \alpha_{L_{\zeta,1}}$, and that all chords which can appear as limits are nondegenerate (this includes the assumption that the intersection $L_{\zeta,0,D} \cap L_{\zeta,1,D} = \partial L_{\zeta,0} \cap \partial L_{\zeta,1}$ inside $D$ is transverse); and the Hamiltonian is as in \eqref{eq:global-hamiltonian}.
}
\end{equation}
The observations from Section \ref{subsec:local-intersection} carry over to this context, leading to the definition of $\mu_z(u)$ and $\mu(u)$ for solutions not contained in $D$. From Lemma \ref{th:bound-intersection} we get:
\begin{equation} \label{eq:global-u-dot-dc}
\mybox{
Suppose that $u$ is not contained in $D$, and has limits $x_\zeta$. For small $c<0$,
\[
u \cdot D_c \geq \mu(u)  + \sum_{\zeta \in \Sigma^+} 
\begin{cases}
w_{\zeta}+1 & x_\zeta \in D, \\
0 & \text{otherwise}
\end{cases}
- \sum_{\zeta \in \Sigma^-} \begin{cases}
w_\zeta & x_\zeta \in D, \\
0 & \text{otherwise.}
\end{cases}
\]
If all $x_\zeta$ lie outside $D$, equality holds, meaning that $u \cdot D_c = \mu(u)$.
}
\end{equation}
Recall that by assumption, $B \subset D$ is an almost complex submanifold for each of our almost complex structures. Moreover, the vector field $X$ vanishes on $D$. This yields the following counterpart of \eqref{eq:global-u-dot-dc}:
\begin{equation} \label{eq:u-dot-dc-2}
\mybox{
Consider a solution $u$ which is contained in $D$. Then for small $c < 0$,
\[
u \cdot D_c \geq 0.
\]
Equality holds iff $u$ is contained in $D \setminus B$.
}
\end{equation}

Since our Lagrangian submanifolds are exact  in the complement of $D$, chords outside $D$ have well-defined actions. The local computations \eqref{eq:energy-defect-1} and \eqref{eq:energy-defect-2} yield the following:
\begin{equation} \label{eq:global-energy}
\mybox{
Suppose that $u$ is a solution whose limits lie outside $D$. Then
\[
E^{\mathit{top}}(u) = (u \cdot D) + \sum_{\zeta \in \Sigma_S^{\pm}} \mp A(x_\zeta).
\]
}
\end{equation}
Similarly, chords outside $D$ have well-defined indices, and we have the following modified  version of \eqref{eq:index-formula}:
\begin{equation} \label{eq:global-index}
\mybox{
Let $S$ be a boundary-punctured disc. Let $u$ be a solution whose limits lie outside $D$. Then, the linearization $D_u$ at that solution satisfies
\[
\mathrm{ind}(D_u) + \mathrm{ind}(x_d) + \cdots + \mathrm{ind}(x_1) = \mathrm{ind}(x_0) + 2 (u \cdot D).
\]
}
\end{equation}

%
%

\subsection{Definition of the algebraic structure\label{subsec:formal-c}}
The formal aspect of our construction proceeds along lines that are entirely parallel to those in Sections \ref{subsec:choice}--\ref{subsec:af}. We fix a countable set $\bfL$ of Lagrangian submanifolds \eqref{eq:global-lagrangian}, satisfying the same conditions as in \eqref{eq:collection}, \eqref{eq:collection-2}. For $\alpha_{L_0} > \alpha_{L_1}$, we define the Floer cochains using only chords outside $D$:
\begin{equation} \label{eq:modified-floer}
\mathit{CF}^*(L_0,L_1;w) = \bigoplus_{x^{-1}(D) = \emptyset} \mathit{or}(x).
\end{equation}
Next, we choose sub-closed one-forms \eqref{eq:asymptotically-consistent-one-forms} and almost complex structures \eqref{eq:floer-2}, \eqref{eq:cr-2}, within the class \eqref{eq:global-almost-complex}. We then define moduli spaces $\scrR^{d+1,p,w}(x_0,\dots,x_d)$ of stable weighted popsicles together with a solution of the relevant equation \eqref{eq:cr}, subject to the following additional constraints:
\begin{equation} \label{eq:constraints-1}
\mybox{
The limits $x_k$ are chords lying outside $D$.
}
\end{equation}
\begin{equation} \label{eq:constraints-2}
\mybox{
If $z \in S$ is a point where $k$ sprinkles are located, then $u$ intersects $D$ at that point with local intersection multiplicity $k$. Everywhere else, $u$ is disjoint from $D$. In other words,
\[
\mu_z(u) = \# \{f: \sigma_f = z\}.
\]
}
\end{equation}
The assumption \eqref{eq:constraints-2} implies that $u(\partial S) \cap D = \emptyset$. Together with \eqref{eq:constraints-1}, this shows that the intersection number $u \cdot D$ is well-defined, and that in fact (where as usual $c<0$ is small)
\begin{equation} \label{eq:total-intersection}
u \cdot D = u \cdot D_c = \mu(u) = |p|.
\end{equation}
Hence, \eqref{eq:global-energy} and \eqref{eq:e-e} can be used to obtain an upper bound on the energies of solutions. The term $2(u \cdot D_c) = 2(u \cdot D)$ in \eqref{eq:global-index} cancels out with the reduction in dimension effected by the evaluation constraints \eqref{eq:constraints-2}; hence, the expected dimension of our moduli space is still \eqref{eq:dimension-formula}. Moreover, the linearization of the constraints in \eqref{eq:constraints-2} involves the normal bundle to $D$, which comes with a canonical orientation. Hence, the version of the linearized operator which includes the constraints \eqref{eq:constraints-2} still satisfies \eqref{eq:or-1}. Assuming suitable compactness and transversality results (which we will provide below), one then defines maps $\mu^{d,p,w}$ as in \eqref{eq:sign-contributions-2}. 

The simplest form of the resulting algebraic construction is to form a category $\scrA^{\mathit{ord}}$, using only boundary-punctured discs with no sprinkles. By \eqref{eq:constraints-2}, this means that the maps $u$ remain inside $M \setminus D$. This $\scrA^{\mathit{ord}}$ sits inside a category $\scrF^{\mathit{ord}}$ whose morphism spaces are as in \eqref{eq:f-ord-morphisms}, and whose $A_\infty$-structure is obtained by adding up the $\mu^{d,p,w}$ with weights in $\{-1,0\}$. The strategy of proof of the $A_\infty$-relations is based on the same observations concerning popsicle moduli spaces as before. The main additional worry is the possibility that in the compactification, chords or entire components could lie in $D$; explaining why this does not happen will take up most of our subsequent discussion, Sections \ref{subsec:compactness-ex-1}--\ref{subsec:general-compactness}. 

Building on that, we can introduce positive perturbations as before \eqref{eq:positive-perturbation}, within our set $\bfL$ of Lagrangian submanifolds \eqref{eq:global-lagrangian}; and each such perturbation gives rise to a continuation cocycle \eqref{eq:continuation-cocycle}. Here, the Floer cochain groups are understood in the sense of \eqref{eq:modified-floer}, and the solutions of the continuation map equation are required to remain disjoint from $D$. Localizing with respect to such continuation cocycles yields $A_\infty$-categories $\scrA$ and $\scrF$. Finally, exactly as in Section \ref{subsec:divisor-1}, one can modify the algebraic formulation of $\scrF$ to obtain a noncommutative divisor $\scrD$, which satisfies the analogue of Proposition \ref{th:divisor-1}.

\begin{remark}
Our $\scrA$ is closely related to the category of the same name from Section \ref{subsec:af}, where the ``total space'' would be $M \setminus D$. The differences are partly ones of geometric setup, and partly technical, having to do with how exactly maps $u$ are prevented from escaping to infinity (here, we are using intersection numbers with the divisor $D$ at infinity; while the previous construction was based on the integrated maximum principle). One could relate the two versions rigorously, but we prefer not to do that. Instead, we will adapt arguments from the previous discussion to the present context on the single occasion where that becomes necessary (for Proposition \ref{th:restriction-functor-2}).
\end{remark}

To conclude this initial overview, we want to mention how our construction, when restricted to closed Lagrangian submanifolds, is related to the disc-counting invariant \eqref{eq:wl}. The connection is straightforward but can be useful, since there are well-developed techniques for computing such enumerative invariants. Let $L$ be a Lagrangian submanifold as in \eqref{eq:global-lagrangian}, but assumed to be closed, and hence actually an exact Lagrangian submanifold in $M \setminus D$. Take $\hat{S}$ to be the closed unit disc, and equip it with a generic family of almost complex structures as in \eqref{eq:global-almost-complex}. Consider holomorphic discs intersecting the divisor once, at a specified point $\hat\sigma \in \hat{S} \setminus \partial \hat{S}$:
\begin{equation} \label{eq:holo-disc}
\left\{
\begin{aligned}
& \hat{u}: \hat{S} \longrightarrow M, \\
& (D\hat{u})^{0,1} = 0, \\
& \hat{u}(\partial \hat{S}) \subset L, \\
& \hat{u}^{-1}(D) = \{\hat\sigma\}, \; \text{ and } \mu_{\hat\sigma}(\hat{u}) = 1.
\end{aligned}
\right.
\end{equation}
Transversality ensures that the moduli space of such maps is smooth, of dimension $n$, for a generic choice of almost complex structures; it is also oriented. Since these are the lowest energy holomorphic discs, the only possible bubbling is into a constant holomorphic disc with a holomorphic sphere attached, where that sphere has intersection number $1$ with $D$. One easily shows that this phenomenon generically has codimension $2$. Additionally, fix $\hat{z} \in \partial \hat{S}$. For a generic point of $L$, there are $W(L)$ solutions (counted with sign) of \eqref{eq:holo-disc} such that $\hat{u}(\hat{z})$ equals that point. Now consider the map which enters into the definition of the endomorphisms of $L$ as an object of $\scrF$. On the Floer cohomology level, this has the form
\begin{equation} \label{eq:not-continuation}
\mathit{HF}^*(L^+,L;-1) \longrightarrow \mathit{HF}^*(L^+,L;0),
\end{equation}
where $L^+$ is a positive perturbation of $L$. Since these are closed Lagrangian submanifolds, and the Floer cohomologies are taken in $M \setminus D$, the two Floer groups in \eqref{eq:not-continuation} are canonically isomorphic, by the continuation map (they are also canonically isomorphic to the ordinary cohomology of $L$).

\begin{prop} \label{th:disc-counting}
The map in \eqref{eq:not-continuation} is $W(L)$ times the continuation isomorphism.
\end{prop}

\begin{proof}[Sketch of proof]
By definition, the moduli space underlying \eqref{eq:not-continuation} consists of solutions of an equation \eqref{eq:cr} on $Z = \bR \times [0,1]$, where the boundary conditions are always $L$, and the one-form is asymptotic to zero as $s \rightarrow -\infty$, respectively to $-\mathit{dt}$ as $s \rightarrow \infty$. Moreover, the solutions must satisfy a special case of \eqref{eq:constraints-2}:
\begin{equation} \label{eq:1-sprinkle}
u^{-1}(D) = \{\sigma\}, \; \text{ and } \mu_{\sigma}(u) = 1;
\end{equation}
where the sprinkle $\sigma$ is a fixed point on the popsicle stick $Q = \bR \times \{\half\} \subset Z$. What we do is to construct a degeneration which moves $\sigma$ to a boundary point. Then, solutions degenerate into pairs $(u,\hat{u})$, where $\hat{u}$ is as in \eqref{eq:holo-disc}, and $u$ is a solution of the same equation as before, but where the intersection condition \eqref{eq:1-sprinkle} is replaced with $u^{-1}(D) = \emptyset$. The two components are related by having a common value at a specified boundary point. Counting maps $u$ by themselves just constructs the continuation isomorphism, and counting pairs $(u,\hat{u})$ multiplies that by $W(L)$.
\end{proof}

\subsection{Transversality\label{subsec:global-transversality}}
We now begin our discussion of the technical details underlying the construction from Section \ref{subsec:formal-c}. The counterpart of Lemma \ref{th:injective}, applied directly to the moduli spaces of interest here, is the following:

\begin{lemma} \label{th:injectivity-2}
For $(S,\sigma,u) \in \scrR^{d+1,p,w}(x_0,\dots,x_d)$, there is a point $z \in S$ where $u(z) \notin \bar{U}_D$ and $Du - X \otimes \beta_S \neq 0$.
\end{lemma}

\begin{proof}
By assumption, the chords $x_k$ lie outside $D$; because the Lagrangians involved have pairwise different angles $\alpha_{L_k}$, they must therefore lie outside $\bar{U}_D$. Hence, $U = \{z \in S\;:\; u(z) \notin \bar{U}_D\}$ is a nonempty subset of $S$. Suppose that $Du = X \otimes \beta_S$ holds on all of $U$. Since $\bar{U}_D$ is invariant under the flow of $X$ by assumption, the same argument as in Lemma \ref{th:injective} shows that $U$ is also closed, hence all of $S$. But the assumption that $Du = X \otimes \beta_S$ everywhere leads to a contradiction to the properties of our collection $\bfL$ of Lagrangian submanifolds, exactly as in Lemma \ref{th:transversality-2}.
\end{proof}

Outside $\bar{U}_D$, the conditions in \eqref{eq:global-almost-complex} do not constrain the almost complex structures. Hence, Lemma \ref{th:injectivity-2} easily leads to generic regularity for the moduli spaces $\scrR^{d+1,p,w}(x_0,\dots,x_d)$ (for the technical aspects of setting up transversality arguments for maps with tangency conditions, we refer to \cite[Section 6]{cieliebak-mohnke07}).

\subsection{Compactness by example: convergence and bubbling\label{subsec:compactness-ex-1}}
To complete the construction, we will need to discuss compactness issues for the relevant moduli spaces. In order to make this as transparent as possible, we will first work through a number of exercises which are special cases of the general argument. We assume that asymptotically consistent choices of one-forms and almost complex structures have been fixed.
\begin{equation} \label{eq:convergence}
\mybox{
Fix boundary conditions $(L_0,\dots,L_d)$ in $\bfL$, with $\alpha_{L_0} > \cdots > \alpha_{L_d}$, as well as weights $(w_0,\dots,w_d)$. Let $(S_k,\sigma_k,u_k)$ be a sequence in $\scrR^{d+1,p,w}(x_0,\dots,x_d)$, such that the underlying popsicle converges to some $(S,\sigma)$ in $\scrR^{d+1,p}$, and correspondingly, the maps $u_k$ converge uniformly to some $u: S \rightarrow M$. Note that then, $u$ has the same limits $(x_0,\dots,x_d)$, hence automatically satisfies \eqref{eq:constraints-1}.
}
\end{equation}

\begin{lemma} \label{th:exercise-1}
In the situation \eqref{eq:convergence}, the map $u$ again satisfies \eqref{eq:constraints-2}.
\end{lemma}

\begin{proof}
Consider some $z \in S \setminus \partial S$. Take a small closed disc $T \subset S$ around that point, such that $T \setminus \{z\}$ does not contain any sprinkles, and with $u^{-1}(D) \cap T$ containing only the original point. Then, the local multiplicity $\mu_z(u _\infty)$ is the intersection number $(u|T) \cdot D$. This  is a topological invariant, which can be computed from similar intersection numbers for $u_k$, $k \gg 0$. By \eqref{eq:constraints-2}, that intersection number equals the number of sprinkles on $S_k$ that converge to $z$ in the limit. Therefore, $\mu_z(u) = \#\{f \;:\; \sigma_f = z\}$, as desired.

Suppose that we have a $z \in \partial S$ such that $u(z) \in D$. Take a closed half-disc $T^+ \subset S$ around that point, which contains no sprinkles. Let $L$ be the relevant Lagrangian boundary condition. Since $u_k|T^+$ is disjoint from $D$, we have, in the notation from \eqref{eq:energy-defect-2},
\begin{equation} \label{eq:stokes-uk}
\int_{T^+} u_k^*\omega_P - \int_{\partial_{\mathit{in}}T^+} u_k^*\theta_{P \setminus D} - K_L(u_k(1)) - K_L(u_k(-1)) = 0
\end{equation}
by Stokes; passing to the limit and using \eqref{eq:energy-defect-2} yields $\mu_z(u) = 0$, a contradiction (we could have used a similar integral argument, appealing to the right hand side of \eqref{eq:energy-defect-1}, for interior points).
\end{proof}

Let's look at a slightly more complicated situation:
\begin{equation} \label{eq:bubbling}
\mybox{
Consider \eqref{eq:convergence}, with the following modification: $(u_k)$ converges to $u$ together with a single (sphere or disc) bubble $\hat{u}$.
}
\end{equation}

\begin{lemma} \label{th:exercise-2}
In \eqref{eq:bubbling}, the bubble is necessarily a sphere bubble, attached at a point $z \in S$ where at least one sprinkle is located. The limiting map $u$ satisfies \eqref{eq:constraints-2} except at $z$, where we instead have
\begin{equation} \label{eq:drop-tangency}
\mu_z(u) + \hat{u} \cdot D = \# \{f \;: \; \sigma_f = z\}.
\end{equation}
\end{lemma}

\begin{proof}
Compared to Lemma \ref{th:exercise-1}, only the behaviour of the limit $u$ at the bubbling point $z$ needs renewed discussion. One again starts with the right hand side of \eqref{eq:energy-defect-1} or \eqref{eq:energy-defect-2} for the maps $u_k$ and a suitable neighbourhood $T \subset S$ or $T^+ \subset S$. Passing to the limit now yields $\mu_z(u) + \int \hat{u}^*\omega_M = \# \{f \;:\; \sigma_f = z\}$, which is equivalent to the desired statement.
\end{proof}

Let's briefly consider the codimension in which this bubbling occurs. The principal component $(S,\sigma,u)$ is a solution of the same Cauchy-Riemann equation, and differs only in its tangency condition \eqref{eq:drop-tangency}, which does not affect the expected dimension. If we take into account the coincidence of sprinkles, we find that compared to the original moduli space $\scrR^{d+1,p,w}(x_0,\dots,x_d)$, the codimension of the space of such $(S,\sigma,u)$ is $\geq \# \{f \;:\; \sigma_f = z\}-1$. For us, only bubbling in codimension $<2$ matters, and the relevant cases are:
\begin{equation}
\mybox{
Suppose that $\# \{f\;:\; \sigma_f = z\} = 1$. Then, by \eqref{eq:drop-tangency} we necessarily have $\mu_z(u) = 0$ and $\hat{u} \cdot D = \hat{u} \cdot c_1(M) = 1$. While the space of such $(S,\sigma,u)$ has codimension $0$ compared to $\scrR^{d+1,p,w}(x_0,\dots,x_d)$, the condition that $u(z)$ should lie on a Chern number $1$ holomorphic sphere is of codimension $2$, assuming suitable transversality.
}
\end{equation}
\begin{equation}
\mybox{
Suppose that $\# \{f\;:\; \sigma_f = z\} = 2$. This of course implies that for the original popsicles, there were at least two sprinkles on the same popsicle stick. Our $A_\infty$-structure does not involve counting such popsicles. They do occur as components in boundary points of the compactification of one-dimensional moduli spaces, see \eqref{eq:cancel-1}, \eqref{eq:cancel-2}. However, in that case, the set of boundary points itself has dimension zero, hence the fact that the principal component of the bubble itself has positive codimension (because two sprinkles coincide) will be enough to rule out bubbling.
}
\end{equation}

\subsection{Compactness by example: splitting\label{subsec:compactness-ex-2}}
In the previously considered cases, property \eqref{eq:constraints-1} was obviously preserved under convergence. This is no longer obvious if the domain degenerates into a broken popsicle. We will now introduce the arguments that are necessary in order to deal with that. Here is the test case:
\begin{equation} \label{eq:two-component-splitting}
\mybox{
Take $(S_k,\sigma_k,u_k)$ as in \eqref{eq:convergence}, but now assume that the underlying popsicles converge to a point in a codimension one boundary stratum of $\bar\scrR^{d+1,p}$, corresponding to a broken popsicle with two components $(\bfS_v, \bfsigma_v)$ ($v = 1,2$; see \eqref{eq:two-vertex-tree} for the combinatorics). Correspondingly, the maps converge to a limit with components $\bfu_v: \bfS_v \rightarrow M$.
}
\end{equation}

\begin{lemma} \label{th:exercise-3}
In \eqref{eq:two-component-splitting}, the limit maps $\bfu_v$ again satisfy \eqref{eq:constraints-1}, \eqref{eq:constraints-2}.
\end{lemma}

\begin{proof} 
Convergence implies that (for small $c<0$)
\begin{equation} \label{eq:add-intersection-numbers}
\bfu_1 \cdot D_c + \bfu_2 \cdot D_c = u_k \cdot D_c, \;\; k \gg 0.
\end{equation}
We know that no component $\bfu_v$ is entirely contained in $D$, since at least one of its limiting chords lies outside $D$. The same argument as in Lemma \ref{th:exercise-1} shows that the $\bfu_v$ satisfy \eqref{eq:constraints-2}. The only issue we have to consider is whether the new limiting chord (common to both components, and which appears as the result of breaking up the surface) can lie in $D$. Suppose that this is the case. By applying \eqref{eq:global-u-dot-dc} to each component and adding up the outcome, we get
\begin{equation} \label{eq:add-c-intersections}
\bfu_1 \cdot D_c + \bfu_2 \cdot D_c \geq (\mu(\bfu_1) + \bfw_{1,i} + 1) + (\mu(\bfu_2) - \bfw_{2,0})
= |\bfp_1| + |\bfp_2| + 1 = |p| + 1.
\end{equation}
On the other hand, $u_k \cdot D_c = |p|$, which leads to a contradiction with \eqref{eq:add-intersection-numbers}.
\end{proof}

For more complicated broken popsicles, an additional issue appears: it is no longer a priori clear that components cannot entirely lie in $D$.
\begin{equation} \label{eq:three-component-splitting}
\mybox{
Let's modify the situation from \eqref{eq:two-component-splitting} by asking that the $(S_k,\sigma_k)$ should converge to a point in a codimension two boundary stratum of $\bar\scrR^{d+1,p}$. That stratum corresponds to a tree with three vertices. We ask that this should be a linear tree (there is a path from the root to a leaf going through all vertices), and label the vertices by $v \in \{1,2,3\}$, starting with the one next to the root. We ask that the second vertex should have valency $2$ and $|\bfp_2| = 1$. If the second vertex is attached to the $i$-th edge of the first vertex, the weights satisfy
\[
\bfw_{1,i} = \bfw_{2,0} = \bfw_{2,1} + 1 = \bfw_{3,0} + 1.
\]
Correspondingly, our limit point is described by a broken popsicle with three pieces $(\bfS_{v},\bfsigma_{v})$, where $\bfS_{2} \iso Z$ carries one sprinkle (see Figure \ref{fig:three-components}); and the limit of our sequence $(u_k: S_k \rightarrow M)$ will have three components $\bfu_{v}: \bfS_{v} \rightarrow M$.
}
\end{equation}
\begin{figure}
\begin{centering}
\begin{picture}(0,0)%
\includegraphics{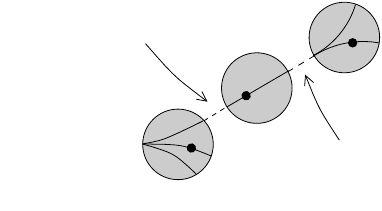}%
\end{picture}%
\setlength{\unitlength}{3355sp}%
\begingroup\makeatletter\ifx\SetFigFont\undefined%
\gdef\SetFigFont#1#2#3#4#5{%
  \reset@font\fontsize{#1}{#2pt}%
  \fontfamily{#3}\fontseries{#4}\fontshape{#5}%
  \selectfont}%
\fi\endgroup%
\begin{picture}(3581,1966)(-1342,-1638)
\put(-1327, 26){\makebox(0,0)[lb]{\smash{{\SetFigFont{10}{12.0}{\rmdefault}{\mddefault}{\updefault}{\color[rgb]{0,0,0}weight $\bfw_{1,i} = \bfw_{2,0} = \bfw_{2,1}+1$}%
}}}}
\put(1674,-1187){\makebox(0,0)[lb]{\smash{{\SetFigFont{10}{12.0}{\rmdefault}{\mddefault}{\updefault}{\color[rgb]{0,0,0}weight $\bfw_{2,1} = \bfw_{3,0}$}%
}}}}
\put(131,-1574){\makebox(0,0)[lb]{\smash{{\SetFigFont{10}{12.0}{\rmdefault}{\mddefault}{\updefault}{\color[rgb]{0,0,0}$\bfS_{1}$}%
}}}}
\put(1799,-552){\makebox(0,0)[lb]{\smash{{\SetFigFont{10}{12.0}{\rmdefault}{\mddefault}{\updefault}{\color[rgb]{0,0,0}$\bfS_3$}%
}}}}
\put(970,-1021){\makebox(0,0)[lb]{\smash{{\SetFigFont{10}{12.0}{\rmdefault}{\mddefault}{\updefault}{\color[rgb]{0,0,0}$\bfS_{2}$}%
}}}}
\end{picture}%
\caption{\label{fig:three-components}The degeneration from \eqref{eq:three-component-splitting}. The interesting case is when the image of $\bfS_{2}$ is entirely contained in $D$.}
\end{centering}
\end{figure}

\begin{lemma} \label{th:exercise-4}
In \eqref{eq:three-component-splitting}, the limit maps $\bfu_{v}$ satisfy \eqref{eq:constraints-1}, \eqref{eq:constraints-2}.
\end{lemma}

\begin{proof}
If none of the $\bfu_{v}$ is contained in $D$, we can apply the same kind of argument as in Lemma \ref{th:exercise-3}. Let's consider the only other possibility, namely, that $\bfu_2$ is contained in $D$. By \eqref{eq:u-dot-dc-2},
\begin{equation}
\bfu_1 \cdot D_c + \bfu_{3} \cdot D_c \leq \bfu_{1} \cdot D_c + \bfu_{2} \cdot D_c + \bfu_{3} \cdot D_c = u_k \cdot D_c, \;\; k \gg 0.
\end{equation}
Applying \eqref{eq:global-u-dot-dc} to the components $\bfu_{v}$ for $v \neq 2$, we get
\begin{equation}
\begin{aligned}
\bfu_{1} \cdot D_c + \bfu_{3} \cdot D_c & \geq \mu(\bfu_{1}) + \mu(\bfu_{3}) + \bfw_{1,i} + 1- \bfw_{3,0}
\\ & = |\bfp_1| + |\bfp_3| + \bfw_{1,i} - \bfw_{3,0} + 1 = |\bfp_1| + |\bfp_2| +| \bfp_3| + 1 = |p| + 1,
\end{aligned}
\end{equation}
which again leads to a contradiction.
\end{proof}

\subsection{The general compactness argument\label{subsec:general-compactness}}
Consider a sequence $(S_k,\sigma_k, u_k)$ in a moduli space $\scrR^{d+1,p,w}(x_0,\dots,x_d)$, assumed to be Gromov convergent, to a limit of the following general type.
\begin{equation} \label{eq:possible-limit-1}
\mybox{
The combinatorial structure of the limit is governed by a broken popsicle of type $(T,\bfp,\bfr)$. For each vertex $v$ of the tree $T$ we have an extended popsicle, which means a popsicle $(\bfS_v, \bfsigma_v)$ with finitely many bubble trees $\hat{\bfS}_{v,m}$ attached at (pairwise distinct) points $\bfz_{v,m} \in \bfS_v$. More precisely, if $\bfz_{v,m}$ is an interior point, the bubble tree is a nodal genus zero curve (a finite collection of spheres, forming a tree), with a marked point $\hat{\bfz}_{v,m}$; for a boundary point, it is a nodal disc (a finite collection of discs and spheres, forming a tree, where the discs are attached to each other along boundary points, and the spheres are attached along interior points), again with marked point $\hat{\bfz}_{v,m}$.
}
\end{equation}\begin{equation}
\label{eq:possible-limit-1b}
\mybox{
Each component of the popsicle comes with a solution $\bfu_v: \bfS_v \rightarrow M$ of the relevant Cauchy-Riemann equation. The limits of these maps along the semi-infinite edges are the given $(x_0,\dots,x_d)$; the other limits can be any chords of the length given by the relevant weights (obviously, for each finite edge, the limits from the two adjacent components must coincide). Each bubble tree comes with a map $\hat{\bfu}_{v,m}: \hat{\bfS}_{v,m} \rightarrow M$, which is pseudo-holomorphic with respect to the almost complex structure associated to the weighted popsicle $(\bfS_v,\bfsigma_v,\bfw_v)$ at the point $\bfz_{v,m}$; similarly, if $\bfz_{v,m}$ is a boundary point, the Lagrangian boundary condition at that point is inherited by the disc components of the bubble tree. We also have the connectedness conditions $\hat{\bfu}_{v,m}(\hat{\bfz}_{v,m}) = \bfu_{v}(\bfz_{v,m})$. Finally, the entire limit must satisfy the standard stability condition (automorphism group is finite). Besides the classical condition on each bubble tree, this means the following: if $v$ is a two-valent vertex with $|\bfp_v| = 0$, then either the map $\bfu_v$ is a non-constant Floer trajectory, or it can be constant with at least one bubble attached.
}
\end{equation}
In \eqref{eq:possible-limit-1b} we have omitted any mention of incidence conditions at the sprinkles (indeed, what we have described is the Gromov compactification after forgetting those incidence conditions). Hence, the question of what limits actually occur requires further discussion.

\begin{lemma} \label{th:big-compactness-1}
Take a component $\bfS_v$ such that $\bfu_v$ is not entirely contained in $D$. Then, the bubble trees attached to it contain only spheres, and at least one sprinkle is located at each of the attaching points $\bfz_{v,m}$. The map $\bfu_v$ satisfies \eqref{eq:constraints-2} except at those points, where instead we have
\begin{equation} \label{eq:weaker-tangency}
\mu_{\bfz_{v,m}}(\bfu_v) + \hat{\bfu}_{v,m} \cdot D = \# \{f : \bfsigma_{v,f} = \bfz_{v,m}\}.
\end{equation}
\end{lemma}

This is proved exactly as in Lemma \ref{th:exercise-2}.

\begin{lemma} \label{th:big-compactness-2}
All chords which appear as limits of the components $\bfu_v$ in \eqref{eq:possible-limit-1b} lie outside $D$. In particular, no component can be contained in $D$.
\end{lemma}

\begin{proof}
Cut open every finite edge of $T$ whose corresponding chord lies in $D$, so as to produce two semi-infinite ones. The outcome is a finite collection of sub-trees $T^\gamma$, which inherit the same data as $T$ itself. We will be interested in these quantities:
\begin{equation}
|p^\gamma| \stackrel{\mathrm{def}}{=} \sum_{v \in \mathit{Ve}(T^\gamma)} |\bfp_v|,
\end{equation}
where the sum is over vertices of $T^\gamma$ (in words, $|p^\gamma|$ is the number of sprinkles ``belonging to'' $T^\gamma$); and
\begin{equation} \label{eq:big-delta-nu}
\Delta^\gamma \stackrel{\mathrm{def}}{=} \sum_{v \in \mathit{Ve}(T^\gamma)} \Big( \bfu_v \cdot D_c + \sum_m\hat{\bfu}_{v,m} \cdot D_c \Big) +
\begin{cases} w^\gamma_0 & \text{if $x^\gamma_0 \in D$,} \\
0 & \text{otherwise}
\end{cases}
- 
\sum_{j>0}
\begin{cases}
w^\gamma_j & \text{if $x^\gamma_j \in D$,} \\
0 & \text{otherwise,}
\end{cases}
\end{equation}
where the $w^\gamma_j$ are the weights associated to the semi-infinite edges of $T^\gamma$, and $x^\gamma_j$ the corresponding chords. Adding up \eqref{eq:big-delta-nu} produces (via pairwise cancellation of the weight terms)
\begin{equation} \label{eq:sum-delta-gamma}
\sum_\gamma \Delta^\gamma = \sum_{v \in \mathit{Ve}(T)} \Big(\bfu_v \cdot D_c + \sum_m \hat{\bfu}_{v,m} \cdot D_c \Big).
\end{equation}
Because of Gromov convergence, the right hand side of \eqref{eq:sum-delta-gamma} equals $u_k \cdot D_c$, $k \gg 0$, which by \eqref{eq:constraints-2} is $|p|$; hence,
\begin{equation} \label{eq:delta-is-p}
\sum_{\gamma} \Delta^\gamma = \sum_{\gamma} |p^\gamma|.
\end{equation}
Let's examine each of the \eqref{eq:big-delta-nu} in more detail.
\begin{equation} \label{eq:big-delta-1}
\mybox{
$T^\gamma$ may be a single-vertex tree where the map $\bfu_v$ is contained in $D$. In that case, \eqref{eq:u-dot-dc-2} implies that 
\[
\Delta^\gamma \geq w_0^\gamma - \sum_{j>0} w_j^\gamma = |p^\gamma|.
\]
}
\end{equation}
\begin{equation} \label{eq:big-delta-2}
\mybox{
Consider any other sub-tree, which by definition has only components $\bfu_v$ not contained in $D$. By applying \eqref{eq:global-u-dot-dc} and Lemma \ref{th:big-compactness-1} to each vertex of $T^\gamma$, one sees that
\[
\Delta^\gamma \geq 
\sum_{v \in \mathit{Ve}(T^\gamma)} \Big( \mu(\bfu_v) + \sum_m \hat{\bfu}_{v,m} \cdot D \Big) = |p^\gamma|.
\] 
Moreover, because of the additional $+1$ in \eqref{eq:global-u-dot-dc}, equality holds if and only if all chords corresponding to edges of our sub-tree, except possibly for that associated to the root, lie outside $D$.
}
\end{equation}
Adding up the results from \eqref{eq:big-delta-1} and \eqref{eq:big-delta-2} shows that $\sum_\gamma \Delta^\gamma \geq \sum_{\gamma} |p^\gamma|$. Moreover, because of the last observation in \eqref{eq:big-delta-2}, equality can hold only if all the chords involved lie outside $D$; comparing that with \eqref{eq:delta-is-p} yields the desired result.
\end{proof}

At this point, we add transversality arguments to our discussion. Every principal component $\bfu_v: \bfS_v \rightarrow M$ of our limit belongs to a moduli space which is roughly of the same general kind as the original maps, except that it possibly comes with a condition \eqref{eq:weaker-tangency} prescribing a lower order of tangency. We also want to take into account that the underlying popsicle $(\bfS_v,\bfsigma_v)$ may have several coincident sprinkles, hence belongs to the subspace $\scrR^{\bfd_v+1,\bfp_v}_{\bfC_v}$ of popsicles fixed by some conjugacy class of subgroups $\bfC_v$ in $\mathit{Aut}(\bfp_v)$. The virtual dimension is unaffected by orders of tangency; hence, for such a $\bfu_v$, it is given by an expression similar to \eqref{eq:dimension-formula},
\begin{equation} \label{eq:vertex-expression}
\mathrm{dim}(\scrR^{\bfd_v+1,\bfp_v}_{\bfC_v}) + \mathrm{ind}(\bfx_{v,0}) - \textstyle \sum_{j>0}
\mathrm{ind}(\bfx_{v,j}),
\end{equation}
wgere the $\bfx_{v,j}$ are the limits of $\bfu_v$. Adding up this over all $v$, we find that the expected codimension of the entire collection $\{\bfu_v\}$, with respect to the original space $\scrR^{d+1,p,w}(x_0,\dots,x_d)$, is
\begin{equation} \label{eq:codimension}
|\mathrm{Ve}(T)| - 1 + \sum_v \mathrm{codim}( \scrR^{\bfd_v+1,\bfp_v}_{\bfC_v} \subset \scrR^{\bfd+1,\bfp_v}).
\end{equation}
The same arguments as in Section \ref{subsec:global-transversality} show generic regularity for moduli spaces of $(\bfS_v,\bfsigma_v,\bfu_v)$. Together with transversality for holomorphic spheres (as in Lemma \ref{th:local-transversality}), this also shows that the space of such $(\bfS_v,\bfsigma_v,\bfu_v)$ together with a bubbles $\hat{\bfu}_{v,m}$ satisfying $\hat{\bfu}_{v,m} \cdot D = 1$, and not contained in $D$, are generically regular; each such bubble causes the dimension to drop by $2$. Let's assume from now that the (asymptotically consistent) almost complex structures have been chosen so that all this holds.

\begin{lemma} \label{th:big-compactness-3}
Suppose that the original moduli space $\scrR^{d+1,p,w}(x_0,\dots,x_d)$ is zero-dimensional. Then all limits lie in that space (it is compact with respect to Gromov convergence).
\end{lemma}

\begin{proof}
The codimension in \eqref{eq:codimension} is necessarily zero. Hence, we have a limit with a single vertex, and where the symmetry group $C = \{\mathit{id}\}$ is trivial, meaning that the limiting popsicle has pairwise distinct sprinkles. In that case, Lemma \ref{th:big-compactness-1} implies that the only possible bubbling goes as follows: we have sphere bubbles $\hat{u}_k$ with $\hat{u}_k \cdot D = 1$, at most one per sprinkle, and where the attaching points $\hat{u}_k(\hat{z}_k)$ lie outside $D$. Such bubbles can't be entirely contained in $D$, hence are covered by our transversality-of-evaluation result. This shows that bubbling is a codimension $\geq 2$ phenomenon, hence ruled out in our case.  
\end{proof}

\begin{lemma} \label{th:big-compactness-4}
Suppose that the original moduli space $\scrR^{d+1,p,w}(x_0,\dots,x_d)$ is one-dimensional, and that the weights $(w_0,\dots,w_d)$ involved lie in $\{-1,0\}$. Then, all the possible limits either lie in $\scrR^{d+1,p,w}(x_0,\dots,x_d)$ itself, or else involve the popsicle splitting into two pieces, with no sphere bubbling.
\end{lemma}

\begin{proof}
If \eqref{eq:codimension} is $0$, we are in the same situation as in Lemma \ref{th:big-compactness-3}, and the argument there suffices to rule out bubbling, ensuring that the limit lies in $\scrR^{d+1,p,w}(x_0,\dots,x_d)$ itself. Let's focus on the remaining case, where \eqref{eq:codimension} is $1$. In principle, that could be achieved by having a single vertex, where the associated popsicle has a $\bZ/2$-symmetry. However, that would require two sprinkles to coincide, which is excluded by the condition on the weights. The other option is to have a tree with two vertices, where each of the two components is a popsicle having pairwise distinct sprinkles. In that case, bubbling can again be ruled out as in Lemma \ref{th:big-compactness-3}.
\end{proof}

Together, these results cover all the moduli spaces that appear in the definition of $\scrF^{\mathit{ord}}$ in Section \ref{subsec:formal-c}. Localising to $\scrF$ additionally involves continuation cocycles, which are derived from Cauchy-Riemann equations with moving boundary conditions, but considering only solutions which do not intersect $D$ at all. Compactness for such maps is an easy argument involving intersection multiplicities, which (as pointed out at the end of Section \ref{subsec:local-intersection}) generalize to this case. Details are left to the reader.

\subsection{A variant\label{subsec:variant}}
Let's generalize the setup from Section \ref{subsec:global-geometry} by asking for $Y$ and $Z$ to be sections of $K_M^{-N}$ for some $N>0$. This primarily affects the index formulae: in the situation of \eqref{eq:global-index}, we have
\begin{equation}
\mathrm{ind}(D_u) + \mathrm{ind}(x_d) + \cdots + \mathrm{ind}(x_1) = \mathrm{ind}(x_0) + 2N (u \cdot D_c);
\end{equation}
in view of \eqref{eq:constraints-2}, the dimension formula \eqref{eq:dimension-formula} then becomes
\begin{equation}
\mathrm{dim}(\scrR^{d+1,p,w}(x_0,\dots,x_d)) =  d - 2 + (2N-1)|p| + \mathrm{ind}(x_0) - \mathrm{ind}(x_d) - \cdots - \mathrm{ind}(x_1).
\end{equation}
The implication is that we can still define a $\bZ$-graded $A_\infty$-category $\scrF^{\mathit{ord}}$, but with nontrivial morphism spaces
\begin{equation}
\scrF^{\mathit{ord}}(L_0,L_1) = \mathit{CF}^*(L_0,L_1) \oplus \mathit{CF}^{*+2N-1}(L_0,L_1;-1),  \quad \alpha_{L_0} > \alpha_{L_1}. 
\end{equation}
The degree of the continuation cocycles is unaffected by this, since they lie in the first summand, and are defined in terms of holomorphic curves in $M \setminus D$. Hence, one can localize to $\scrF$ as before. The analogue of Section \ref{subsec:divisor-1} would say that there is a graded noncommutative divisor $\scrD$ in the sense of \eqref{eq:graded-linear-system}, over $\bZ[v]$ with grading $|v| = 2N-2$, and whose fibre is quasi-equivalent to $\scrF$.

Concerning technical aspects, the only significant difference appears in the last step of the compactness argument, where the codimension formula \eqref{eq:codimension} should be replaced by
\begin{equation} \label{eq:codimension-2}
|\mathrm{Ve}(T)| - 1 + \sum_v \mathrm{codim}( \scrR^{\bfd_v+1,\bfp_v}_{\bfC_v} \subset \scrR^{\bfd+1,\bfp_v}) + (2N-2) \sum_{v,m} (\hat{u}_{v,m} \cdot D).
\end{equation}
The previously discussed case is $N = 1$. For $N>1$, the additional term in \eqref{eq:codimension-2} means that any bubbling further increases codimension, hence making our task easier: transversality considerations for pseudo-holomorphic spheres as in the proofs of Lemmas \ref{th:big-compactness-3}, \ref{th:big-compactness-4} become unnecessary. 

\section{The noncommutative pencil\label{sec:flavour}}
We have used popsicles in two different ways: in Section \ref{sec:first-construction}, the position of the marked points (sprinkles) were used as parameters (moduli) to govern the choice of one-forms, but otherwise, had no geometric meaning in relation to the target manifold; while in Section \ref{sec:anticanonical}, we imposed intersection conditions at those points. We now combine the two ideas, while remaining in the same general situation as in Section \ref{sec:anticanonical}. The outcome is the noncommutative pencil announced in Construction \ref{th:1} (and its variant, Construction \ref{th:2}).

\subsection{Flavoured popsicles}
We introduce a version of popsicles which carry sprinkles of two flavours (``$\va$nilla'' and ``$\ch$ocolate''). 
\begin{equation}
\mybox{
Take $d \geq 1$ and $m_\va \geq 0$, $m_\ch \geq 0$, as well as nondecreasing maps $p_\va: \{1,\dots,m_\va\}  \rightarrow \{1,\dots,d\}$, $p_\ch: \{1,\dots,m_\ch\} \rightarrow \{1,\dots,d\}$. We usually write $|p_\va| = m_\va$, $|p_\ch|= m_\ch$. A flavoured popsicle of type $(d,p_\va,p_\ch)$ is a disc with $(d+1)$ boundary punctures $S$, together with sprinkles $\sigma_{\va,f} \in Q_{p_\va(f)}$ ($1 \leq f \leq m_\va$) and $\sigma_{\ch,f} \in Q_{p_\ch(f)}$ ($1 \leq f \leq m_\ch$). 
}
\end{equation}
Stable flavoured popsicles form a moduli space $\scrR^{d+1,p_\va,p_\ch}$, which is a copy of $\scrR^{d+1,p}$ for a suitable $p$ with $|p| = |p_\va| + |p_\ch|$. To be precise, the identification between those moduli spaces, which forgets the distinction between the flavours, is not quite canonical: it requires one to choose a bijection
\begin{equation} \label{eq:forget-flavours}
\{1,\dots,|p|_\va\} \sqcup \{1,\dots,|p|_\ch\} \iso \{1,\dots,|p|\}
\end{equation}
so that the map $p$ obtained from $p_\va$ and $p_\ch$ is nondecreasing. On a related note, the natural automorphism group $\Aut(p_\va, p_\ch)$ which acts on $\scrR^{d+1,p_\va,p_\ch}$ can be smaller than $\Aut(p)$, since it allows only permutations that preserve flavours. Similarly, we have broken flavoured popsicles. Compared to the original definition \eqref{eq:broken-popsicle-data}, these also come with two labeling maps for each vertex, $\bfr_{\va,v}: \{1,\dots,\bfm_{\va,v}\} \rightarrow \{1,\dots,m_\va\}$ and $\bfr_{\ch,v}: \{1,\dots,\bfm_{\ch.v}\} \rightarrow \{1,\dots,m_\ch\}$. This gives rise to a  compactification $\bar\scrR^{d+1,p_\va,p_\ch}$, which is again a copy of $\bar\scrR^{d+1,p}$.
One can also introduce weights as in Section \ref{subsec:weights}. As usual, our attention is focused on the case where the weights lie in $\{-1,0\}$. The counterpart of Lemma \ref{th:wrong-weights} is this:

\begin{lemma} \label{th:wrong-weights-2}
Take a space $\bar\scrR^{d+1,p_\va,p_\ch,w}$ where the weights lie in $\{-1,0\}$. Consider a boundary stratum of codimension $1$, such that the induced weights for the components of broken popsicles in that stratum do not lie in $\{-1,0\}$. Then, in the notation from \eqref{eq:two-vertex-tree}, one of the three following applies:
\begin{equation} \label{eq:more-symmetry-2}
\mybox{$\# \bfp_{1,\va}^{-1}(i) \geq 2$ or $\# \bfp_{1,\ch}^{-1}(i) \geq 2$; equivalently, $\Aut(\bfp_{1,\va}, \bfp_{1,\ch})$ is nontrivial.}
\end{equation}
\begin{equation} \label{eq:two-flavour-stick}
\mybox{
$\bfp_{1,\va}^{-1}(i) = \{f_\va\}$ and $\bfp_{1,\ch}^{-1}(i) = \{f_\ch\}$ both consist of one element. Then, there are $\{k_1, k_2\}$ as in \eqref{eq:switch-sprinkle}; $p_{\va}(\bfr_{\va,1}(f_\va))$ is one of them, and $p_{\ch}(\bfr_{\ch,1}(f_\ch))$ is the other one.
}
\end{equation}
\begin{equation} \label{eq:switch-sprinkle-2}
\mybox{$\bfw_{1,0} = w_0 =  -1$, $\bfw_{1,i} = \bfw_{2,0} = -2$, $\bfp_{1,\va}^{-1}(i) = \{f\}$ consists of one element, and $\bfp_{1,\ch}^{-1}(i) = \emptyset$ (or vice versa). In that case, $p(\bfr_{1,\va}(f))$ is one of the $\{k_1,k_2\}$ from \eqref{eq:switch-sprinkle}.
}
\end{equation}
\end{lemma}

We make an asymptotically consistent choice of one-forms on the universal families of stable flavoured popsicles, exactly as in \eqref{eq:asymptotically-consistent-one-forms} (one could pull back this choice from $\scrR^{d+1,p,w}$ by the flavour-forgetting map, but there is no particular reason to impose such a restriction). 

\subsection{Definition of the algebraic structure}
We work in the setup from Section \ref{subsec:formal-c}, and consider the same kind of collection $\bfL$ of Lagrangian submanifolds. We make an asymptotically consistent choice of families of almost complex structures, in the usual sense of \eqref{eq:floer-2}, \eqref{eq:cr-2}, for each stable flavoured popsicle with boundary conditions $(L_0,\dots,L_d)$ taken from $\bfL$, with $\alpha_{L_0} > \cdots > \alpha_{L_d}$. We consider moduli spaces $\scrR^{d+1,p_\va,p_\ch,w}(x_0,\dots,x_d)$, where: the domains are flavoured popsicles; the limits again satisfy \eqref{eq:constraints-1}; and \eqref{eq:constraints-2} applies only to the chocolate flavour:
\begin{equation} \label{eq:constraints-3}
\mu_z(u) = \# \{f: \sigma_{\mathit{ch},f} = z\}.
\end{equation}
Counting points in zero-dimensional moduli spaces yields operations $\mu^{d,p_\va,p_\ch,w}$ as in \eqref{eq:mu-p}, where the Floer groups are those defined in \eqref{eq:modified-floer}. Consider a two-variable symmetric algebra $\bZ[v_\va,v_\ch]$. In a preliminary version of the noncommutative pencil along the lines of \eqref{eq:d-ord-morphisms}, we set
\begin{equation}
\scrP^{\mathit{ord}}(L_0,L_1) = \begin{cases}
\mathit{CF}^*(L_0,L_1)[v_\va,v_\ch] \oplus \mathit{CF}^{*+1}(L_0,L_1;-1)[v_\va,v_\ch] & \alpha_{L_0} > \alpha_{L_1}, \\
\bZ[v_\va,v_\ch] e & L_0 = L_1, \\
0 & \text{otherwise.}
\end{cases}
\end{equation}
Here, $\mathit{CF}^*(L_0,L_1)v_{\va}^{r_\va} v_\ch^{r_\ch}$ has weight $-r_\va-r_\ch$, and $\mathit{CF}^{*+1}(L_0,L_1;-1)v_{\va}^{r_\va} v_\ch^{r_\ch}$ weight $-r_\va-r_\ch-1$. The $A_\infty$-structure $\mu^d_{\scrP^{\mathit{ord}}}$ is defined by adding up $v_\va^{|p_\va|} v_\ch^{|p_\ch|} \mu^{d,p_\va,p_\ch,w}$ over all $(p_\va,p_\ch,w)$ where the weights lie in $\{-1,0\}$. Transversality and compactness arguments repeat those from Section \ref{subsec:global-transversality} and \ref{subsec:general-compactness}). A minor tweak is required in the proof of the $A_\infty$-relations, because of the extra case \eqref{eq:two-flavour-stick} in Lemma \ref{th:wrong-weights-2}, but the strategy follows that from \eqref{eq:cancel-2b}: each such degeneration appears twice as a boundary point of our moduli spaces, with different combinatorics (switching $k_1,k_2$; see Figure \ref{fig:two-flavours} for the simplest example). 
\begin{figure}
\begin{centering}
\begin{picture}(0,0)%
\includegraphics{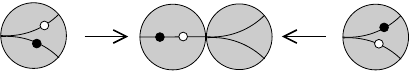}%
\end{picture}%
\setlength{\unitlength}{3355sp}%
\begingroup\makeatletter\ifx\SetFigFont\undefined%
\gdef\SetFigFont#1#2#3#4#5{%
  \reset@font\fontsize{#1}{#2pt}%
  \fontfamily{#3}\fontseries{#4}\fontshape{#5}%
  \selectfont}%
\fi\endgroup%
\begin{picture}(3854,648)(-196,-1021)
\end{picture}%
\caption{\label{fig:two-flavours}A broken flavoured popsicle as in \eqref{eq:two-flavour-stick}, appearing as the boundary point of two moduli spaces. Both of those spaces have $|p_\va| = |p_\ch| = 1$. On the right, $p_\va(1) = 1$, $p_\ch(1) = 2$; while the values are switched on the left.}
\end{centering}
\end{figure}

To complete our argument, we need to return briefly to sign considerations. The definition of $\mu^{d,p_\va,p_\ch,w}$ requires a choice of orientation of $\scrR^{d,p_\va,p_\ch}$. The only case that matters to us is where at most one of the popsicles lies on each popsicle stick. Then, there is a unique suitable choice of \eqref{eq:forget-flavours}, and hence a unique identification $\scrR^{d+1,p_\va,p_\ch} \iso \scrR^{d+1,p}$, which we use to pull back the previously chosen orientation of $\scrR^{d+1,p}$. After that, we insert further signs as in \eqref{eq:sign-contributions-2}. The only necessary additional argument concerns \eqref{eq:two-flavour-stick}: 
\begin{equation}
\mybox{
Consider a point of type \eqref{eq:two-flavour-stick} in a codimension one boundary face of a moduli space $\bar\scrR^{d,p_\va,p_\ch,w}$, where the weights lie in $\{-1,0\}$. After exchanging $p_{\va}(\bfr_{\va,1}(f_\va))$ and $p_{\ch}(\bfr_{\ch,1}(f_\ch))$, this becomes a boundary point of a different moduli space of the same kind. If we forget the flavours (meaning that we look at the identification $\bar\scrR^{d,p_\va,v_\ch,w} \iso \bar\scrR^{d,p}$), this change involves a transposition in the ordering of two popsicles. Hence, the two moduli spaces inherit opposite orientations from that identification. This difference of orientation carries over to $\bar\scrR^{d,p_\va,p_\ch,w}(x_0,\dots,x_d)$, which ensures that boundary contributions where the popsicles degenerate in this way cancel in pairs.
}
\end{equation} 

The resulting $\scrP^{\mathit{ord}}$ is a weight-graded $A_\infty$-category over $\bZ[v_\va,v_\ch]$. By construction, the weight zero part is the same $\scrA^{\mathit{ord}}$ as in Section \ref{subsec:formal-c}, hence contains continuation cocycles. We localize with respect to those, then consider the subcomplex generated in weights $\{-1,0\}$, and apply \eqref{eq:weight01}. The outcome is the noncommutative pencil $\scrP$ whose existence was announced in Construction \ref{th:1}. Note that even though the algebraic structure treats the variables $v_\va$ and $v_\ch$ symmetrically, they have different geometric meaning. We take this into account by slightly tweaking the terminology:
\begin{equation} \label{eq:terminology-fibres}
\mybox{
We write $\scrF_w$, for $w \in \bZ$, for the fibre of $\scrP$ obtained by setting $(v_\va,v_\ch) = (w,1)$. Similarly, the fibre $\scrF_\infty$ is that obtained by setting $(v_\va,v_\ch) = (1,0)$.
}
\end{equation}
One can treat the situation from Section \ref{subsec:variant} in the same way: $v_\va$ still has degree zero, but $v_\ch$ has degree $2N-2$.

Setting $v_\va = 0$ in $\scrP$ yields a noncommutative divisor, which by definition is exactly that from Section \ref{sec:anticanonical}. In particular, the fibre $\scrF_0$ is the category $\scrF$ defined there. It remains to consider the opposite extreme, which means the fibre $\scrF_\infty$, obtained by setting $v_\ch = 0$. By construction, this means that the maps $u$ involved in its definition remain in $M \setminus D$. The desired counterpart of Proposition \ref{th:restriction-functor} is:

\begin{proposition} \label{th:restriction-functor-2}
The fibre $\scrF_\infty$ of the noncommutative pencil is quasi-equivalent to the full subcategory the Fukaya category of $D \setminus B$ whose objects are $\partial L$, for $L$ a Lagrangian submanifold in our collection $\bfL$ (with the induced gradings and {\em Spin} structures). 
\end{proposition}

We will not go through the entire proof, since the idea is largely as before, but we will explain how to construct the crucial geometric ingredients, which underlie a partial energy argument as in Section \ref{subsec:partial-energy}. Note that Proposition \ref{th:restriction-functor-2} completes the proof of the claims concerning our noncommutative pencil made in Constructions \ref{th:1} and \ref{th:2}.

\subsection{The local geometry, revisited}
Let's return to the situation from Section \ref{subsec:local-geometry}. We make the following slight modifications of the basic geometric objects considered there:
\begin{equation}
\mybox{Within the class of Hamiltonians as in \eqref{eq:local-hamiltonian}, we specify the choice $H = -\|y\|^2$.
}
\end{equation}
We apply a deformation of $Z$ as in Remark \ref{th:local-almost-complex-2}, and then rescale it by large constant (which has the effect of bringing $D_c$, for fixed value of $c$, closer to $D$). After that, we can use Remark \ref{th:local-almost-complex-2} to show that there are almost complex structures satisfying the following additional condition:
\begin{equation} \label{eq:local-almost-complex-2}
\mybox{
For the almost complex structures from \eqref{eq:local-almost-complex}, we additionally ask that $D_c$, for $\{5/2 < |c| < 5\}$, should be almost complex submanifolds.
}
\end{equation}
The counterpart of the basic object from \eqref{eq:tilde} in our context is:
\begin{equation}
\mybox{
Take $\tilde{\omega}_{\bC} \in \Omega^2(\bC)$ as in \eqref{eq:tilde}. We define $\tilde{\omega}_{P \setminus D}$ to be the pullback of $\tilde{\omega}_{\bC}$ under
\[
\Pi(x,y) \mapsto Z_x/y: P \setminus D \longrightarrow \bC.
\]
}
\end{equation}

\begin{lemma}
Let $J$ be as in \eqref{eq:local-almost-complex-2}, and $x$ a point in $D_c \setminus B$, for $5/2 < |c| < 5$. On $TP_x$, the form $\tilde{\omega}_{P \setminus D}(\cdot, J\cdot)$ is symmetric, nonnegative, and 
\begin{equation}
\tilde{\omega}_{P \setminus D}(X,JX) = 0 
\Longleftrightarrow\; \text{$X$ is tangent to $D_c$.}
\end{equation}
\end{lemma}
 
\begin{proof}
Consider the $\omega_P$-orthogonal splitting $TP_x = T(D_c)_x \oplus K_x$, and give $K_x$ its symplectic orientation. For $c = 0$, $K_x$ would just be the tangent space in direction of the fibres of $P$. For the $c$ relevant here, this is no longer the case, but (due to our previous rescaling of $Z$) $K_x$ is still close to the fibrewise tangent subspaces. In particular, $D\Pi_x: K_x \rightarrow \bC$ is an orientation-preserving isomorphism. We know that $J$ preserves $K_x$, and is compatible with its orientation, while $\tilde{\omega}_{P \setminus D}$ is a positive two-form on that space. This implies the desired properties.
\end{proof}

We can transplant the whole situation to the context from Section \ref{subsec:global-geometry}, where we end up with a two-form $\tilde{\omega}_{M \setminus D}$ supported on the union of $D_c \setminus B$, $5/2 \leq |c| \leq 5$. Removing that annular region from $M \setminus D$ disconnects it into two parts, one of which is the ``exterior region'', the union of $D_c \setminus B$ for $|c| < 5/2$. One can arrange, through additional technical assumptions, that the Lagrangian intersection points are divided into interior and exterior ones. Partial energy arguments as in Sections \ref{subsec:partial-energy} and \ref{subsec:q-category} then show that the fibre $\scrF_\infty$ is quasi-equivalent to an $A_\infty$-category $\scrQ$ involving only maps to the exterior region. Further analysis of that category, as in Section \ref{subsec:more-q}, then completes the proof of Proposition \ref{th:restriction-functor-2}.
 
\section{An example\label{sec:example}}
In this section, we consider hyperplane pencils on quadrics. Geometrically, these are the simplest nontrivial Lefschetz pencils; from our perspective, they are not quite straightforward, since the resulting Floer-theoretic structures fall into the framework of graded noncommutative pencils (Construction \ref{th:2}). Those two aspects make this a particularly appealing case for exploring what information (in an algebraic as well as geometric sense) is carried by our constructions. Because of the need to call on various computational techniques, this section is less self-contained than the rest of the paper.
To make the algebraic computations simpler, we work with $\bQ$-coefficients (instead of our usual $\bZ$).

\subsection{The graded Kronecker quiver}
The classical Kronecker quiver ($n = 1$ in our notation) describes coherent sheaves on the projective line. The graded version is discussed in \cite{seidel04c}, with the same symplectic geometry motivation as here. The definition is:
\begin{equation}
\label{eq:kronecker}
\mybox{
For $n \geq 1$, consider the graded quiver
\[
\protect{
\xymatrix{
e \ar@/^1pc/[rrr]^-{a} \ar@/_1pc/[rrr]^-{b} &&& f
}
\qquad |a| = 0, \; |b| = n-1.
}
\]
To this we associate a graded algebra over $\bQ$ spanned (as a vector space) by $e,f,a,b$. Our convention is to think of the product in categorical (right-to-left) order:
\[
e^2 = e, \; f^2 = f, \; ef = fe = 0, \; f b e = b, \; f a e = a.
\]
Equivalently, one can view this as an $A_\infty$-category $\scrA$ with objects $X$ and $Y$, and
\[
\scrA(X,X) = \bQ e, \; \scrA(Y,Y) = \bQ f,\;  \scrA(X,Y) = \bQ a \oplus \bQ b, \; \scrA(Y,X) = 0.
\]
The $A_\infty$-structure is trivial, except for left and right multiplication with the identity morphisms $e$ and $f$.
}
\end{equation}
We want to introduce two noncommutative divisors whose ambient space is $\scrA$ (see \eqref{eq:ambient-space} for the terminology).
\begin{equation} \label{eq:f-plus}
\mybox{
Consider the path algebra of the graded quiver
\[
\protect{
\xymatrix{
e \ar@/^2pc/[rrr]^-{a} \ar@/^1pc/[rrr]^-{b} &&& \ar@/^2pc/[lll]_-{a^*} \ar@/^1pc/[lll]_{b^*} f
}
\qquad |a| = 0, \; |b| = n-1, \; |a^*| = n-1, \;|b^*| = 0.
}
\]
We impose relations
\[
aa^* = bb^*, \;
a^*a = b^* b, \;
b a^* = 0, \; 
a^* b = 0, \;
b^* a = \lambda e, \;
a b^* = \lambda f
\]
for some $\lambda \in \bQ^\times$. The outcome is a Frobenius (Calabi-Yau) algebra of dimension $n-1$, over $\bQ e \oplus \bQ f$. One can again think of it as an $A_\infty$-category $\scrF_\infty$ with objects $(X,Y)$, and
\[
\begin{aligned}
& \scrF_\infty(X,X) = \bQ e \oplus \bQ a^*a, && \scrF_\infty(Y,Y) = \bQ f \oplus \bQ aa^*,\\ & \scrF_\infty(X,Y) = \bQ a \oplus \bQ b, && \scrF_\infty(Y,X) = \bQ a^* \oplus \bQ b^*.
\end{aligned}
\]
As one sees from the inclusion $\scrA \subset \scrF_\infty$, this is in fact the fibre of a noncommutative divisor, with ambient space $\scrA$ and dual bundle $(\Delta_{\scrF_\infty}/\Delta_\scrA)[-1] \iso \Delta_{\scrA}^\vee[-n]$; see \eqref{eq:dual-bundle}, \eqref{eq:dual-bimodule} for terminology (and recall that $[1]$ stands for a downwards shift in the grading).
}
\end{equation}
We have omitted the parameter $\lambda$ from the notation, since it affects the outcome only in a minor way (any two such algebras $\scrF_\infty$ are related by rescaling $a^*,b^*$; equivalently, the noncommutative divisors are related by rescaling the generator of the polynomial ring over which they are defined). Because of the last two relations in \eqref{eq:f-plus}, the objects $X$ and $Y$ are isomorphic in $\scrF_\infty$; hence, a simpler way to describe $\scrF_\infty$ itself would be as the $A_\infty$-category formed by two isomorphic spherical objects \cite{seidel-thomas99} of dimension $(n-1)$. 

For the second construction, we essentially reverse the gradings (even though the actual formulation is a little different, to ensure that it still contains $\scrA$ as a subalgebra).
\begin{equation} \label{eq:f-minus}
\mybox{
In the same vein as in \eqref{eq:f-plus}, consider 
\[
\protect{
\xymatrix{
e \ar@/^2pc/[rrr]^-{a} \ar@/^1pc/[rrr]^-{b} &&& \ar@/^2pc/[lll]_-{a^*} \ar@/^1pc/[lll]_{b^*} f
}}
\qquad |a| = 0, \; |b| = n-1, \; |a^*| = 1-n, \;|b^*| = 2-2n,
\]
with the relations
\[
aa^* = bb^*, \;
a^* a = b^* b, \;
b^* a = 0, \;
a b^* = 0, \;
a^* b = \lambda e, \;
b a^* = \lambda f.
\]
Let's denote the resulting $A_\infty$-category by $\scrF_0$. It has the same properties as before, with the opposite degrees; in particular, $(\Delta_{\scrF_0}/\Delta_\scrA)[-1] \iso \Delta_{\scrA}^\vee[n-2]$.
}
\end{equation}

We will need some computations with twisted complexes over $\scrA$ (results stated without proof).
\begin{equation}
\mybox{
For each $d \geq 0$, introduce a twisted complex
\[
\begin{aligned}
Y_d = & \,\mathit{Cone}\big( X[1-n] \oplus X[2-2n] \oplus \cdots \oplus X[d(1-n)] \\
& \qquad \qquad \longrightarrow Y \oplus Y[1-n] \oplus \cdots Y[d(1-n)] \big),
\end{aligned}
\]
where the mapping cone (between the direct sums, in the given order) is taken along the degree zero map
\[
 \left(\begin{smallmatrix} 
0 &&&& \\
1 & 0 &&& \\
& 1 & 0 \\
&& \cdots \\
&&& 1 & 0 \\
&&&& 1
\end{smallmatrix}\right) \otimes a
+
\left(\begin{smallmatrix}
1 &&&& \\
0 & 1 \\
& 0 & 1 \\
&& \cdots \\
&&& 0 & 1 \\
&&&& 0
\end{smallmatrix}\right) \otimes b.
\]
}
\end{equation}
Note that $Y_0 = Y$. We also set $Y_{-1} = X$. (In terms of the general theory of exceptional collections \cite{rudakov90}, the $Y_d$ are the right half of the helix obtained from the two original objects.) We have
\begin{equation}
\left\{
\begin{aligned}
& H^*(\scrA^{\mathit{tw}}(X,Y_d)) = \bQ \oplus \bQ[1-n] \oplus \cdots \oplus \bQ[(d+1)(1-n)], \\
& H^*(\scrA^{\mathit{tw}}(Y,Y_d)) = \bQ \oplus \bQ[1-n] \oplus \cdots \oplus \bQ[d(1-n)], \\
& H^*(\scrA^{\mathit{tw}}(Y_d,Y_{d-2})) = \bQ[n-2].
\end{aligned}
\right.
\end{equation}
Moreover, the composition maps
\begin{equation} \label{eq:nondegenerate-1}
\left\{
\begin{aligned}
& H^*(\scrA^{\mathit{tw}}(X,Y_{d-2})) \otimes H^{2-n-*}(\scrA^{\mathit{tw}}(Y_d,X)) \longrightarrow H^{2-n}(\scrA^{\mathit{tw}}(Y_{d}, Y_{d-2})) \iso \bQ,
\\
& H^*(\scrA^{\mathit{tw}}(Y,Y_{d-2})) \otimes H^{2-n-*}(\scrA^{\mathit{tw}}(Y_d,Y)) \longrightarrow H^{2-n}(\scrA^{\mathit{tw}}(Y_{d}, Y_{d-2})) \iso \bQ
\end{aligned}
\right.
\end{equation}
are nondegenerate pairings.

At this point, we reduce the grading mod $(2n-2)$, still keeping the same notation $\scrA^{\mathit{tw}}$ as before, for simplicity. Consider the following family of $\bZ/(2n-2)$-graded twisted complexes, parametrized by $\mu \in \bQ^\times$:
\begin{equation} \label{eq:t-mu}
T_{\mu} = \mathit{Cone}\big( X \oplus X[1-n] 
\xrightarrow{\left(\begin{smallmatrix} 1 & \\ & \mu \end{smallmatrix}\right) \otimes a + \left(\begin{smallmatrix} & 1 \\ 1 &  \end{smallmatrix}\right) \otimes b}
Y \oplus Y[1-n] \big).
\end{equation}
Then
\begin{equation} \label{eq:t-mu-hom}
\left\{
\begin{aligned}
& H^*(\scrA^{\mathit{tw}}(X,T_\mu)) = \bQ \oplus \bQ[1-n], \\
& H^*(\scrA^{\mathit{tw}}(T_\mu,X)) = \bQ[-1] \oplus \bQ[-n], \\
& H^*(\scrA^{\mathit{tw}}(Y,T_\mu)) = \bQ \oplus \bQ[1-n], \\
& H^*(\scrA^{\mathit{tw}}(T_\mu,Y)) = \bQ[-1] \oplus \bQ[-n], \\
& H^*(\scrA^{\mathit{tw}}(T_\mu, T_\mu)) \iso \bQ \oplus \bQ[-1] \oplus \bQ[1-n] \oplus \bQ[n-2],
\end{aligned}
\right.
\end{equation}
Moreover, for $n>2$, the products
\begin{equation} \label{eq:nondegenerate-2}
\left\{
\begin{aligned}
& H^*(\scrA^{\mathit{tw}}(X,T_\mu)) \otimes H^{2-n-*}(\scrA^{\mathit{tw}}(T_\mu,X))
\longrightarrow H^{2-n}(\scrA^{\mathit{tw}}(T_\mu,T_\mu)) \iso \bQ, \\
& H^*(\scrA^{\mathit{tw}}(Y,T_\mu)) \otimes H^{2-n-*}(\scrA^{\mathit{tw}}(T_\mu,Y))
\longrightarrow H^{2-n}(\scrA^{\mathit{tw}}(T_\mu,T_\mu)) \iso \bQ \\
\end{aligned}
\right.
\end{equation}
are nondegenerate pairings. Finally:
\begin{equation} \label{eq:t-mu-dies}
\mybox{
The image of $T_\mu$ under the map $\scrA^{\mathit{tw}} \rightarrow \scrF_\infty^{\mathit{tw}}$ is quasi-isomorphic to zero. Indeed, the morphism that forms the cone in \eqref{eq:t-mu} has an inverse in $\scrF_\infty^{\mathit{tw}}$. The same holds for $\scrF_0$.
}
\end{equation}

\subsection{Bimodules}
Before continuing, we need to review some additional algebraic language. Consider the categories of left and right $A_\infty$-modules over $\scrA$, denoted by $\scrA^{\mathit{left}}$ and $\scrA^{\mathit{right}}$, together with their Yoneda embeddings
\begin{equation}
\left\{
\begin{aligned}
& \scrI^{\mathit{left}}: \scrA^{\mathit{tw}} \longrightarrow (\scrA^{\mathit{left}})^{\mathit{opp}}, \\
& \scrI^{\mathit{right}}: \scrA^{\mathit{tw}} \longrightarrow \scrA^{\mathit{right}};
\end{aligned}
\right.
\end{equation}
and the functors of taking duals,
\begin{equation}
\left\{
\begin{aligned}
& (\cdot)^\vee: (\scrA^{\mathit{left}})^{\mathit{opp}} \longrightarrow \scrA^{\mathit{right}}, \\
& (\cdot)^\vee: \scrA^{\mathit{right}} \longrightarrow (\scrA^{\mathit{left}})^{\mathit{opp}}.
\end{aligned}
\right.
\end{equation}
Let $\scrA^{\mathit{bi}}$ be the category of $A_\infty$-bimodules. Given two right modules $\scrM_0$ and $\scrM_1$, one can form $\mathit{Hom}(\scrM_0,\scrM_1)$ (the graded space of $\bQ$-linear maps), which is naturally a bimodule. If $\scrM_1$ is proper (meaning that it is has finite-dimensional cohomology when evaluated at any object), there is a quasi-isomorphism
\begin{equation} \label{eq:dual-module-1}
\scrM_0^\vee \otimes \scrM_1 \stackrel{\htp}{\longrightarrow} \mathit{Hom}(\scrM_0,\scrM_1),
\end{equation}
where the tensor product on the left is also over $\bQ$. One recovers the actual module homomorphisms from $\mathit{Hom}(\scrM_0,\scrM_1)$ by using the diagonal bimodule,
\begin{equation} \label{eq:dual-module-2}
\scrA^{\mathit{bi}}(\Delta_{\scrA}, \mathit{Hom}(\scrM_0,\scrM_1)) \stackrel{\htp}{\longrightarrow} \scrA^{\mathit{right}}(\scrM_0,\scrM_1).
\end{equation}
In our application, $\scrA$ is proper and homologically smooth (see e.g.\ \cite{kontsevich-soibelman06} for the terminology). Hence, the image of the Yoneda embedding consists of all the proper modules up to quasi-isomorphism. As a consequence, $\scrA^{\mathit{tw}}$ has an essentially canonical autoequivalence, the Serre functor $\scrS$. One point of view is as follows (see e.g.\ \cite{shklyarov07b}): the bimodule $\Delta_{\scrA}^\vee$ is invertible with respect to tensor product. It induces an autoequivalence of $\scrA^{\mathit{right}}$, which takes $\scrI^{\mathit{right}}(Z)$ to $\scrI^{\mathit{left}}(Z)^\vee$ up to quasi-isomorphism for any object $Z$, hence preserves properness; one obtains $\scrS$ by restricting to the image of the Yoneda embedding. More formally, what we are saying is that Serre functor fits into the homotopy commutative diagram
\begin{equation} \label{eq:serre}
\xymatrix{
\ar[dd]_-{\scrI^{\mathit{left}}} \ar[dr]^-{\scrI^{\mathit{right}}}
\scrA^{\mathit{tw}} \ar[rr]^-{\scrS}_-{\htp} && \scrA^{\mathit{tw}} \ar[dd]^-{\scrI^{\mathit{right}}} \\ 
& \scrA^{\mathit{right}} 
\ar[dr]^-{\cdot\, \otimes_{\scrA} \Delta_\scrA^{\vee}}_{\htp} 
& \\
(\scrA^{\mathit{left}})^{\mathit{opp}} \ar[rr]^-{(\cdot)^\vee} && \scrA^{\mathit{right}}.
}
\end{equation}

\begin{lemma} \label{th:hh}
For any two objects $Z_0,Z_1$ of $\scrA$, 
\begin{equation}
H^*\big(\scrA^{\mathit{bi}}(\Delta_{\scrA}, \scrI^{\mathit{left}}(Z_0) \otimes \scrI^{\mathit{right}}(Z_1))\big) \iso
H^*(\scrA(\scrS Z_0,Z_1)).
\end{equation}
\end{lemma}

\begin{proof}
Since the Yoneda modules are proper, double dualization is quasi-isomorphic to the identity. From that and \eqref{eq:serre}, one gets $\scrI^{\mathit{left}}(Z_0) \htp \scrI^{\mathit{right}}(\scrS Z_0)^\vee$. In view of \eqref{eq:dual-module-1}, it follows that $\scrI^{\mathit{left}}(Z_0) \otimes \scrI^{\mathit{right}}(Z_1) \htp \mathit{Hom}(\scrI^{\mathit{right}}(\scrS Z_0), \scrI^{\mathit{right}}(Z_1))$. Now apply \eqref{eq:dual-module-2} and use the fact that the Yoneda embedding is full and faithful on cohomology.
\end{proof}

Let's return to the specifics of our algebra \eqref{eq:kronecker}. The previous computations \eqref{eq:nondegenerate-1}, \eqref{eq:nondegenerate-2} say that
\begin{align}
\label{eq:serre-1}
& \scrS Y_d \htp Y_{d-2}[2-n], \\
\label{eq:serre-2}
& \scrS T_\mu \htp T_\mu[2-n].
\end{align}
We also need another piece of information (which is a special case of Beilinson-style resolutions of the diagonal associated to exceptional collections): there is a quasi-isomorphism in $\scrA^{\mathit{bi}}$,
\begin{equation} \label{eq:beilinson}
\Delta_{\scrA} \htp 
\mathit{Cone}\big(\scrI^{\mathit{left}}(Y_1) \otimes \scrI^{\mathit{right}}(X)[1-n]
\longrightarrow \scrI^{\mathit{left}}(Y) \otimes \scrI^{\mathit{right}}(Y)\big).
\end{equation}

\begin{lemma} \label{th:beilinson}
Denote by $(\Delta_{\scrA}^\vee)^{\otimes_{\scrA} r}$ the $r$-fold tensor product of the bimodule $\Delta_{\scrA}^\vee$.
For any $r \geq 0$, there is a long exact sequence
\begin{equation}
\begin{aligned}
\cdots \rightarrow H^{*+(r+1)(n-2)}(\scrA(Y,Y_{2r+2})) \longrightarrow & \,
H^*(\scrA^{\mathit{bi}}( (\Delta_{\scrA}^\vee)^{\otimes_{\scrA} r},\Delta_{\scrA})) \\ & \qquad \quad
\longrightarrow H^{*+r(n-2)}(\scrA(X,Y_{2r-1})) \rightarrow \cdots
\end{aligned}
\end{equation}
\end{lemma}

\begin{proof}
For $r \geq 0$, denote by $(\Delta_{\scrA}^{\vee})^{\otimes_{\scrA} -r}$ the tensor product of $r$ copies of the bimodule inverse to $\Delta_{\scrA}^\vee$. Then
\begin{equation}
H^*\big(\scrA^{\mathit{bi}}((\Delta_\scrA^\vee)^{\otimes_{\scrA} r}, \Delta_{\scrA})\big) \iso
H^*\big(\scrA^{\mathit{bi}}(\Delta_{\scrA}, (\Delta_{\scrA}^\vee)^{\otimes_{\scrA} -r})\big)
\end{equation}
From \eqref{eq:beilinson}, \eqref{eq:serre}, and \eqref{eq:serre-1}, we get
\begin{equation} \label{eq:iterated-tensor}
\begin{aligned}
& (\Delta_{\scrA}^\vee)^{\otimes_{\scrA} -r} \htp 
\Delta_{\scrA} \otimes_{\scrA} (\Delta_{\scrA}^\vee)^{\otimes_{\scrA} -r} \\
& \htp 
\mathit{Cone}\big(\scrI^{\mathit{left}}(Y_1) \otimes \scrI^{\mathit{right}}(\scrS^{-r} X)[1-n]
\longrightarrow \scrI^{\mathit{left}}(Y) \otimes \scrI^{\mathit{right}}(\scrS^{-r} Y) \big) \\
& \htp
\mathit{Cone}\big(\scrI^{\mathit{left}}(Y_1) \otimes \scrI^{\mathit{right}}(Y_{2r-1})[(r-1)(n-2)-1]
\longrightarrow \scrI^{\mathit{left}}(Y) \otimes \scrI^{\mathit{right}}(Y_{2r})[r(n-2)] \big).
\end{aligned}
\end{equation}
By Lemma \ref{th:hh} and another application of \eqref{eq:serre-1},
\begin{equation} \label{eq:iterated-tensor-2}
\left\{
\begin{aligned}
& H^*(\scrA^{\mathit{bi}}\big(\Delta_{\scrA},  \scrI^{\mathit{left}}(Y_1) \otimes \scrI^{\mathit{right}}(Y_{2r-1})[(r-1)(n-2)])\big) \\ & \qquad \qquad
\iso H^{*+(r-1)(n-2)}(\scrA(\scrS Y_1, Y_{2r-1})) \iso H^{*+r(n-2)}(\scrA(X,Y_{2r-1})), \\
& H^*(\scrA^{\mathit{bi}}(\Delta_{\scrA}, \scrI^{\mathit{left}}(Y) \otimes \scrI^{\mathit{right}}(Y_{2r})[r(n-2)])\big) 
\\ & \qquad \qquad
\iso
H^{*+r(n-2)}(\scrA(\scrS Y, Y_{2r})) \iso H^{*+(r+1)(n-2)}(\scrA(Y,Y_{2r+2})).
\end{aligned}
\right.
\end{equation}
The morphisms from the diagonal bimodule into \eqref{eq:iterated-tensor} obviously fit into a long exact sequence involving the parts \eqref{eq:iterated-tensor-2}.
\end{proof}


\begin{lemma} \label{th:classification}
Any noncommutative divisor with ambient space $\scrA$ and dual bundle $\Delta_{\scrA}^\vee[n-2]$, for $n>2$, is either trivial or quasi-isomorphic to that underlying \eqref{eq:f-minus} (for some value of the parameter $\lambda \in \bQ^\times$).
\end{lemma}

\begin{proof}
Thanks to Lemma \ref{th:beilinson}, the shifted groups $H^*(\scrA^{\mathit{bi}}( (\Delta_{\scrA}^\vee[n-2])^{\otimes_{\scrA} r}, \Delta_\scrA))$, $r \geq 1$, vanish in negative degrees; and in degree zero, only the $r = 0$ group is nontrivial, and that is one-dimensional. One now applies \cite[Lemma 2.12]{seidel14b}.
\end{proof}

Consider a noncommutative divisor as in Lemma \ref{th:classification}. In view of \eqref{eq:serre-2}, its section (see \eqref{eq:dual-section} for the terminology) gives rise to an element of
\begin{equation} \label{eq:detecting-element}
\begin{aligned}
& H^0\big(\scrA^{\mathit{right}}(\scrI^{\mathit{right}}(T_\mu) \otimes_{\scrA} \Delta_{\scrA}^\vee[n-2], \,
\scrI^{\mathit{right}}(T_\mu) \otimes_{\scrA} \Delta_{\scrA})\big) \\
& \qquad \iso H^{2-n}(\scrA^{\mathit{tw}}(\scrS T_\mu, T_\mu)) \iso H^0(\scrA^{\mathit{tw}}(T_\mu, T_\mu)) \iso \bQ.
\end{aligned}
\end{equation}
Here, the first isomorphism is a consequence of the fact that the Yoneda embedding is cohomologically full and faithful, together with the diagram \eqref{eq:serre} defining the Serre functor; the second isomorphism comes from \eqref{eq:serre-2}.

\begin{lemma} \label{th:detecting-element}
In the situation of Lemma \ref{th:classification}, the noncommutative divisor is trivial iff \eqref{eq:detecting-element} vanishes.
\end{lemma}

\begin{proof}
By definition, \eqref{eq:detecting-element} is zero for the trivial noncommutative divisor. Because of Lemma \ref{th:classification}, we only need to show that it is nontrivial for the divisor underlying $\scrF_0$. In view of \eqref{eq:fibres}, the module map in \eqref{eq:detecting-element} is the boundary homomorphism of the short exact sequence of modules
\begin{equation} \label{eq:t-sequence}
0 \rightarrow \scrI^{\mathit{right}}(T_\mu) \otimes_{\scrA} \Delta_\scrA \longrightarrow
\scrI^{\mathit{right}}(T_\mu) \otimes_{\scrA} \Delta_{\scrF_0 }\longrightarrow \scrI^{\mathit{right}}(T_\mu) \otimes_{\scrA}
\Delta_\scrA^\vee[n-1] \rightarrow 0.
\end{equation}
The object in the middle of \eqref{eq:t-sequence} is an $\scrA$-module obtained by restricting an $\scrF_0$-module. That module can be viewed as the image of $T_\mu$ under $\scrA^{\mathit{tw}} \rightarrow \scrF_0^{\mathit{tw}}$ followed by the Yoneda embedding for $\scrF_0$. By \eqref{eq:t-mu-dies}, that object is zero. Hence, \eqref{eq:t-sequence} can't possibly be homotopically split, which means that the boundary homomorphism is nontrivial.
\end{proof}

\subsection{The Lefschetz fibration on an affine quadric\label{subsec:affine-quadric}}
Consider a smooth affine quadric of complex dimension $n > 2$ (in principle, one could discuss $n = 2$ as well, but that case is more complicated in several technical respects), with its linear projection to $\bC$, say
\begin{equation} \label{eq:linear-pi}
\Pi = x_1: M = \{x_1^2 + \cdots + x_{n+1}^2  + 1 = 0\} \longrightarrow \bC.
\end{equation}
With the standard symplectic form, this becomes a symplectic Lefschetz fibration. It does not quite fit into the context from Section \ref{subsec:target} because the fibres are noncompact, and because the fibration is not symplectically locally trivial outside a compact subset. Both deficiencies could be corrected without significantly altering the geometry; however, for the sake of brevity, we prefer to discuss \eqref{eq:linear-pi} in its given form, tacitly assuming that the construction from Section \ref{sec:first-construction} has been adapted accordingly. The total space is symplectically isomorphic to $T^*S^n$, and the smooth fibres to $T^*S^{n-1}$. Let $(X,Y)$ be the Lefschetz thimbles associated to paths emanating from the critical values $\pm i \in \bC$, and going to infinity as shown in Figure \ref{fig:xyt}. We have a clean intersection in a single fibre, $X \cap Y \iso S^{n-1} \subset T^*S^{n-1}$. 
\begin{figure}
\begin{centering}
\begin{picture}(0,0)%
\includegraphics{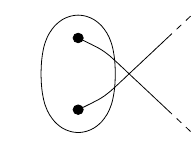}%
\end{picture}%
\setlength{\unitlength}{3355sp}%
\begingroup\makeatletter\ifx\SetFigFont\undefined%
\gdef\SetFigFont#1#2#3#4#5{%
  \reset@font\fontsize{#1}{#2pt}%
  \fontfamily{#3}\fontseries{#4}\fontshape{#5}%
  \selectfont}%
\fi\endgroup%
\begin{picture}(1827,1403)(136,9)
\put(1451, 64){\makebox(0,0)[lb]{\smash{{\SetFigFont{10}{12.0}{\rmdefault}{\mddefault}{\updefault}{\color[rgb]{0,0,0}$\Pi(Y)$}%
}}}}
\put(1451,1289){\makebox(0,0)[lb]{\smash{{\SetFigFont{10}{12.0}{\rmdefault}{\mddefault}{\updefault}{\color[rgb]{0,0,0}$\Pi(X)$}%
}}}}
\put(-301,614){\makebox(0,0)[lb]{\smash{{\SetFigFont{10}{12.0}{\rmdefault}{\mddefault}{\updefault}{\color[rgb]{0,0,0}$\Pi(T) = l$}%
}}}}
\end{picture}%
\caption{\label{fig:xyt}The Lagrangian submanifolds $X,Y,T$ in the quadric, projected to the plane as in \eqref{eq:linear-pi}. The area inside the loop $l$ is fixed by our choice of monotonicity constant.}
\end{centering}
\end{figure}

\begin{prop} \label{th:f-plus}
Let's associate to \eqref{eq:linear-pi} an $A_\infty$-category, and a noncommutative divisor having that category as ambient space; as in Section \ref{sec:first-construction}, but using only $(X,Y)$ as objects. Then, the $A_\infty$-category is (quasi-isomorphic to) \eqref{eq:kronecker}, and the noncommutative divisor has fibre \eqref{eq:f-plus}.
\end{prop}

\begin{proof}
The first part is immediate from the definition; the second part follows from the relation with the Fukaya category of the fibre, Proposition \ref{th:restriction-functor}.
\end{proof}

The formulation of Proposition \ref{th:f-plus} is a little imprecise: given the approach we have taken, the construction also requires looking at (infinitely many) positive perturbations of $X$ and $Y$. In the resulting category, all these perturbations are quasi-isomorphic to the original object, by definition. Hence, we can forget about them a posteriori, which is how one gets \eqref{eq:kronecker}. Taking a step back, there is no intrinsic need to restrict to $(X,Y)$ as objects. Indeed, geometrically it makes sense to include other Lagrangian submanifolds as well, even though algebraically that does not lead to more information. We spell out this observation in two versions:
\begin{equation} \label{eq:generate-1}
\mybox{
We can add closed Lagrangian submanifolds $L \subset M$ satisfiying \eqref{eq:lagrangian} as objects to our category. That does not affect $\scrA^{\mathit{tw}}$, since any such $L$ can be expressed as a twisted complex in terms of $(X,Y)$ (this is shown in \cite[Proposition 18.17]{seidel04}, in a context that is technically different but equivalent to the one here; see also \cite{biran-cornea17}).
}
\end{equation}
\begin{equation} \label{eq:generate-2}
\mybox{
More generally, fix constants $d \geq 2$ (integral) and $\gamma > 0$ (a real number). Going a little beyond the framework in Section \ref{sec:first-construction}, we can define a larger version of our category which also includes closed Lagrangian submanifolds which are {\em Spin}, have a grading mod $2d$ (see \cite{seidel99} for this notion), and are monotone with monotonicity constant $\gamma$ (Floer cohomology for such Lagrangian submanifolds is well-defined over $\bQ$ by \cite{oh93}; note that the minimal Maslov number is $2d \geq 4$). By an argument analogous to the previous one, such objects can be expressed as $(\bZ/2d)$-graded twisted complexes in $(X,Y)$. As a minor further extension, one can allow the 
Lagrangian submanifolds to come with flat $\bQ$-vector bundles.
}
\end{equation}
For our specific application, we will only need particular instance of the last-mentioned idea:
\begin{equation} \label{eq:generate-3}
\mybox{
Consider the Lagrangian submanifold 
\[
T \iso S^1 \times S^{n-1} \subset M
\]
 fibered over the loop $l$ from Figure \ref{fig:xyt}, with an $S^{n-1}$ in each fibre, as in \cite[Section 6]{khovanov-seidel98}. This is monotone, with the monotonicity constant prescribing the area inside the loop, and it admits a mod $(2n-2)$ grading. Equip it with a flat line bundle with holonomy $\mu \in \bQ^\times$ (going anticlockwise around the loop). This gives a family of objects $T_\mu$ in the category from \eqref{eq:generate-2}, all of which can be written as $\bZ/(2n-2)$-graded twisted complexes in $(X,Y)$. These objects are nonzero, since $\mathit{HF}^*(T_\mu,T_\mu) \neq 0$ (one can see that e.g.\  from the spectral sequence in \cite{oh96}).
}
\end{equation}

\begin{remark}
There is an explicit recipe (compare \cite[Sections 5k and 5l]{seidel04}) for writing a Lagrangian submanifold as in \eqref{eq:generate-1} or \eqref{eq:generate-2} as a twisted complex:
\begin{equation} \label{eq:explicit-cone}
L \htp \mathit{Cone}(\mathit{HF}^*(X,L) \otimes X[-n] \longrightarrow \mathit{HF}^*(Y,L) \otimes Y), 
\end{equation}
where the morphism is obtained as follows. Take the Floer cohomology $\mathit{HF}^*(Y,X)$, which by Floer-theoretic Poincar{\'e} duality is canonically isomorphic to $\mathit{HF}^{n-*}(X,Y)^\vee$. Using that duality, we think of the product $\mathit{HF}^*(X,L) \otimes \mathit{HF}^*(Y,X) \rightarrow \mathit{HF}^*(Y,L)$ as an element of degree $n$ in
\begin{equation}
\begin{aligned}
& \mathit{Hom}^*\big(\mathit{HF}^*(X,L), \mathit{HF}^*(Y,L)\big) \otimes \mathit{HF}^*(X,Y) \\ & \qquad \iso
\mathit{Hom}^*\big(\mathit{HF}^*(X,L), \mathit{HF}^*(Y,L)\big) \otimes (\bQ a \oplus \bQ b),
\end{aligned}
\end{equation}
which enters into \eqref{eq:explicit-cone}. For instance, the zero-section $S^n \subset M$ is quasi-isomorphic to $\mathit{Cone}(a: X \rightarrow Y)$. A bit less obviously, the objects $T_\mu$ from \eqref{eq:generate-3} correspond to the twisted complexes \eqref{eq:t-mu} (up to a sign ambiguity in the $\mu$ parametrization, due to the fact that we have not specified the {\em Spin} structure). We will not really use this explicit expression; but it is instructive to keep it in mind, since it links the algebraic argument from Lemma \ref{th:detecting-element} to a geometric analogue, which will appear below. 
\end{remark}

\subsection{The Lefschetz pencil on a projective quadric\label{subsec:projective-quadric}}
Take a smooth quadric $M \subset \bC P^{n+1}$, and on that a generic pencil of hyperplane sections. Fix a smooth member of that pencil, denoted by $D$, and consider the resulting Lefschetz fibration $\Pi: M \setminus D \rightarrow \bC$. Concretely, we can take $M = \{x_0^2 + \cdots + x_{n+1}^2 = 0\}$, with the pencil spanned by $\{x_0,x_1\}$, in which case $D = \{x_0 = 0\}$ and the Lefschetz fibration is exactly \eqref{eq:linear-pi}. We will again be interested in the two Lefschetz thimbles $(X,Y)$  and the closed Lagrangian submanifold $T$ from \eqref{eq:generate-3}. The latter has the following alternative description from the Lefschetz pencil viewpoint:
\begin{equation} \label{eq:t-again}
\mybox{
Remove the base locus of the pencil from $D$. The outcome is an affine quadric of dimension $(n-1)$; from the preferred Lagrangian sphere in it, we can construct a monotone Lagrangian submanifold in $M \setminus D$, by the method from \cite[Section 5]{biran-cieliebak01b}. This is not unique, but rather depends on a choice of (small) radius. One can show that the resulting submanifolds are (up to Hamiltonian isotopy) those from \eqref{eq:generate-3}, for a certain value of the monotonicity constant (smaller radii correspond to larger monotonicity constants).
}
\end{equation}
Let's strengthen our previous dimension assumption to $n>3$ (again for technical simplicity). We would like to introduce a version of \eqref{eq:wl} for $T_\mu$. There is a $\bZ$ family of homotopy classes of discs with boundary on $T$ and intersection number $1$ with $D$; these are distinguished by the degree of their boundary projected to $l$. Denoting that number by $q$, we find that the Maslov number of such discs is $q(2n-2) + 2n$. Taking into account the intersection constraint \eqref{eq:holo-disc}, the expected dimension of the moduli space of discs is
\begin{equation}
\mathrm{dim}\, \hat\scrM_q = n +(q+1)(2n-2).
\end{equation}
We define $W(T_\mu) \in \bQ$ by counting discs in $\hat\scrM_{-1}$, having prescribed value at one boundary point (with signs, and with weights given by the boundary holonomies of the flat bundle, which in our context means $\mu^{-1}$). For generic choices, disc bubbling does not occur in the compactification of this space: a disc bubble would have to lie in $M \setminus D$, hence have positive boundary degree (with respect to projection to $l$, as before) because of monotonicity; while the remaining component would lie in $\hat\scrM_q$, $q<-1$, which is empty for dimension reasons. The same holds for one-parameter families of the choices involved in setting up \eqref{eq:holo-disc}, which shows that $W(T_\mu)$ is well-defined. 

\begin{lemma} \label{th:trivial-discs}
The disc-counting invariant satisfies $W(T_\mu) = \pm \mu^{-1}$ (the sign ambiguity comes from not having specified the {\em Spin} structure).
\end{lemma}

\begin{proof}
Let $D_* \subset M$ be a smooth hypersurface in our pencil, which is the closure of the fibre of $\Pi$ over a point lying inside the loop defining $T$. For simplicity, let's adopt the concrete model \eqref{eq:linear-pi} and choose $D_* = \{x_1 = 0\}$. The discs that define $W(T_\mu)$ have intersection number $0$ with $D_*$, hence lie in $M \setminus D_*$. Explicitly,
\begin{equation} \label{eq:1-over-pi}
M \setminus D_* = \{x_0^2 + x_2^2 + \cdots x_{n+1}^2 = 1\} \xrightarrow{1/\Pi = x_0} \bC.
\end{equation}
We are looking for discs in $M \setminus D_*$ whose projection under \eqref{eq:1-over-pi} is the disc $\Delta \subset \bC$ bounded by $1/l$, and which have boundary degree $1$ over $1/l$. One can explicitly trivialize that part of the fibration:
\begin{equation} \label{eq:delta-trivialization}
\begin{aligned}
& \{ x \in M \setminus D_* \;:\; 1/x_0 \in \Delta \} \stackrel{\iso}{\longrightarrow} \Delta \times \{x_2^2+\cdots + x_{n+1}^2 = 1\}, \\
& x \longmapsto \big(1/x_0, (1-x_0^2)^{-1/2} (x_2,\dots,x_{n+1})\big), 
\end{aligned}
\end{equation}
for some choice of root $(1-x_0^2)^{1/2}$ (this is possible because both critical values $x_0 = \pm 1$ lie outside $\Delta$). Under \eqref{eq:delta-trivialization}, the Lagrangian submanifold $T_\mu$ goes to $\partial \Delta \times S^{n-1}$; and hence, the discs under consideration correspond to ones that parametrize $\Delta$, together with a constant map in the $S^{n-1}$ factor. In particular, there is one such disc through every point of $T_\mu$.
\end{proof}

Let's apply Construction \ref{th:2}, again using only $(X,Y)$ as objects. This yields a graded noncommutative pencil. Proposition \ref{th:restriction-functor-2} implies that the fibre at $\infty$ of the noncommutative pencil, in the terminology of \eqref{eq:terminology-fibres}, is \eqref{eq:f-plus}.

\begin{prop}
The fibre at $0$ of the noncommutative pencil is \eqref{eq:f-minus}.
\end{prop}

\begin{proof}
Because they are part of a graded noncommutative pencil, the fibres at $0$ and $\infty$ are obtained from the same $A_\infty$-bimodule up to a shift. As a consequence of Proposition \ref{th:f-plus}, the bimodule responsible for the fibre at $\infty$ is $\Delta_{\scrA}^\vee[-n]$; and that responsible for the fibre at $0$ is therefore $\Delta_{\scrA}^\vee[n-2]$ (strictly speaking, these statements are true only up to quasi-isomorphism, but see Remark \ref{th:switch-q}). By Lemma \ref{th:classification}, we only have to show that the section which underlies the fibre at $0$ is nontrivial.

At this point, we enlarge $\scrA$ geometrically as in \eqref{eq:generate-3}, by allowing the $T_\mu$ as objects. By a mild generalization of Proposition \ref{th:disc-counting}, part of the section which gives rise to the fibre at $0$ is a map $\mathit{HF}^*(T_\mu,T_\mu;-1) \rightarrow \mathit{HF}^*(T_\mu,T_\mu;0)$ which is $W(T_\mu)$ times the continuation isomorphism. Since $W(T_\mu)$ is nonzero by Lemma \ref{th:trivial-discs}, and $T_\mu$ is a nonzero object, that section is nontrivial. On the other hand, since $T_\mu$ is a twisted complex in $(X,Y)$, including it does not change the category of $\scrA$-bimodules. Hence, the section was nontrivial already in the original category, having only $(X,Y)$ as objects.
\end{proof}

So far, we have focused on two parts of the noncommutative pencil structure, which correspond to the constructions from Section \ref{sec:first-construction} and \ref{sec:anticanonical}, respectively. We will not determine the entire pencil (this would require additional algebraic arguments, along the lines of Lemma \ref{th:classification}), and instead conclude by looking at the other fibres. The fibre at a point $w \neq 0,\infty$, again considering only $(X,Y)$ as objects, is 
a $\bZ/(2n-2)$-graded $A_\infty$-category. It turns out that these fibres are independent of $w$ (which is along similar lines as the general phenomenon observed in \eqref{eq:graded-fibres}, but stronger):

\begin{prop} \label{th:generic-fibre}
In the fibre at $w \neq 0,\infty$, the objects $(X,Y)$ are quasi-isomorphic, and each of them has (semisimple) endomorphism ring isomorphic to $\bQ[t]/(t^2-1)$, for $|t| = n-1$.
\end{prop}

\begin{proof}
We know that the section underlying the fibre at $w$ is a nontrivial linear combination of those for the fibres at $0$ and $\infty$. Hence, with the same notation as in \eqref{eq:f-plus} and \eqref{eq:f-minus}, the product on that fibre satisfies
\begin{equation} \label{eq:both-products}
\left\{
\begin{aligned}
& b^* a = \text{\it (nontrivial multiple of $e$),} \\
& a^* b = \text{\it (nontrivial multiple of $e$),} \\
& a b^* = \text{\it (nontrivial multiple of $f$),} \\
& b a^* = \text{\it (nontrivial multiple of $f$).}
\end{aligned}
\right.
\end{equation}
We know that the objects $X$ and $Y$ have two-dimensional endomorphism algebras, with generators in degrees $0$ and $n-1$. By \eqref{eq:both-products}, $X$ is quasi-isomorphic to both $Y$ and $Y[n-1]$, hence has an automorphism of degree $n-1$. The desired result follows from that. Note that the graded algebra $\bQ[t]/(t^2-1)$ is intrinsically formal (carries no nontrivial $A_\infty$-structure).
\end{proof}



\end{document}